\documentclass[12pt]{article}

\usepackage[english]{babel}
\usepackage{amsmath}
\usepackage{amsthm}
\usepackage{amsfonts}
\usepackage{tikz}
\usepackage{amssymb}
\usepackage{extpfeil}
\usepackage[all]{xy}
\usepackage{comment}
\usepackage{caption}

\numberwithin{equation}{section}

\newtheorem{proposition}{Proposition}[section]
\newtheorem{theorem}[proposition]{Theorem}
\newtheorem{lemma}[proposition]{Lemma}
\newtheorem{definition}[proposition]{Definition}
\newtheorem{example}[proposition]{Example}
\newtheorem{remark}[proposition]{Remark}
\newtheorem{corollary}[proposition]{Corollary}
\newtheorem{assumption}{Assumption}

\def\R{\mathbb{R}}
\def\g{\mathfrak{g}}
\def\w{\mathbf{w}}

\usepackage[letterpaper,top=2cm,bottom=2cm,left=3cm,right=3cm,marginparwidth=1.75cm]{geometry}

\usepackage{graphicx}
\usepackage[colorlinks=true, allcolors=blue]{hyperref}

\title{Reduction of symmetric time-dependent Hamiltonian systems I: presymplectic principal $\R$-bundles}
\author{\small
C. Ben\'{\i}tez,
D. Iglesias Ponte,
J. C. Marrero,
E. Padr\'on\\[0.5em]
\footnotesize Departamento de Matemáticas, Estadística e Investigación Operativa \\ 
\footnotesize Instituto Universitario de Matemáticas y Aplicaciones (IMAULL)\\
\footnotesize ULL-CSIC Geometr\'{\i}a Diferencial y Mec\'anica Geom\'etrica  \\ 
\footnotesize University of La Laguna, La Laguna, Spain\\[0.5em]
\scriptsize e-mail: \texttt{cbeniteg@ull.edu.es}, \texttt{diglesia@ull.edu.es}, \texttt{jcmarrer@ull.edu.es}, \texttt{mepadron@ull.edu.es}
}

\date{}

\begin{document}

\maketitle

\begin{abstract}
The reduction of mechanical presymplectic structures and its application to the reduction of time-dependent Hamiltonian systems were developed in a recent paper \cite{GUTIERREZSAGREDO}. This approach overcomes some limitations of the original Albert reduction \cite{ALBERT1989627}. In this paper we will describe the reduction of time-dependent Hamiltonian systems using the extended formalism. This process combines cotangent bundle reduction with reduction of presymplectic structures of corank 1 and 2. These results generalize previous work \cite{IGNACIO} in the same direction. Several examples will be presented to illustrate the theory from this new perspective.

\bigskip
{\footnotesize 

}
\end{abstract}


\section{Introduction}

\subsection{Marsden-Weinstein reduction of symmetric Hamiltonian systems on symplectic manifolds}

Given the configuration space $Q$ of a classical mechanical system, its cotangent bundle $T^*Q$, equipped with the canonical symplectic structure $\omega_Q$, is the basic ingredient to describe Hamiltonian Mechanics. More specifically, for a Hamiltonian function $H\colon T^*Q\to \mathbb{R}$, the solutions of Hamilton's equations are just the integral curves of the corresponding Hamiltonian vector field (see, for instance, \cite{AM, MR1021489, LM}).

A method that simplifies the integration of the dynamical equations of a Hamiltonian system with a symmetry Lie group and that sometimes provides new examples of symplectic manifolds, is the reduction process. A particular
case is the so-called Marsden-Weinstein reduction \cite{MARSDEN} (see also \cite{AM}). More precisely, a free and proper action $\phi$ of a Lie group $G$ on a symplectic manifold $(M,\Omega )$ is said to be Hamiltonian, if it is an action by symplectomorphisms and, moreover, there is a map $J\colon M\to \mathfrak{g}^*$,  $\mathfrak{g}^*$ being
the dual of the Lie algebra $\mathfrak{g}$ of $G$, 
such that 
\[
i_{\xi_M}\Omega = dJ_\xi,\quad  \mbox{for all} \quad \xi\in \mathfrak{g},
\]
where $\xi_M\in \mathfrak{X}(M)$ is the fundamental vector field of $\phi$ corresponding to the Lie algebra element $\xi$ and $J_\xi\colon M\to \mathbb{R}$ is the smooth function given by
\[
J_\xi (x) = J(x)( \xi), \quad \mbox{for} \quad x \in M.
\] 
$J$ is the momentum map for the symplectic action $\phi$. In addition, if $J$ is equivariant with respect to $\phi$ and the coadjoint representation $\mbox{Ad}^\ast \colon G\times \mathfrak{g}^* \to \mathfrak{g}^*$, given a regular value $\nu$ of $J$, then the isotropy subgroup $G_\nu$ for the coadjoint representation acts freely and properly
on $J^{-1}(\nu)$, the quotient $M_\nu := J^{-1}(\nu ) / G_\nu$ is a smooth manifold and  there is a naturally induced ``reduced'' symplectic structure $\Omega_\nu$ on $M_\nu$. If, in addition, there is a $G$-invariant
Hamiltonian function $H \colon M \to \mathbb{R}$ then, for each $\xi\in \mathfrak{g}$, the smooth function $J_\xi$ is a first integral of the 
Hamiltonian vector field $X_H$ and the reduction process also allows to describe the reduced dynamics as the symplectic dynamics in $(M_\nu, \Omega_\nu)$ associated with a reduced Hamiltonian function $H_\nu \colon M_\nu \to \mathbb{R}$.

In the particular case when $M$ is the cotangent bundle $T^*Q$ of the configuration space $Q$ endowed with the canonical symplectic structure, then the typical Hamiltonian action of a Lie group $G$ on $T^*Q$ is the cotangent lift of a free and proper action of $G$ on $Q$ and the momentum map $J_c: T^*Q \to \g^*$ is just the canonical momentum map given by
$$J_c(\alpha_q)(\xi) = \alpha_q(\xi_Q(q)), \quad \text{for} \quad \alpha_q\in T_q^*Q \quad \text{and} \quad \xi \in \g.$$

\subsection{Time-dependent Hamiltonian systems and presymplectic structures}

For time-dependent Hamiltonian systems, the configuration space is the total space of a fibration over the real line,
\[
  \Pi:Q\longrightarrow \mathbb{R},
\]
called the configuration bundle. Then, the phase space of momenta is the dual bundle $V^*\Pi$ of the vertical bundle $V\Pi$ of $\Pi$. One may also consider the extended phase space of momenta $T^\ast Q$ and it is important to stress that $T^*Q$ is a principal $\mathbb{R}$-bundle over $V^*\Pi$, where the dual map  $\mu \colon T^*Q\to V^\ast \Pi$ of the inclusion $V\Pi\hookrightarrow TQ$ is the principal $\mathbb{R}$-bundle projection (see \cite{GRABOWSKA2004398,IGNACIO}). 
In this context, one may use presymplectic structures for the geometric formulation of time-dependent Hamiltonian systems. Indeed, given a Hamiltonian section $h:V^{*}\Pi\to T^{*}Q$ of $\mu\colon T^*Q\to V^*\Pi$, it determines a homogeneous function $F_h:T^{*}Q\to\mathbb{R}$ on $T^*Q$. Moreover, the Hamiltonian section induces the closed two-forms $\Omega_h$ and $\omega_h$ on $V^*\Pi$ and $T^*Q$, respectively, given by
\[
  \Omega_h=h^{*}\omega_Q, \qquad
  \omega_h=\omega_Q+dF_h\wedge \eta_0,
\]
with $\eta_0=(\Pi\circ\pi_Q)^{*}(dt)$.  The 2-form $\omega_h$ defines a presymplectic structure of corank two on $T^{*}Q$ and the form $\Omega_h$ defines a presymplectic structure of corank one on $V^{*}\Pi$, that is, $(V^*\Pi, \Omega_h)$ is a mechanical presymplectic manifold following the terminology in \cite{GUTIERREZSAGREDO}. Indeed, one can show the following facts:
\begin{itemize}
    \item[(i)] $\ker \omega_h = \langle R , X_{F_h}^{\omega_Q} \rangle$, where $R \in \mathfrak{X}(T^*Q)$ is the infinitesimal generator of the $\R$-principal action on $T^*Q$ and $X_{F_h}^{\omega_Q}\in \mathfrak{X}(T^*Q)$ is the Hamiltonian vector field of $F_h$ with respect to $\omega_Q$.

    \item[(ii)] $X_{F_h}^{\omega_Q}$ is $\mu$-projectable over a vector field $R_h \in \mathfrak{X}(V^*\Pi)$, which encodes the time-dependent Hamiltonian system associated to $h$, and $\ker\Omega_h = \langle R_h \rangle$.
\end{itemize}
The following diagram illustrates the situation.
\begin{figure}[ht]
\captionsetup{name=Diagram}
\centering
\[
\xymatrix{
\left(T^*Q,\omega_h,R,X_{F_h}^{\omega_Q}\right)
\ar[d]^{\mu}
\ar@(ul,u)[]^{\mathbb{R}}
\\
\left(V^*\Pi,\Omega_h,R_h\right)
}
\]

\captionsetup{font=footnotesize}
\caption{Principal $\R$-bundle}
\label{diag:reduction}
\end{figure}

In order to develop a reduction procedure for time-dependent Hamiltonian systems, given the configuration bundle $\Pi\colon Q \to \mathbb{R}$, we consider an action $\phi$ of  a Lie group $G$ on $Q$. In \cite{IGNACIO}, the fibration $\Pi$ was assumed to be $G$-invariant. However, as already shown in \cite{GUTIERREZSAGREDO}, this approach is not entirely satisfactory since, for instance, it does not allow conserved quantities with explicit time dependence.

Therefore, we will relax this assumption here, allowing the action of the symmetry group not to preserve $\Pi$ itself, that is, transformations translating the time variable are allowed.
Thus, it is enough to assume that the one-form $\eta=\Pi^{*}(dt)$ is $G$-invariant. This condition is equivalent to the existence of a multiplicative function $m:G\to\mathbb{R}$, or, equivalently,  of an associated Lie algebra $1$-cocycle $c_\phi\colon \mathfrak{g}\to\mathbb{R}$, such that
\begin{equation}\label{intro: condition}
    \Pi \circ \phi_g = \Pi + m(g), \quad \text{for all} \quad g\in G.
\end{equation}
The previous condition is also equivalent to the natural fact that the principal $\R$-action on $T^*Q$ and the cotangent-lifted action on $T^*Q$ commute.

\subsection{Our contribution}

Under condition (\ref{intro: condition}), the principal $\mathbb{R}$-bundle projection $\mu$ is $G$-equivariant  with respect to the cotangent lift of $\phi$ on $T^{*}Q$ and the induced action on $V^{*}\Pi$.  Moreover, if the Hamiltonian section is equivariant, the corresponding homogeneous Hamiltonian function $F_h$ is invariant and one obtains momentum maps adapted to the presymplectic setting.  On $T^{*}Q$, the momentum map takes the form
\[
  J(\alpha_q)=J_c(\alpha_q)-F_h(\alpha_q)c_\phi,
\]
which modifies the canonical cotangent bundle momentum map by a term determined by the extended Hamiltonian function $F_h$ and the $1$-cocycle $c_\phi$. On $V^*\Pi$, the momentum map $J^V:V^*\Pi \to \g^*$ is characterized by the relation 
$$J = J^V \circ \mu.$$

Then, we prove a version of Noether theorem for the time-dependent Hamiltonian dynamics $R_h$. Namely, for each $\xi \in \g$, the smooth function $J_\xi^V: V^*\Pi \to \R$ is a first integral of $R_h$. The function $J_\xi^V$ may depend explicitly on time (see Example \ref{ex: momentum map} and Section \ref{subsection: ex Presymplectic Hamiltonian Systems}). Moreover, using the reduction theory of presymplectic manifolds developed in \cite{Echevarria}, 
we discuss a reduction procedure (for time-dependent Hamiltonian systems) at the extended phase space of momenta as well as the restricted phase space of momenta. More precisely, for a regular value $\nu\in\mathfrak{g}^{*}$ of the momentum maps, the level sets $J^{-1}(\nu)$ and $(J^V)^{-1}(\nu)$ give rise to quotient spaces
\[
  (T^{*}Q)_\nu=J^{-1}(\nu)/G_\nu, \qquad
  (V^{*}\Pi)_\nu=(J^V)^{-1}(\nu)/G_\nu.
\]
These quotients inherit presymplectic structures $\omega_h^\nu$ and $\Omega_h^\nu$ of corank two and corank one, respectively, and they remain related by a reduced principal $\mathbb{R}$-bundle
\[
  \mu_\nu:(T^{*}Q)_\nu\longrightarrow (V^{*}\Pi)_\nu.
\]

Moreover, as in the unreduced setting, we prove that 
$$\ker \omega_h^\nu = \langle R_\nu, X_\nu \rangle,$$
where $R_\nu$ is the infinitesimal generator of the $\R$-action on $(T^*Q)_\nu$ and $X_\nu \in \mathfrak{X}((T^*Q)_\nu)$ is the reduction of the Hamiltonian vector field $X_{F_h}^{\omega_Q} \in \mathfrak{X}(T^*Q)$. In addition, $X_\nu$ is $\mu_\nu$-projectable on the reduced time-dependent dynamics $R_h^\nu \in \mathfrak{X}((V^*\Pi)_\nu)$ and 
$$\ker \Omega_h^\nu = \langle R_h^\nu\rangle.$$
We remark that, in order to ensure that the reduced 2-form $\omega^\nu _h$ (respectively, $\Omega^\nu _h$) is of corank two (resp., of corank one), we have to impose that 
\[
  X_{F_h}^{\omega_Q}(\alpha_q) \notin T_{\alpha_q}(G \cdot \alpha_q), \quad \text{for} \quad \alpha_q \in T_q^*Q, \quad q\in Q,
\]
that is, we have to exclude the relative equilibria of the vector field $X_{F_h}^{\omega_Q}$ (see \cite[Definition 4.1.1]{Marsden_1992}). We postpone for a future paper the analysis of these points (see Section \ref{sec: future}).

So, the following diagram illustrates the reduced setting.
\begin{figure}[ht]
\captionsetup{name=Diagram}
\centering
\[
\xymatrix{
\left((T^*Q)_\nu,\omega_h^\nu,R_\nu,X_\nu\right)
\ar[d]^{\mu_\nu}
\ar@(ul,u)[]^{\mathbb{R}}
\\
\left((V^*\Pi)_\nu,\Omega_h^\nu,R_h^\nu\right)
}
\]
\captionsetup{font=footnotesize}
\caption{Reduced principal $\R$-bundle}
\label{diag:reduction 2}
\end{figure}

The reader can compare with Diagram \ref{diag:reduction}, realizing that both structures coincide.

\subsection{Organization of the paper}
In Section \ref{sec: hamiltonian formulation} we recall the Hamiltonian formulation of time-dependent Hamiltonian systems in terms of the principal $\mathbb{R}$-bundle $\mu:T^{*}Q\to V^{*}\Pi$. In particular, for a Hamiltonian section $h:V^*\Pi \to T^*Q$ it is introduced the presymplectic forms $\omega_h$ and $\Omega_h$, and developed the associated dynamics.  In Section \ref{sec: actions}, we analyze the lifted actions on $T^{*}Q$ and $V^{*}\Pi$ under the invariance condition for $\eta$. In Section \ref{sec: momentum map}, the corresponding presymplectic momentum maps and a Noether theorem are derived. In Section \ref{sec: reduction}, we describe the reduction procedure, proving that the reduced spaces inherit the appropriate presymplectic structures and dynamics. In Section \ref{sec: example elroy}, the previous results are applied to the reduction of an interesting mechanical system: the time-dependent Elroy's Beanie. Finally, in Section \ref{sec: future}, we present the conclusions and some future lines of research.

\bigskip 
\textbf{Acknowledgements:} C. Benítez thanks Gobierno de Canarias and University of La Laguna for an Initium ULL predoctoral grant. The authors acknowledge financial
support from the Spanish Ministry of Science and Innovation under grant PID2022-137909-NB-C22 and useful comments from Nicola Sansonetto about the example in Section \ref{sec: example elroy}.


\section{Hamiltonian formulation}\label{sec: hamiltonian formulation}

\subsection{Principal $\R$-bundle}
Let $Q$ be a connected manifold of dimension $n+1$ and $\Pi: Q\to \R$ be a fibration. Denote by 
$\tau_Q:TQ\to Q$ the canonical projection of the tangent bundle $TQ$ on $Q$. The vertical bundle $V\Pi$ is the subbundle of $TQ$ given by 
\begin{equation}\label{eq: def VPi}
    V\Pi = \{u \in TQ \,\mid\, \eta(\tau_Q(u))(u) = 0\},
\end{equation}
where 
\begin{equation}\label{eq: def eta}
    \eta= \Pi^*(dt) \in \Omega^1(Q).
\end{equation}

If $T^*Q$ (respectively, $V^*\Pi$) is the cotangent bundle of $Q$ (respectively, the dual vector bundle of $V\Pi)$, then the dual map $\mu : T^*Q \to V^*\Pi$ of the inclusion $i: V\Pi \hookrightarrow TQ$ is a principal $\R$-bundle whose principal action $\psi:\R \times T^*Q \to T^*Q$ is defined by
\begin{equation}\label{eq:action psi}
    \psi(s,\alpha_q) = \alpha_q + s \eta(q), \quad \text{for all} \,\,\,\, s\in \R \,\,\,\,\text{and}\,\,\,\, \alpha_q\in T^*_qQ,
\end{equation}
(see, for instance, \cite{GRABOWSKA2004398, IGNACIO}).

If we consider a Hamiltonian section $h\colon V^*\Pi \to T^*Q$ of the principal $\R$-bundle $\mu: T^*Q \to V^*\Pi$, one can associate an extended Hamiltonian function $F_h : T^*Q \to \mathbb{R}$, which is homogeneous of degree $1$ with respect to $\psi$, i.e., $R(F_h)=1$, where $R$ is the infinitesimal generator of $\psi$.

\begin{proposition}\label{prop:function F_h} \cite{GRABOWSKA2004398,IGNACIO} 
    There is a one-to-one correspondence between Hamiltonian sections $h:V^*\Pi \to T^*Q$ of $\mu$, and homogeneous functions $F_h:T^*Q\to\R$ of degree $1$ with respect to the action $\psi$ defined in (\ref{eq:action psi}). This relation is characterized by
    \begin{equation}\label{eq: alpha_q - h o mu = F_h eta}
        \alpha_q-h(\mu(\alpha_q)) = F_h(\alpha_q)\eta(q), \quad \text{for all} \,\,\,\, \alpha_q\in T^*_qQ \,\,\,\,\text{and}\,\,\,\, q\in Q.
    \end{equation}
    
\end{proposition}

A direct consequence of \eqref{eq: alpha_q - h o mu = F_h eta} is the identity
    \begin{equation}\label{eq:Fh(beta_q+seta(q))=Fh(beta_q)+s}
        F_h(\alpha_q +s\eta(q))=F_h(\alpha_q)+s, \quad \text{for all} \,\,\,\, s\in \R \,\,\,\,\text{and}\,\,\,\, \alpha_q\in T^*_qQ.
    \end{equation}

The following diagram illustrates the previous situation
$$\xymatrix{ & T^*Q \ar@(u,ur)[]^{(\R, \psi)} \ar[dl]_{\pi_Q} \ar[dr]_\mu \ar[r]^{F_h} & \R \\ Q  \ar[d]_\Pi && V^*\Pi \ar@/_0.8pc/[ul]_h \ar[ll]_{\Pi^*_{1,0}} \ar[dll]^{\Pi^*_1} \\ \R &&}$$
where $\pi_Q: T^*Q \to Q$ is the canonical projection, 
$\Pi^*_{1,0}:V^*\Pi \to Q$ is the bundle projection of $V^*\Pi$ and we denote by $\Pi^*_1:V^*\Pi\to\R$ the composition $\Pi^*_1 = \Pi \circ \Pi^*_{1,0}$.

\begin{remark}    
The aforementioned spaces can be described in local coordinates as follows (see \cite{IGNACIO}). First, since $\Pi:Q\to \R$ is a submersion, one may consider local coordinates $(q^i,t)$ on $Q$ adapted to $\Pi$, i.e., such that $\Pi(q^i,t) = t$. Denote by $(q^i,t,p_i,p)$ (resp., $(q^i,t,p_i)$) the corresponding local coordinates on $T^*Q$ (resp., on $V^*\Pi$). With respect to them, we have that  
$$\Pi^*_{1,0}(q^i,t,p_i) = (q^i,t) \quad \text{and} \quad \Pi^*_1(q^i,t,p_i) = t.$$
Moreover, the local expressions of the principal $\R$-bundle projection $\mu$ and the corresponding principal action $\psi$ are 
\begin{equation}\label{eq: local mu and psi}
    \mu(q^i,t,p_i,p) = (q^i,t,p_i) \quad \text{and} \quad \psi(s, (q^i,t,p_i,p)) = (q^i, t, p_i, s+p),
\end{equation}
for all $s\in \R$. Finally, if the local expression of the Hamiltonian section $h:V^*\Pi \to T^*Q$ is given by 
\begin{equation}\label{eq: local h}
    h(q^i,t,p_i) = (q^i,t,p_i, -H(q^i,t,p_i)),
\end{equation}
where $H$ is a local function on $V^*\Pi$, then,
\begin{equation}\label{eq: local F_h}
    F_h(q^i,t,p_i,p) = p + H(q^i,t,p_i).
\end{equation}

\end{remark}

\subsection{Presymplectic structures on $T^*Q$ and $V^*\Pi$}

Let $\Pi: Q \to \R$ be a fibration and $h:V^*\Pi \to T^*Q$ a Hamiltonian section of the $\R$-principal bundle $\mu : T^*Q \to V^*\Pi$. The purpose of this subsection is to endow $T^*Q$ and $V^*\Pi$ with certain $2$-forms that we will refer to as presymplectic structures induced by the section $h$. 

On $V^*\Pi$ we consider the closed $2$-form 
\begin{equation}\label{eq:def Omega_h}
    \Omega_h = h^*(\omega_Q),
\end{equation}
where $\omega_Q = -d\lambda_Q$ is the canonical symplectic structure on $T^*Q$, with $\lambda_Q$ the Liouville $1$-form on $T^*Q$. Moreover, we have another closed $2$-form on $T^*Q$ given by
\begin{equation}\label{eq:def omega_h}
    \omega_h = \omega_Q + dF_h \wedge \eta_0,
\end{equation}
where
\begin{equation}\label{eq:def eta_0}
    \eta_0 = (\Pi \circ \pi_Q)^*(dt)=\pi _Q^*\eta .
\end{equation}

The following result shows how the closed $2$-forms $\Omega_h$ and $\omega_h$ are related.
\begin{proposition}\label{prop:mu*(Omega_h)=omega_h Omega_h = h*(omega_h)}
    Let $\mu: T^*Q \to V^*\Pi$ be the principal $\R$-bundle and $h:V^*\Pi \to T^*Q$ a Hamiltonian section of $\mu$. Then, 
    \begin{itemize}
        \item[(i)] $\Omega_h = h^*(\omega_h)$.
        \item[(ii)] $\omega_h = \mu^*(\Omega_h)$.
    \end{itemize}
\end{proposition}

\begin{proof}
    First, note that, using \eqref{eq: alpha_q - h o mu = F_h eta},
    \begin{equation}\label{eq: F_h o h = 0}
        F_h \circ h = 0.
    \end{equation}
    
    As a result, it follows that
    $$h^*(dF_h \wedge \eta_0) = d(F_h \circ h) \wedge h^*(\eta_0) = 0.$$
    Thus, $h^*(\omega_h) = h^*(\omega_Q) = \Omega_h$.

    For the second statement, let 
    \begin{equation}\label{eq: lambda_h}
        \lambda_h = h^*\lambda_Q.
    \end{equation}
    
    A straightforward computation, using the definition of $\lambda_Q$, (\ref{eq: alpha_q - h o mu = F_h eta}) and (\ref{eq: lambda_h}),
    shows that
    \begin{equation}\label{eq: mu*lamda_h = lamda_Q -f_h etA_0}
        \mu^*\lambda_h = \lambda_Q - F_h\eta_0.
    \end{equation}
    Using (\ref{eq:def Omega_h}), (\ref{eq:def omega_h}) and (\ref{eq: lambda_h}), we deduce
    
    $$d(\mu^*\lambda_h) =-\mu^*(h^*(\omega_Q))= -\mu^*(\Omega_h),$$
    and 
    $$d(\lambda_Q - F_h\eta_0)= -\omega_Q -dF_h \wedge \eta_0=-\omega _h.$$
    Taking differentials in both sides of (\ref{eq: mu*lamda_h = lamda_Q -f_h etA_0}) we conclude
    $$\mu^*(\Omega_h)= \omega_h.$$
\end{proof}
\begin{remark}\label{remark: structures coordinates}
If we consider Darboux coordinates on $T^*Q$ adapted to the fibration $\Pi$, then
\begin{equation}\label{eq: coordinates omega_Q}
    \omega_Q = dq^i \wedge dp_i + dt\wedge dp, \quad \eta = dt \quad \text{and} \quad \eta_0 = dt.
\end{equation}
Hence, from (\ref{eq: local h}) and (\ref{eq: local F_h}), 
\begin{equation}\label{eq: coordinates omega_h and coordinates Omega_h}
    \begin{split}
        \omega_h & = dq^i \wedge dp_i + \frac{\partial H}{\partial q^i}dq^i \wedge dt + \frac{\partial H}{\partial p_i}dp_i \wedge dt,
        \\ \Omega_h & = dq^i \wedge dp_i + \frac{\partial H}{\partial q^i}dq^i \wedge dt + \frac{\partial H}{\partial p_i}dp_i \wedge dt.
    \end{split}
\end{equation}
\end{remark}

\begin{definition}\cite{Echevarria, GUTIERREZSAGREDO}
    A presymplectic structure on an $m$-dimensional manifold $M$ is a closed $2$-form $\omega\in \Omega^2(M)$ with constant rank $2k$. The pair $(M,\omega)$ is called a presymplectic manifold of corank $m-2k$.
    
    Moreover, a vector field $X\in \mathfrak{X}(M)$ is said to be a Hamiltonian vector field with respect to the presymplectic structure $\omega$ if there exists a function $f_X: M \to \R$ such that
$$i_X\omega = df_X.$$
\end{definition}

\begin{remark}
    In \cite{GUTIERREZSAGREDO}, the presymplectic manifolds $(M, \omega)$ of corank $1$ are called mechanical presymplectic structures.
\end{remark}

In what follows we will show that $\omega_h$ defines a presymplectic structure of corank $2$ on $T^*Q$. First, we prove the following lemma.
\begin{lemma}
If $R$ denotes the infinitesimal generator of the principal action $\psi$ on $T^*Q$, given in (\ref{eq:action psi}), and $\eta_0 =(\Pi \circ \pi_Q)^*(dt)$, then 
\begin{equation}\label{eq:eta_0 = -i_Romega_Q}
    \eta_0 = -i_R\omega_Q.
\end{equation}
\end{lemma}
\begin{proof}
From the definition of $\psi$,
\begin{equation}\label{eq:def R}
    R(\alpha_q) = \left.\frac{d}{dt}\right\vert_{t=0}(\alpha_q + t \eta(q))
\end{equation}
that is, $R$ is just the vertical lift $\eta^V$ of $\eta$, which implies that $i_{\eta^V}\omega_Q = -\pi_Q^*(\eta)$,
(see, for instance, \cite[page 241]{MR1021489}). So,
$$i_R\omega_Q = -\pi_Q^*(\eta) = -(\Pi \circ \pi_Q)^*(dt) = -\eta_0.$$    
\end{proof}

Now, we prove the announced result.

\begin{proposition}\label{prop:ker omega_h}
    The $2$-form $\omega_h$ defined in (\ref{eq:def omega_h}) is a presymplectic structure on $T^*Q$ of corank $2$. In fact, we have 
    $$\ker\omega_h = \langle R, X_{F_h}^{\omega_Q}\rangle,$$
    where $R$ denotes the infinitesimal generator of the action $\psi$ associated to the principal bundle $\mu: T^*Q \to V^*\Pi$ and $X_{F_h}^{\omega_Q}\in \mathfrak{X}(T^*Q)$ is the Hamiltonian vector field of $F_h: T^*Q \to \R$ with respect to $\omega_Q$, that is,
    $$i_{X_{F_h}^{\omega_Q}}\omega_Q = d F_h.$$
    
\end{proposition}
\begin{proof}
    Firstly, we have that the vector fields $R, \, X_{F_h}^{\omega_Q} \in \ker\,\omega_h$. Indeed, 
    $$i_{X_{F_h}^{\omega_Q}}\omega_h = i_{X_{F_h}^{\omega_Q}}(\omega_Q + dF_h \wedge \eta_0) = dF_h - \eta_0(X_{F_h}^{\omega_Q})dF_h.$$
    Note that, using the homogeneity of $F_h$ with respect to $\psi$ and (\ref{eq:eta_0 = -i_Romega_Q}), we have that 
    \begin{equation}\label{eq:eta_0(X_(F_h)(omega_q))=1}
        \eta_0(X_{F_h}^{\omega_Q}) = -i_R\omega_Q(X_{F_h}^{\omega_Q}) = i_{X_{F_h}^{\omega_Q}}\omega_Q(R) = R(F_h)=1.
    \end{equation}
    So, $X_{F_h}^{\omega_Q} \in \ker\,\omega_h$.

    In addition, using again (\ref{eq:eta_0 = -i_Romega_Q}), it follows that
    $$i_R\omega_h = i_R(\omega_Q + dF_h \wedge \eta_0) = -\eta_0 +R(F_h)\eta_0 -\eta_0(R)dF_h,$$
    and
    \begin{equation}\label{eq:eta_0(R)=0}
        \eta_0(R) = -i_R\omega_Q(R) =0.
    \end{equation}
Therefore, from these equations and the homogeneity of $F_h$ with respect to $\psi$, we have $R \in \ker\,\omega_h$.

    Note that, using (\ref{eq:eta_0(X_(F_h)(omega_q))=1}) and (\ref{eq:eta_0(R)=0}), it follows that $R$ and $X_{F_h}^{\omega_Q}$ are independent. To conclude, we will see that if there exists $X\in \ker\,\omega_h$, then $X$ is a linear combination of $R$ and $X_{F_h}^{\omega_Q}$. In fact, 
    \[
    0 = i_X\omega_h = i_X\omega_Q + X(F_h) \eta_0 -\eta_0(X) dF_h
      =i_X\omega_Q-X(F_h)i_{R}\omega_Q - \eta_0(X)i_{X_{F_h}^{\omega_Q}}\omega_Q.
    \]
    Consequently, $X = X(F_h)R + \eta_0(X)X_{F_h}^{\omega_Q}$, which means that 
    $$\ker\,\omega_h = \langle R, X_{F_h}^{\omega_Q} \rangle.$$
\end{proof}
\begin{remark}
The local expressions (with respect to the coordinates in Remark \ref{remark: structures coordinates}) of the vector fields involved in the previous result are as follows
\begin{equation}\label{eq:coordinates X_F_h^omega_Q and R}
    X_{F_h}^{\omega_Q} = \frac{\partial}{\partial t} - \frac{\partial H}{\partial t}\frac{\partial}{\partial p} + \frac{\partial H}{\partial p_i}\frac{\partial}{\partial q^i} - \frac{\partial H}{\partial q^i}\frac{\partial}{\partial p_i} \quad \text{and} \quad R = \frac{\partial}{\partial p}.
\end{equation}
\end{remark}
From \cite{GUTIERREZSAGREDO}, we know that $\Omega_h$ is a presymplectic structure on $V^*\Pi$ of corank 1, and
\begin{equation}\label{eq: ker Omega_h = R_h}
    \ker\Omega_h = \langle R_h \rangle,
\end{equation}
where the vector field $R_h$ on $V^*\Pi$ is characterized by the conditions 
\begin{equation}\label{eq: i_R_h Omega_h = 0 y i_R_h eta_1 =1}
    i_{R_h}\Omega_h = 0 \quad \text{and} \quad i_{R_h}\eta_1 =1,
\end{equation}
and $\eta_1 = \Pi_{1,0}^*(\Pi^*(dt))= \Pi_{1,0}^*\eta$. The local expression (with respect to the coordinates in Remark \ref{remark: structures coordinates}) of the vector field $R_h$ is given by
\begin{equation}\label{eq: R_h coordinates}
    R_h = \frac{\partial}{\partial t} + \frac{\partial H}{\partial p_i}\frac{\partial}{\partial q^i} - \frac{\partial H}{\partial q^i}\frac{\partial}{\partial p_i}.
\end{equation}

\subsection{Dynamics on $T^*Q$ and $V^*\Pi$}

The vector fields $X_{F_h}^{\omega_Q}$ and $R_h$ on $T^*Q$ and $V^*\Pi$, respectively, determine the dynamics on the corresponding spaces. In this subsection, we analyze how the principal $\mathbb{R}$-bundle projection $\mu$ and the Hamiltonian section $h$ relate these vector fields.

\begin{proposition}\label{prop:X_F_h^omega_Q is mu projectable}
    The vector field $X_{F_h}^{\omega_Q}$ is 
    \begin{itemize}
        \item[(i)] $\R$-invariant with respect to the action $\psi$, i.e.,
        $$T_{\alpha_q}\psi_s(X_{F_h}^{\omega_Q}(\alpha_q)) = X_{F_h}^{\omega_Q}(\psi_s(\alpha_q)), \quad \text{for all} \quad \alpha_q\in T^*_qQ \quad \text{and} \quad s\in \R.$$
        \item[(ii)] $\mu$-projectable over $R_h \in \mathfrak{X}(V^*\Pi)$, i.e.,
        $$T_{\alpha_q}\mu(X_{F_h}^{\omega_Q}(\alpha_q)) = R_h(\mu(\alpha_q)), \quad \text{for all} \quad \alpha_q\in T^*_qQ.$$
    \end{itemize}
\end{proposition}

\begin{proof}
    To prove \textit{(i)} we first see that 
    \begin{equation}\label{eq: psi*omegaQ = omegaQ}
        \psi_s^*(\omega_Q) = \omega_Q, \quad \text{for all} \quad s\in \R.
    \end{equation}
    Indeed, if we consider Darboux coordinates adapted to the fibration $\Pi$, (\ref{eq: psi*omegaQ = omegaQ}) follows from (\ref{eq: local mu and psi}) and (\ref{eq: coordinates omega_Q}). From this fact, we have
    \begin{equation}\label{eq: T X_Fh^omegaQ}
        T_{\alpha_q}\psi_s(X_{F_h}^{\omega_Q}(\alpha_q)) = X_{F_h \circ \psi_s}^{\omega_Q}(\psi_s(\alpha_q)), \quad \text{for all} \quad \alpha_q\in T^*_qQ \quad \text{and} \quad s\in \R.
    \end{equation}
    Moreover, using (\ref{eq:action psi}) and (\ref{eq:Fh(beta_q+seta(q))=Fh(beta_q)+s}), we deduce that 
    \begin{equation}\label{eq: Fh o psi_s = Fh +s}
        F_h \circ \psi_s = F_h +s, \quad \text{for all} \quad s \in \R.
    \end{equation}
    Thus, from (\ref{eq: T X_Fh^omegaQ}) and (\ref{eq: Fh o psi_s = Fh +s}), we conclude
    $$T_{\alpha_q}\psi_s(X_{F_h}^{\omega_Q}(\alpha_q)) = X_{F_h}^{\omega_Q}(\psi_s(\alpha_q)), \quad \text{for all} \quad \alpha_q\in T^*_qQ \quad \text{and} \quad s\in \R.$$
    Item \textit{(ii)} is proved in \cite{GRABOWSKA2004398,IGNACIO} (it also follows using (\ref{eq: local mu and psi}), (\ref{eq:coordinates X_F_h^omega_Q and R}) and (\ref{eq: R_h coordinates})).
\end{proof}

On the other hand, using (\ref{eq: local h}), (\ref{eq:coordinates X_F_h^omega_Q and R}) and (\ref{eq: R_h coordinates}), we obtain the following result.

\begin{proposition}\label{prop: R_h h-projectable X_F_h^omega_Q}
    The vector field $R_h$ on $V^*\Pi$ is $h$-related with $X_{F_h}^{\omega_Q}$, that is, 
    $$T_{\beta_q}h(R_h(\beta_q)) = X_{F_h}^{\omega_Q}(h(\beta_q)), \quad \text{for all} \quad  \beta_q \in V_q^*\Pi.$$
\end{proposition}

The following diagram summarizes the dynamics and the geometry of this system

$$\xymatrix{\left(T^*Q,\,\omega_h,\, X_{F_h}^{\omega_Q}\right)  \ar[dd]^{\mu} \\
  \\ \left(V^*\Pi,\,\Omega_h, \, R_h\right) \ar@/^3pc/[uu]^{h} 
}$$

To conclude this section, we present an example (see \cite{GUTIERREZSAGREDO}), which illustrates the theoretical results developed so far.

\begin{example}[$N$-dimensional harmonic oscillator seen by an observer that moves with a constant velocity]\label{ex:harmonic oscillator}
    Let us consider a classical particle of mass $m$ moving in a harmonic oscillator potential with frequency $\Omega$ in $N$ dimensions. The potential is centered at the origin, $\mathbf{0} = (0, \ldots, 0) \in \mathbb{R}^N$. Suppose that the system is observed from an inertial frame moving at a constant velocity $\mathbf{v} = (v_1, \ldots, v_N)$. Without loss of generality, we may choose $\mathbf{v} = (v, 0, \ldots, 0)$. 
    
    Under this assumption, the total space $Q$ of the configuration bundle is $\R^{N+1}$, the bundle projection is $\Pi(q^1, \ldots, q^N,t)= t$, the spaces $V^*\Pi$ and $T^*Q$ are $(\R^N \times \R) \times \R^N$ and $(\R^N \times \R) \times (\R^N \times \R)$ with coordinates $((q^1, \ldots, q^N,t),(p_1, \ldots, p_N))$ and $((q^1, \ldots, q^N,t),(p_1, \ldots, p_N,p))$, respectively, and the principal $\R$-bundle projection $\mu: T^*Q \to V^*\Pi$ is
    \begin{equation}\label{ex: mu}
        \mu((q^1, \ldots, q^N,t),(p_1, \ldots, p_N,p)) = ((q^1, \ldots, q^N,t),(p_1, \ldots, p_N)).
    \end{equation}
    
    Now, we consider the Hamiltonian section $h: V^*\Pi \to T^*Q$ induced by the time-dependent Hamiltonian function $H:V^*\Pi \cong (\R^N \times \R) \times \R^N \to \R$ given by
    $$H(q^i,t,p_i) = \frac{1}{2m}p_i^2 + \frac{1}{2}m\Omega^2\left( (q^1+vt)^2 + (q^\alpha)^2\right),$$
    where Latin indices run from $1$ to $N$, and Greek indices from $2$ to $N$. So, using (\ref{eq: local h}), the Hamiltonian section $h: (\R^N \times \R) \times \R^N \to (\R^N \times \R) \times (\R^N \times \R)$ is
    \begin{equation}\label{ex: h}
        h(q^i,t,p_i) = \left(q^i,t,p_i, -\frac{1}{2m}p_i^2 - \frac{1}{2}m\Omega^2\left( (q^1+vt)^2 + (q^\alpha)^2\right)\right),
    \end{equation}
    and, by (\ref{eq: local F_h}), the function $F_h: (\R^N \times \R)  \times (\R^N \times \R) \to \R$ is given by 
    \begin{equation}\label{ex: F_h}
        F_h(q^i,t,p_i, p) = p + H(q^i,t,p_i) = p + \frac{1}{2m}p_i^2 + \frac{1}{2}m\Omega^2\left( (q^1+vt)^2 + (q^\alpha)^2\right).
    \end{equation}
    In this setting, the differential of the Hamiltonian is given by
    $$dH = \frac{1}{m}p_idp_i + m\Omega^2\left( (q^1 + vt)dq^1 + q^\alpha dq^\alpha + (q^1 + vt)vdt\right).$$

    We consider the canonical symplectic structure 
    \begin{equation}\label{ex: omega_Q coordinates}
        \omega_Q = dq^i \wedge dp_i + dt \wedge dp
    \end{equation}
    on $T^*Q \cong (\R^N \times \R) \times (\R^N \times \R)$. According to (\ref{eq: coordinates omega_h and coordinates Omega_h}), the presymplectic structure $\omega_h$ of corank $2$ on $T^*Q$ takes the form

    \begin{equation}\label{ex: omega_h coordinates}
        \omega_h = dq^i \wedge dp_i + \frac{1}{m}p_idp_i \wedge dt +m\Omega^2\left(q^idq^i\wedge dt +vt dq^1 \wedge dt\right).
    \end{equation}

    Similarly, using (\ref{eq: coordinates omega_h and coordinates Omega_h}), the presymplectic structure $\Omega_h$ of corank $1$ on $V^*\Pi \cong (\R^N \times \R) \times \R^N$ is given by
    \begin{equation}\label{ex: Omega_h coordinates}
        \Omega_h = dq^i \wedge dp_i + \frac{1}{m}p_idp_i \wedge dt +m\Omega^2\left(q^idq^i\wedge dt +vt dq^1 \wedge dt\right).
    \end{equation}

    Finally, from equations (\ref{eq:coordinates X_F_h^omega_Q and R}) and (\ref{eq: R_h coordinates}), it follows that the Hamiltonian vector field $X_{F_h}^{\omega_Q}$ and the vector field $R_h$ are given by
    \begin{equation}\label{ex: X_F_h^omega_Q}
        X_{F_h}^{\omega_Q} = \frac{\partial}{\partial t} + \frac{1}{m}p_i \frac{\partial}{\partial q^i}- m\Omega^2\left( (q^1+vt)\left(v\frac{\partial}{\partial p} + \frac{\partial}{\partial p_1}\right) + q^\alpha\frac{\partial}{\partial p_\alpha}\right),
    \end{equation}
    and
    \begin{equation}\label{ex: R_h}
        R_h = \frac{\partial}{\partial t} + \frac{1}{m}p_i \frac{\partial}{\partial q^i}- m\Omega^2\left( (q^1+vt)\frac{\partial}{\partial p_1} + q^\alpha\frac{\partial}{\partial p_\alpha}\right).
    \end{equation}

\end{example}


\section{The actions on $T^*Q$ and $V^*\Pi$}\label{sec: actions}

Given a fibration, $\Pi \colon Q\to \R$, consider an action $\phi: G\times Q \to Q$ of a Lie group $G$ on $Q$. In \cite{IGNACIO}, the fibration $\Pi$ was assumed to be $G$-invariant. Here, we will relax this assumption, by allowing the action not to preserve $\Pi$.

Denote by $T\phi : G\times TQ \to TQ$ the tangent lift of $\phi$ and $T^*\phi: G\times T^*Q \to T^*Q$ the cotangent lift of $\phi$, given by 
\begin{equation}\label{eq: def T*phi}
    (T^*\phi)(g,\alpha_q)(u_{\phi_g(q)}) = (T_g^*\phi(\alpha_q))(u_{\phi_g(q)}) = \alpha_q((T_{\phi_g(q)}\phi_{g^{-1}})(u_{\phi_g(q)})),
\end{equation}
for $g\in G$, $q\in Q$, $\alpha_q\in T_q^*Q$ and $u_{\phi_g(q)}\in T_{\phi_g(q)}Q$.
\begin{proposition}\label{prop: T*phi and psi commute}
    The cotangent lift $T^*\phi : G\times T^*Q \to T^*Q$ and the principal action $\psi: \R\times T^*Q \to T^*Q$, given by (\ref{eq:action psi}), commute if and only if $\phi_g^*\eta = \eta$ for all $g \in G$.
\end{proposition}
\begin{proof}
    Let $\beta_q\in T^*_qQ$ and $u_{\phi_g(q)}\in T_{\phi_g(q)}Q$. Then,
    \[
       \psi_s(T_g^*\phi(\beta_q))(u_{\phi_g(q)})  
         =  T_g^*\phi(\beta_q)(u_{\phi_g(q)}) + s\eta(\phi_g(q))(u_{\phi_g(q)}) 
    \]
    and 
    \[
        T_g^*\phi(\psi_s(\beta_q))(u_{\phi_g(q)})  = T_g^*\phi(\beta_q)(u_{\phi_g(q)}) + s(\phi_{g^{-1}}^*\eta)(\phi_g(q))(u_{\phi_g(q)}). 
    \]

    Thus, $\psi_s \circ T_g^*\phi = T_g^*\phi \circ \psi_s$, for all $s\in \R$ and $g\in G$ if and only if $\phi_g^*\eta = \eta$.
    
\end{proof}

From now on, we will always assume the following hypothesis. 

\begin{assumption}\label{assum 1}
    $\phi_g^*\eta = \eta$ for all $g \in G$, i.e., $\eta$ is $G$-invariant.
\end{assumption}

Note that the condition $\Pi \circ \phi_g = \Pi$, for all $g \in G$, imposed in \cite{IGNACIO}, implies Assumption \ref{assum 1}. However, the converse does not hold, as shown by the following result.

\begin{proposition}
The $1$-form $\eta$ is $G$-invariant if and only if there exists a multiplicative function $m: G \to \R$ such that
$$\Pi\circ \phi_g= \Pi + m(g), \quad \text{for all} \quad g\in G.$$
As a consequence, there exists $c_\phi : \g \to \R$, $\g$ being the Lie algebra of $G$, given by
\begin{equation}\label{eq:cocycle}
        c_\phi = dm(e) \in \g^*,
\end{equation}
which satisfies,
\begin{equation}\label{eq: cocycle equation}
    c_\phi[\xi, \eta]=0, \quad \text{for all} \quad \xi, \eta \in \g,
\end{equation}
that is, $c_\phi \in \g^*$ is a $1$-cocycle of $\g$.
\end{proposition}
\begin{proof}
    As $\phi_g^*\eta = \eta$ for all $g \in G$, it follows that $(\Pi \circ \phi_g)^*(dt) = \Pi^*(dt)$ and 
    \begin{equation}\label{eq: T_q(Pi o phi_g)=T_q Pi}
        d(\Pi \circ \phi_g)=d\Pi, \quad \text{for all} \quad g\in G .
    \end{equation}
    Therefore, using that $Q$ is connected, it holds if and only if there exists $m\in \mathcal{C}^\infty(G)$ such that $$\Pi\circ \phi_g = \Pi + m(g), \quad \text{for all} \quad g\in G.$$
    A straightforward computation proves that $m$ is a multiplicative function, i.e., $m(gh) = m(g) + m(h)$, for any $g,h\in G$.
    
\end{proof}

The following lemma, which relates the $1$-cocycle $c_\phi$ to the $1$-form $\eta$, will be useful in the next section.

\begin{lemma}\label{lemma:eta_0(xi_T*Q)=c_phi(xi)}
If $c_\phi$ is the 1-cocycle given by (\ref{eq:cocycle}), then, $$c_\phi(\xi) =\eta_0(\xi_{T^*Q})= \eta(\xi_Q), \quad \text{for all} \quad \xi \in \g,$$
where $\xi_Q$ (resp. $\xi_{T^*Q}$) is the infinitesimal generator of $\phi$ (resp. of $T^*\phi$) for $\xi\in \g$. Consequently, for each fixed $\xi \in \g$,  the smooth function $\eta_0(\xi_{T^*Q})$ on $T^*Q$ is constant.
\end{lemma}
\begin{proof}
    Let $\xi \in \g$, since $(\Pi \circ \phi_q)(g)= \Pi(q) + m(g)$ for all $q\in Q$ and $g\in G$, 
    \begin{align*}
        c_\phi(\xi) & = dm(e)(\xi) = d(\Pi \circ \phi_q)(e)(\xi) = \Pi^*(dt)(\xi_Q(q)) = \eta(\xi_Q(q)).
    \end{align*}
In addition, using $T_{\alpha_q}\Pi_Q(\xi_{T^*Q}(\alpha_q)) = \xi_Q(q)$ and \eqref{eq:def eta_0},
\begin{align*}
    \eta_0(\xi_{T^*Q})(\alpha_q) & = ((\Pi \circ \Pi_Q)^*(dt)(\alpha_q))(\xi_{T^*Q}(\alpha_q))  = dt(\Pi(q))(T_q\Pi(T_{\alpha_q}\Pi_Q(\xi_{T^*Q}(\alpha_q)))) 
    \\ & = dt(\Pi(q))(T_q\Pi(\xi_Q(q))) = \Pi^*(dt)(\xi_Q(q)) = \eta(\xi_Q(q)).
\end{align*}

\end{proof}

From (\ref{eq: T_q(Pi o phi_g)=T_q Pi}), it follows that the tangent lift $T\phi$ restricts to $V\Pi$. Thus, we have a linear action $V\phi: G \times V\Pi \to V\Pi$ of $G$ on $V\Pi$. Now, we denote by $V^*\phi: G\times V^*\Pi \to V^*\Pi$ the dual action  of $V\phi$, given by
\begin{equation}\label{eq: def V*phi}
    (V_g^*\phi(\beta_q))(u_{\phi_g(q)})= \beta_q((V_{g^{-1}}\phi)(u_{\phi_g(q)})) = \beta_q((T_{\phi_g(q)}\phi_{g^{-1}})(u_{\phi_g(q)})),
\end{equation}
for $g\in G$, $q\in Q$, $\beta_q\in V^*_q\Pi$ and $u_{\phi_g(q)}\in V_{\phi_g(q)}\Pi$.

\begin{proposition}\label{prop: mu is G-equivariant} 
     The principal $\R$-bundle projection $\mu: T^*Q\to V^*\Pi$ is $G$-equivariant when we consider the cotangent lift of $\phi$ on $T^*Q$ and $V^*\phi$ on $V^*\Pi$.
\end{proposition}

\begin{proof}
Since the tangent lift $T\phi$ restricts to $V\Pi$, then
\begin{equation}\label{eq: i o Vphi = Tphi o i}
    i \circ V_g\phi = T_g\phi \circ i, \quad \text{for all} \quad g\in G,
\end{equation}
where $i: V\Pi \hookrightarrow TQ$ denotes the inclusion map. Dualizing (\ref{eq: i o Vphi = Tphi o i}), we obtain that $V_g^*\phi\circ \mu = \mu\circ T_g^*\phi$, for all $g\in G$, that is, $\mu$ is $G$-equivariant.

\end{proof}

As a result of the $G$-equivariance of $\mu$, it follows that 
\begin{equation}\label{eq:mu G-equivariant inf gen}
    T_{\alpha_q}\mu(\xi_{T^*Q}(\alpha_q))= \xi_{V^*\Pi}(\mu(\alpha_q)), \quad \text{for all} \quad \alpha_q\in T^*_qQ,
\end{equation}
where $\xi_{T^*Q}$ and $\xi_{V^*\Pi}$ denote the infinitesimal generators of $\xi \in \g$ associated with $T^*\phi$ and $V^*\phi$, respectively. 

Moreover, we have that 

\begin{proposition}\label{prop:F_h G-invariant}
    The Hamiltonian section $h:V^*\Pi \to T^*Q$ of $\mu$ is $G$-equivariant, i.e.,
    $$T_g^*\phi \circ h = h\circ V_g^*\phi, \quad \text{for all} \quad g\in G,$$
    if and only if the Hamiltonian function $F_h:T^*Q\to\R$ induced by $h$ is $G$-invariant, i.e., $$F_h \circ T^*_g\phi = F_h, \quad \text{for all} \quad g\in G.$$
\end{proposition}
\begin{proof}
Let $\beta_q \in T_q^*Q$. Then, from (\ref{eq: alpha_q - h o mu = F_h eta}), Proposition \ref{prop: mu is G-equivariant} and Assumption \ref{assum 1}, it follows that
\begin{equation}\label{eq: proof h o Vphi}
    F_h(T_g^*\phi(\beta_q))(T_g^*\phi(\eta(q))) = F_h(T_g^*\phi(\beta_q))\eta(\phi_g(q)) = T_g^*\phi(\beta_q)- h(V_g^*\phi(\mu(\beta_q))).
\end{equation}

On the other hand, using again (\ref{eq: alpha_q - h o mu = F_h eta}) and Assumption \ref{assum 1}, we deduce that
\begin{equation}\label{eq: proof Tphi o h}
    F_h(\beta_q)(T_g^*\phi(\eta(q))) = T_g^*\phi(\beta_q-h(\mu(\beta_q))) = T_g^*\phi(\beta_q) - T_g^*\phi(h(\mu(\beta_q))).
\end{equation}

So, from (\ref{eq: proof h o Vphi}) and (\ref{eq: proof Tphi o h}), the result follows.

\end{proof}

From now on, we will assume that:

\begin{assumption}\label{assum: h equivariant}
    The Hamiltonian section $h:V^*\Pi \to T^*Q$ of $\mu$ is $G$-equivariant, i.e.,
    $$T_g^*\phi \circ h = h\circ V_g^*\phi, \quad \text{for all} \quad g\in G.$$
\end{assumption}

The following diagram illustrates the situation with these new actions
$$\xymatrix{ & T^*Q \ar@(l,ul)[]^{(G,T^*\phi)} \ar@(u,ur)[]^{(\R, \psi)}\ar[dl]_{\pi_Q} \ar[dr]_\mu \ar[r]^{F_h} & \R \\ Q \ar@(l,ul)[]^{(G,\phi)} \ar[d]_\Pi && V^*\Pi \ar@/_0.8pc/[ul]_h \ar[ll]_{\Pi^*_{1,0}} \ar[dll]^{\Pi^*_1} \ar@(ur,r)[]^{(G,V^*\phi)} \\ \R &&}$$

\begin{example}\label{ex: actions T*phi V*phi}
    Following Example \ref{ex:harmonic oscillator}, we define the Lie group action $\phi: \R \times Q \to Q$ on $Q$ as follows
    $$\phi(s, (q^i, t)) = (q^1-vs,q^\alpha,t+s),$$
    where Latin indices run from $1$ to $N$, while Greek indices run from $2$ to $N$. It is clear that the $1$-form $\eta = \Pi^*(dt)$ is $G$-invariant. Then, the cotangent lift $T^*\phi : \R \times T^*Q \to T^*Q$ is given by
    \begin{equation}\label{ex: action T*phi}
        T^*\phi(s,(q^i,t,p_i,p))=(q^1-vs,q^\alpha,t+s,p_i,p).
    \end{equation}
     Moreover, the action $V^*\phi : \R \times V^*\Pi \to V^*\Pi$ is given by
    \begin{equation}\label{ex: action V*phi}
        V^*\phi(s,(q^i,t,p_i))=(q^1-vs,q^\alpha,t+s,p_i).
    \end{equation}
    Therefore, the fundamental vector fields associated with these actions are
    \begin{equation}\label{ex:generadores infinitesimales}
        \xi_Q = \xi\left(\frac{\partial}{\partial t} - v \frac{\partial}{\partial q^1} \right), \,\,\, \xi_{T^*Q} = \xi\left(\frac{\partial}{\partial t} - v \frac{\partial}{\partial q^1} \right) \,\,\, \text{and} \,\,\,\, \xi_{V^*\Pi} = \xi\left(\frac{\partial}{\partial t} - v \frac{\partial}{\partial q^1} \right),
    \end{equation}
    for all $\xi \in \R$. Finally, from (\ref{ex: F_h}), we deduce that the function $F_h: T^*Q \to \R$ is $\R$-invariant with respect to $T^*\phi$.
\end{example}


\section{The presymplectic momentum maps}\label{sec: momentum map}

The goal of this section is to show that the cotangent lift $T^*\phi$ and the action $V^*\phi$ on the presymplectic manifolds $(T^*Q, \omega_h)$ and $(V^*\Pi, \Omega_h)$, respectively, are endowed with a momentum map. These momentum maps will play a crucial role in the reduction procedure developed in the next section.

\begin{definition}\label{def: presymplectic Hamiltonian action}
    A \textit{presymplectic Hamiltonian action}  on a presymplectic manifold $(M, $ $\omega)$ is an action $\Phi: G\times M \to M$ of a Lie group $G$ on $M$, satisfying the following properties: 
\begin{itemize}
    \item[(i)] The action $\Phi$ is presymplectic, i.e., $\Phi_g^*\omega = \omega$, for all $g\in G$.
    \item[(ii)] For all $x\in M$,
    \begin{equation}\label{eq: ker omega cap T_x(G x) = 0}
        \ker \omega (x) \,\cap\, T_x(G\cdot x) = \{ 0 \}.
    \end{equation}
    \item[(iii)] For every $\xi \in \g$, the infinitesimal generator $\xi_{M}$ of $\Phi$ is a Hamiltonian vector field with respect to $\omega$, i.e., there exists a smooth map $J:M \to \g^*$ such that
    \begin{equation}\label{eq:i_xi_M omega = dJ_xi}
        i_{\xi_M}\omega = dJ_\xi,
    \end{equation}
    where the function $J_\xi: M\to \R$ is given by $J_\xi(x) = J(x)(\xi)$ for all $x\in M$. The function $J$ is called the \textit{momentum map} for $\Phi$ on $(M,\omega)$.
    
\end{itemize}
We call $(M,\omega, \Phi, J)$ a symmetric presymplectic Hamiltonian system.
\end{definition}

\begin{remark}\label{remark: J_c momentum map}
We recall that the cotangent lift $T^*\phi$ of the $G$-action on $Q$ is a Hamiltonian action on the standard symplectic manifold $(T^*Q, \omega_Q)$ (see, for instance, \cite{AM}), with the $Ad^*$-equivariant canonical momentum map $J_c: T^*Q \to \g^*$ given by 
\begin{equation}\label{eq:J_c momentum map}
    J_c(\alpha_q)(\xi) = \alpha_q(\xi_Q(q)), \quad \text{for all} \quad  \alpha_q\in T_q^*Q \quad \text{and} \quad \xi\in \g.
\end{equation}
In particular, $(T^*Q,\omega_Q, T^*\phi, J_c)$ is a symmetric presymplectic Hamiltonian system.
\end{remark}

In what follows, we will describe condition (\ref{eq: ker omega cap T_x(G x) = 0})  in Definition \ref{def: presymplectic Hamiltonian action} in the particular case of the presymplectic manifold $(T^*Q, \omega_h)$ endowed with the cotangent lift $T^*\phi$, that is, we will characterize the equation
    \begin{equation}\label{eq: transversal condition T*Q}
        \ker \omega_h (\alpha_q) \cap T_{\alpha_q}(G\cdot \alpha_q) = \{ 0 \}, \quad \text{for} \quad \alpha_q \in T^*Q.
    \end{equation}

We will assume that the action $\phi : G \times Q \to Q$ is free and proper. Then, from \cite[Example 2.3.6]{ORTEGA}, the cotangent lift $T^*\phi : G \times T^*Q \to T^*Q$ is also free and proper and the space of orbits $T^*Q/G$ is a smooth manifold in such a way that the projection $\pi : T^*Q \to T^*Q /G$ is a surjective submersion.

 Now, suppose that (\ref{eq: transversal condition T*Q}) doesn't hold. Then, by Proposition \ref{prop:ker omega_h}, there exist $\alpha_q \in T^*Q$, $\lambda_1, \lambda_2\in \R$ and $\xi \in \g^*$ such that
    \begin{equation}\label{eq: lambda R + lamba' x}
        \lambda_1 R(\alpha_q) + \lambda_2 X_{F_h}^{\omega_Q}(\alpha_q) = \xi_{T^*Q}(\alpha_q),
    \end{equation}
    Evaluating (\ref{eq: lambda R + lamba' x}) on $F_h$, we obtain
    $$\lambda_1 R(\alpha_q)(F_h) = \xi_{T^*Q}(\alpha_q)(F_h).$$
    Moreover, from Propositions \ref{prop:function F_h} and \ref{prop:F_h G-invariant}, it follows that $\lambda_1=0$. Then, (\ref{eq: lambda R + lamba' x}) implies 
    \begin{equation}\label{eq: x_FhomegaQ in orbit}
        X_{F_h}^{\omega_Q}(\alpha_q) \in T_{\alpha_q}(G \cdot \alpha_q).
    \end{equation}
    In conclusion, we have proved that (\ref{eq: transversal condition T*Q}) is equivalent to
    \begin{equation}\label{eq: x_FhomegaQ not in orbit}
        X_{F_h}^{\omega_Q}(\alpha_q) \notin T_{\alpha_q}(G \cdot \alpha_q).
    \end{equation}

    Note that the points $\alpha_q \in T^*Q$ that satisfy (\ref{eq: x_FhomegaQ in orbit}) are precisely the relative equilibria of the vector field $X_{F_h}^{\omega_Q}$ (see \cite[Definition 4.1.1]{Marsden_1992}). 
    
    On the other hand, the set of points $\mathcal{C} \subset T^*Q$ that satisfy (\ref{eq: x_FhomegaQ in orbit}) is a closed subset of $T^*Q$. Indeed, since both $\omega_Q$ and $F_h$ are $G$-invariant, it follows easily that $X_{F_h}^{\omega_Q}$ is also $G$-invariant. Therefore, $X_{F_h}^{\omega_Q}$ is $\pi$-projectable to a vector field $\mathcal X$ on the quotient manifold $T^*Q/G$, where $\pi:T^*Q\to T^*Q/G$ denotes the quotient projection. If $\alpha_q\in \mathcal{C}$, then $\pi(\alpha_q)$ is a critical point of $\mathcal X$. Since the set $C$ of critical points of a vector field is closed in $T^*Q/G$, it follows that $\mathcal{C}=\pi^{-1}(C)$ is a closed subset of $T^*Q$.

    As a consequence, the set of points $\mathcal{U} = T^*Q \setminus \mathcal C \subset T^*Q$ satisfying (\ref{eq: x_FhomegaQ not in orbit}) is open in $T^*Q$. In addition, using again that $X_{F_h}^{\omega_Q}$ is $G$-invariant,
    we have that 
    \begin{itemize}
        \item $\mathcal{U}$ is $G$-invariant.
        \item $\mathcal{U}$ is $\R$-invariant with respect to $\psi$.
    \end{itemize}
    
    In fact, let $\alpha_q \in \mathcal{U}$ and $s \in \R$, using (\ref{eq: x_FhomegaQ not in orbit}), it follows that
    \begin{equation}\label{eq: T psi X_Fh^omegaQ}
        T_{\alpha_q} \psi_s(X_{F_h}^{\omega_Q}(\alpha_q)) \notin T_{\alpha_q} \psi_s(T_{\alpha_q}(G \cdot \alpha_q)).
    \end{equation}
    On the other hand, from Proposition \ref{prop: T*phi and psi commute}, we deduce that
    $$T_{\alpha_q} \psi_s(\xi_{T^*Q}(\alpha_q)) = \xi_{T^*Q}(\psi_s(\alpha_q)),$$
    where $\xi_{T^*Q}$ is the infinitesimal generator of $\xi \in \g$ associated with $T^*\phi$. Then, 
    \begin{equation}\label{eq: tangent tangent orbits}
        T_{\alpha_q} \psi_s(T_{\alpha_q}(G \cdot \alpha_q)) = T_{\psi_s(\alpha_q)}(G \cdot \psi_s(\alpha_q)).
    \end{equation}
    Finally, using item \textit{(i)} of Proposition \ref{prop:X_F_h^omega_Q is mu projectable}, together with (\ref{eq: tangent tangent orbits}) and (\ref{eq: T psi X_Fh^omegaQ}), we conclude that
    $$X_{F_h}^{\omega_Q}(\psi_s(\alpha_q)) \notin T_{\psi_s(\alpha_q)}(G \cdot \psi_s(\alpha_q)).$$
    Hence, $\psi_s(\alpha_q) \in \mathcal{U}$. 

    \begin{remark}\label{remark: U = T*Q}
        We have shown that condition (\ref{eq: transversal condition T*Q}) holds on the $G$-invariant and $\mathbb{R}$-invariant open subset $\mathcal{U}$ of $T^*Q$. Henceforth, for simplicity, we assume that $\mathcal{U}=T^*Q$. The case $\mathcal{U}\neq T^*Q$ will be treated separately whenever necessary.
    \end{remark}

Now, we are in a position to obtain the desired symmetric presymplectic Hamiltonian system on $T^*Q$. 

\begin{proposition}\label{prop:J momentum map}
$(T^*Q, \omega_h, T^*\phi, J)$ is a symmetric presymplectic Hamiltonian system with $Ad^*$-equivariant momentum map $J:T^*Q \to \g^*$ given by
    \begin{equation}\label{eq:J momentum map}
        J(\alpha_q)(\xi) = \alpha_q(\xi_Q(q)) - F_h(\alpha_q)c_\phi(\xi),
    \end{equation}
    for all $\alpha_q \in T_q^*Q$ and $\xi \in \g$.
\end{proposition}
\begin{proof}
    First, using Assumption 1, it is deduced that $(T_g^*\phi)^*(\eta_0) = \eta_0$. From this fact, Assumption \ref{assum: h equivariant}, Proposition \ref{prop:F_h G-invariant}
    and that $(T^*\phi)_g^*(\omega_Q)=\omega_Q$, one can prove that $T^*\phi$ is a presymplectic action on $(T^*Q,\omega_h)$, i.e., $$(T^*\phi)_g^*(\omega_h)=\omega_h, \quad \text{for all} \quad g\in G.$$
Now, we show that $J:T^*Q \to \g^*$ is a momentum map for $T^*\phi$ on $(T^*Q,\omega_h)$. Note that, for all $\xi\in \g$, the function $J_\xi: T^*Q\to \R$ is given by
\begin{equation}\label{eq: J_xi}
    J_\xi = (J_c)_\xi-c_\phi(\xi)F_h,
\end{equation}
where $J_c$ is the canonical momentum map (see Remark \ref{remark: J_c momentum map}) and $c_\phi$ is the $1$-cocycle given in (\ref{eq:cocycle}). Then, using Assumption \ref{assum: h equivariant}, Proposition \ref{prop:F_h G-invariant}, Lemma \ref{lemma:eta_0(xi_T*Q)=c_phi(xi)} and Remark \ref{remark: J_c momentum map}, we conclude that
    \begin{equation}\label{eq:i_xi omega_h = dJ_xi}
        \begin{split}
            i_{\xi_{T^*Q}}\omega_h & =  d(J_c)_\xi + dF_h(\xi_{T^*Q})\eta_0 -\eta_0(\xi_{T^*Q})dF_h
        \\ & =  d(J_c)_\xi -c_\phi(\xi)dF_h = dJ_\xi,
        \end{split}
    \end{equation}
    that is, $J$ is a momentum map for $T^*\phi$. Moreover, we will see that $J$ is $Ad^*$-equivariant, i.e., 
    \begin{equation}\label{eq:J equivariant}
        J\circ T_g^*\phi = Ad_{g^{-1}}^* \circ J, \quad \text{for all} \quad g\in G.
    \end{equation}
   
    Indeed, from the fact that $m$ is multiplicative, for all $g\in G$ and $\xi\in\g$, we have that
    \begin{equation}\label{eq:c_phi o Ad = c_phi}
        c_\phi(Ad_{g^{-1}}(\xi)) = c_\phi(\xi),
    \end{equation}
    where $Ad$ is the adjoint action of $G$ on $\g$. Now, taking into account (\ref{eq:c_phi o Ad = c_phi}), that $J_c$ is $Ad^*$-equivariant and that $F_h$ is $G$-invariant, we conclude that
    \begin{align*}
        Ad_{g^{-1}}^*(J(\alpha_q))(\xi) & = J(\alpha_q)((Ad_{g^{-1}})(\xi)) = J_c(\alpha_q)(Ad_{g^{-1}}(\xi)) -F_h(\alpha_q)c_\phi(Ad_{g^{-1}}(\xi))
         \\ &  = J_c(T_g^*\phi(\alpha_q))(\xi) - F_h(T_g^*\phi(\alpha_q))c_\phi(\xi) = J(T_g^*\phi(\alpha_q))(\xi),
    \end{align*}
    for all $\alpha_q \in T^*Q$ and $\xi \in \g$. This proves (\ref{eq:J equivariant}), which ends this proof (see Remark \ref{remark: U = T*Q}).

\end{proof}

In what follows, we present a theorem that endows the presymplectic manifold $(V^*\Pi, \Omega_h)$ of corank $1$ with a presymplectic Hamiltonian action and an $Ad^*$-equivariant momentum map. 

\begin{proposition}\label{prop:J^V momentum map}
    The momentum map $J:T^*Q \to \g^*$ induces an $Ad^*$-equivariant momentum map $J^V: V^*\Pi \to \g^*$ characterized by
    \begin{equation}\label{eq: J=J^V o mu}
        J = J^V \circ \mu,
    \end{equation}
    which makes $(V^*\Pi, \Omega_h, V^*\phi, J^V)$ a symmetric presymplectic Hamiltonian system.
\end{proposition}

\begin{proof}
    If $\xi \in \g$ then, using that $R\in \text{ker } \omega_h$ (see Proposition \ref{prop:ker omega_h}) and (\ref{eq:i_xi omega_h = dJ_xi}), we have that
    \begin{equation}\label{eq:R(J_xi)=0}
        R(J_\xi) = dJ_\xi(R) = i_{\xi_{T^*Q}}\omega_h(R) = -i_R\omega_h(\xi_{T^*Q})=0.
    \end{equation}
    Therefore, the function $J_\xi: T^*Q\to \R$ is constant on the fibers of $\mu$, and thus,
    \begin{equation}\label{eq:J o psi = J}
        J\circ \psi_s = J, \quad \text{for all} \quad s\in \R.
    \end{equation}
    Consequently, $J:T^*Q \to \g^*$ induces a smooth map $J^V: V^*\Pi \to \g^*$ such that
    $$J = J^V \circ \mu.$$

On the other hand, using the fact that $h$ is $G$-equivariant, that $h^*(\omega_Q)=\Omega_h$ and that the cotangent lift of $\phi$ is a symplectic action,
we prove that $V^*\phi$ is a presymplectic action on $(V^*\Pi,\Omega_h)$, i.e., $$(V_g^*\phi)^*(\Omega_h)=\Omega_h, \quad \text{for all} \quad g\in G.$$

    Now, we will see that $J^V$ is a momentum map for $V^*\phi$ on $(V^*\Pi, \Omega_h)$, i.e., 
    \begin{equation}\label{eq: J^V momentum map}
        i_{\xi_{V^*\Pi}}\Omega_h = dJ^V_\xi, \quad \text{for all} \quad \xi \in\g.
    \end{equation}
    Indeed, taking into account (\ref{eq:mu G-equivariant inf gen}), (\ref{eq:i_xi omega_h = dJ_xi}), (\ref{eq: J=J^V o mu}) and Proposition \ref{prop:mu*(Omega_h)=omega_h Omega_h = h*(omega_h)}, we obtain that
    \begin{align*}
        (i_{\xi_{V^*\Pi}}\Omega_h)(\beta_q)(Z_{\beta_q}) & = \Omega_h(\mu(\alpha_q))(\xi_{V^*\Pi}(\mu(\alpha_q)),T_{\alpha_q}\mu(Y_{\alpha_q}))
        \\ & = \Omega_h(\mu(\alpha_q))(T_{\alpha_q}\mu(\xi_{T^*Q}(\alpha_q)),T_{\alpha_q}\mu(Y_{\alpha_q})) = \omega_h(\alpha_q)(\xi_{T^*Q}(\alpha_q),Y_{\alpha_q}) 
        \\ & = (dJ_\xi)(\alpha_q)(Y_{\alpha_q}) = \mu^*(dJ_\xi^V)(\alpha_q)(Y_{\alpha_q}) = dJ^V_\xi(\beta_q)(Z_{\beta_q}),
    \end{align*}
    for any $\xi \in \g$, $\beta_q = \mu(\alpha_q)\in V^*\Pi$ and $Z_{\beta_q} = T_{\alpha_q}\mu(Y_{\alpha_q})\in T_{\beta_q}(V^*\Pi)$. Therefore, (\ref{eq: J^V momentum map}) holds. In addition, since $J$ is $Ad^*$-equivariant, it follows that $J^V$ is also $Ad^*$-equivariant. Indeed, we have that
    $$Ad^*_{g^{-1}}(J^V(\beta_q)) = Ad^*_{g^{-1}}(J^V(\mu(\alpha_q)) = Ad^*_{g^{-1}}(J(\alpha_q))=J(T_g^*\phi(\alpha_q))$$
    and
    $$J^V(V^*_g\phi(\beta_q)) = J^V(V^*_g\phi(\mu(\alpha_q))) = J^V(\mu(T_g^*\phi(\alpha_q))) = J(T_g^*\phi(\alpha_q)),$$
    for any $\beta_q = \mu(\alpha_q)\in V^*\Pi$ and $g\in G$. We conclude that $Ad^*_{g^{-1}}\circ J^V = J^V\circ V^*_g\phi$, for all $g\in G$.

    Finally, taking into account Remark \ref{remark: U = T*Q}, we have that
    $$\ker \omega_h (\alpha_q) \cap T_{\alpha_q}(G\cdot \alpha_q) = \{ 0 \}, \quad \text{for all
    } \quad \alpha_q \in T^*Q.$$
    Therefore, from item \textit{(ii)} of Proposition \ref{prop:mu*(Omega_h)=omega_h Omega_h = h*(omega_h)}, the previous equality projects through $T_{\alpha_q} \mu$ to
    $$\ker \Omega_h (\beta_q) \cap T_{\beta_q}(G\cdot \beta_q) = \{ 0 \}, \quad \text{for all
    } \quad \beta_q \in V^*\Pi.$$
\end{proof}

\begin{remark}
    We consider the open $\R$-invariant subset $\mathcal{U} \subset T^*Q$ on which condition (\ref{eq: transversal condition T*Q}) holds. Then, $\mathcal{V} = \mu(\mathcal{U})$ is an open subset of $V^*\Pi$ such that $\mu: \mathcal{U} \to \mathcal{V}$ is a principal $\R$-bundle with the associated principal action given by the restriction to $\R \times \mathcal{U}$ of the principal action $\psi$. If the assumption of Remark \ref{remark: U = T*Q} is not satisfied, that is, if $\,\mathcal{U}\neq T^*Q$, then the symmetric presymplectic Hamiltonian systems constructed in Propositions \ref{prop:J momentum map} and \ref{prop:J^V momentum map} are replaced by $(\mathcal{U},\, \iota_{\mathcal{U}}^*(\omega_h),\, {T^*\phi}_{\vert G \times \mathcal{U}},\, J \circ \iota_{\mathcal{U}})$ and $(\mathcal{V},\, \iota_{\mathcal{V}}^*(\Omega_h),\, {V^*\phi}_{\vert G \times \mathcal{V}},\, J^V \circ \iota_{\mathcal{V}})$, respectively, where $\iota_{\mathcal{U}}: \mathcal{U} \hookrightarrow T^*Q$ and $\iota_{\mathcal{V}}: \mathcal{V} \hookrightarrow V^*\Pi$ denote the canonical inclusions.
\end{remark}

Next, we will prove a version of Noether Theorem in our setting for the time-dependent Hamiltonian dynamics $R_h$.

\begin{theorem}\label{theorem: noether}
    If $\xi \in \g$ and $\widehat{\xi}_Q: T^*Q \to \R$ is the fiberwise linear function induced by $\xi_Q\in \mathfrak{X}(Q)$, then:
    
    \begin{itemize}
        \item[(i)] The smooth function 
       $$\widehat{\xi}_Q - c_\phi(\xi) F_h: T^*Q \to \R$$
       is $\mu$-basic and the corresponding function on $V^*\Pi$ is $J_\xi^V$. 

       \item[(ii)] $J_\xi^V$ is a first integral of the time-dependent Hamiltonian dynamics $R_h$.
    \end{itemize}
    
\end{theorem}

\begin{proof}

    \begin{itemize}
        \item[\textit{(i)}] Using (\ref{eq:J_c momentum map}), (\ref{eq: J_xi}) and (\ref{eq: J=J^V o mu}), it follows that
    \begin{equation}\label{eq: widehat xi_Q - c_phi(xi) F_h = J_xi^V o mu}
        \widehat{\xi}_Q - c_\phi(\xi) F_h = (J_c)_\xi - c_\phi(\xi) F_h = J_\xi = J_\xi^V \circ \mu.
    \end{equation}

    \item[\textit{(ii)}] Let $\beta_q$ be a point in $V_q^*\Pi$, with $q\in Q$, and $\alpha_q \in T_q^*Q$ such that $\mu(\alpha_q) = \beta_q$. Then, using (\ref{eq: widehat xi_Q - c_phi(xi) F_h = J_xi^V o mu}) and item \textit{(ii)} in Proposition \ref{prop:X_F_h^omega_Q is mu projectable}, we have that
    \begin{align*}
        (dJ_\xi^V&(\beta_q))(R_h(\beta_q)) = (dJ_\xi^V(\beta_q))(T_{\alpha_q}\mu(X_{F_h}^{\omega_Q}(\alpha_q))) \\ & = (d(J_c)_\xi(\alpha_q) - c_\phi(\xi) dF_h(\alpha_q))(X_{F_h}^{\omega_Q}(\alpha_q)) = (d(J_c)_\xi(\alpha_q))(X_{F_h}^{\omega_Q}(\alpha_q)).
    \end{align*}

    But, since the Hamiltonian function $F_h\in \mathcal{C}^\infty(T^*Q)$ is $G$-invariant and $J_c: T^*Q \to \g^*$ is a momentum map for the symplectic action $T^*\phi$ on $(T^*Q, \omega_Q)$ then, from Theorem 4.2.2 in \cite{AM}, we deduce that $(J_c)_\xi$ is a first integral of $X_{F_h}^{\omega_Q}$. So, we conclude that 
    $$(dJ_\xi^V(\beta_q))(R_h(\beta_q)) = 0.$$
    \end{itemize}
    
\end{proof}

\begin{remark}
    Note that the condition (\ref{eq: ker omega cap T_x(G x) = 0}) is not used in the proof of Theorem \ref{theorem: noether}. So, $J_\xi^V$ is a first integral of $R_h$ as a vector field defined on the full space $V^*\Pi$ (it is not necessary to restrict to the open subset $\mathcal{V}$ of $V^*\Pi$).
\end{remark}

\medskip
The following example, which continues Example~\ref{ex: actions T*phi V*phi}, presents a physical system that illustrates the theory.

\begin{example}[$N$-dimensional harmonic oscillator seen by an observer that moves with a constant velocity]\label{ex: momentum map}

Following Example \ref{ex:harmonic oscillator}, we know that the action $T^*\phi$, given in (\ref{ex: action T*phi}), is free and the Hamiltonian vector field, given in (\ref{ex: X_F_h^omega_Q}), satisfy $$X_{F_h}^{\omega_Q}(q^i,t,p_i,p)\in \langle 1_{T^*Q}(q^i,t,p_i,p)\rangle = \langle \frac{\partial}{\partial t} - v\frac{\partial}{\partial q^1}\rangle$$ if and only if $(q^i,t,p_i,p) \in \mathcal{C}$, where $1_{T^*Q}$ is the infinitesimal generator of the action $T^*\phi$ associated to $1$, given in (\ref{ex:generadores infinitesimales}), and $\mathcal{C}$ is the closed submanifold of $T^*Q$ of the relative equilibria of $X_{F_h}^{\omega_Q}$ given by
    $$\mathcal{C} = \{ (q^i,t,p_i,p)\in T^*Q \,\mid\, q^1=-vt, \, p_1=-mv, \,q^\alpha = p_\alpha =0, \,t,p \in \R \}.$$

    Now, consider the open subset $\mathcal{U} = T^*Q\setminus \mathcal{C}$ where
    $$\ker \omega_h(\alpha_q) \cap T_{\alpha_q}(G \cdot \alpha_q) = \{ 0 \}, \quad \text{for all} \quad \alpha_q \in \mathcal{U}.$$

    On the other hand, we have that the action $V^*\phi$, given in (\ref{ex: action V*phi}), is free. So, we can consider the open subset $\mathcal{V} = \mu(\mathcal{U})= V^*\Pi\setminus \mathcal{C}^V$, where
    $$\mathcal{C}^V = \{ (q^i,t,p_i)\in V^*\Pi \,\mid\, q^1=-vt, \, p_1=-mv, \,q^\alpha = p_\alpha =0, \,t \in \R \},$$
    and $\mu : \mathcal{U} \to \mathcal{V}$ is the corresponding $\R$-principal bundle. 
 
Moreover, using (\ref{ex: F_h}), (\ref{ex:generadores infinitesimales}), (\ref{eq:J momentum map}) and (\ref{eq: J=J^V o mu}), we have that the momentum maps $J: T^*Q \to \R$ and $J^V: V^*\Pi \to \R$ are given by
\begin{equation}\label{ex:J}
    J(q^i,t,p_i,p) = \frac{m v^2}{2} - \frac{1}{2m}\left((p_1+mv)^2 + p_\alpha^2 \right) - \frac{1}{2}m\Omega^2\left((q^1+vt)^2+(q^\alpha)^2\right),
\end{equation}
and
\begin{equation}\label{ex: J^V}
    J^V(q^i,t,p_i) = \frac{m v^2}{2} - \frac{1}{2m}\left((p_1+mv)^2 + p_\alpha^2 \right) - \frac{1}{2}m\Omega^2\left((q^1+vt)^2+(q^\alpha)^2\right).
\end{equation}

As we know (see Theorem \ref{theorem: noether}), we have that $J^V$ is a first integral of the time-dependent Hamiltonian dynamics $R_h$. In fact, this may be proved directly using (\ref{ex: R_h}) and (\ref{ex: J^V}). Note that $J^V$ is a time-dependent first integral of $R_h$. We remark that when one may apply the reduction process proposed in \cite{ALBERT1989627}, the resulting first integrals cannot depend explicitly on time (see \cite{GUTIERREZSAGREDO}). This is an advantage of our method.

Finally, as a consequence of the results in this section, we have the presymplectic Hamiltonian systems of corank $2$ and $1$, 
$$(\mathcal{U},\, \iota_{\mathcal{U}}^*(\omega_h),\, {T^*\phi}_{\vert G \times \mathcal{U}},\, J \circ \iota_{\mathcal{U}}) \quad \text{and} \quad (\mathcal{V},\, \iota_{\mathcal{V}}^*(\Omega_h),\, {V^*\phi}_{\vert G \times \mathcal{V}},\, J^V \circ \iota_{\mathcal{V}}),$$ respectively, where $\iota_{\mathcal{U}}: \mathcal{U} \hookrightarrow T^*Q$ and $\iota_{\mathcal{V}}: \mathcal{V} \hookrightarrow V^*\Pi$ denote the canonical inclusions.

\end{example}


\section{Reduction process}\label{sec: reduction}

From the previous sections, we know that the spaces $(T^*Q, \omega_h)$ and $(V^*\Pi, \Omega_h)$ are presymplectic manifolds of corank $2$ and $1$ and the total and base space of a principal $\R$-bundle, respectively. The goal of this section is to apply presymplectic reduction theories to these spaces in order to obtain a reduced principal $\R$-bundle with total space a presymplectic manifold of corank $2$ and with base space a presymplectic manifold of corank $1$.

\begin{remark}\label{remark: T*phi free and proper}
    From \cite[Example 2.3.6]{ORTEGA}, we know that if the action $\phi: G \times Q \to Q$ is free and proper, so is the cotangent lift $T^*\phi: G \times T^*Q \to T^*Q$. Analogously, one may prove that the action $V^*\phi: G \times V^*\Pi \to V^*\Pi$ is also free and proper.  
\end{remark}

To develop the reduction process, we need a regular value $\nu \in \mathfrak{g}^*$ for each momentum map. This is a consequence of the freeness of the respective actions. Following the ideas of \cite[Proposition 4.7]{GUTIERREZSAGREDO} and (\ref{eq: ker omega cap T_x(G x) = 0}), we obtain the following result.

\begin{proposition}\label{prop: regular value}
    Let $\phi: G \times M \to M$ be a presymplectic Hamiltonian action of a Lie group $G$ on a presymplectic manifold $(M, \omega)$ with momentum map $J: M \to \g^*$. Suppose that $\phi$ is free. Then, $J$ is a submersion and, therefore, for all $\nu \in \g^*$, $J^{-1}(\nu)$ is a regular submanifold of dimension $\dim M - \dim G$.
\end{proposition}

\begin{proof}
    We want to see that the momentum map $J: M \to \g^*$ is a submersion. 
    This is equivalent to proving that the map $T_x^*J: T_{J(x)}^*\g^* \simeq \g \to T_x^*M$ is injective for all $x\in M$. In fact, if $\xi \in \ker T_x^*J$, using (\ref{eq:i_xi_M omega = dJ_xi}), it follows that
    \[0 = T_x^*J(\xi)= dJ_\xi(x) = (i_{\xi_M}\omega)(x).\]
    Thus, $\xi_M(x) \in \ker \omega(x)$. In addition, from (\ref{eq: ker omega cap T_x(G x) = 0}) and the fact that $\xi_M(x) \in T_x(G\cdot x)$, we conclude that $\xi_M(x) =0$. Since $\phi$ is free, it is infinitesimally free (i.e., if $\xi_M(x) = 0$ then $\xi=0$), thus we have $\xi =0$. Therefore, $\ker T_x^*J = \{0\}$.
    
\end{proof}

\subsection{Presymplectic reduction}

In this subsection, we will apply the presymplectic reduction theory given in \cite{Echevarria} to reduce the presymplectic manifolds $(T^*Q, \omega_h)$ and $(V^*\Pi, \Omega_h)$ to obtain new reduced presymplectic manifolds.

\subsubsection{Presymplectic reduction of corank $2$}
In the first case, we will reduce the presymplectic manifold $(T^*Q, \omega_h)$. We use the presymplectic reduction theory given in \cite{Echevarria} and we can apply the following theorem.

\begin{theorem}\cite{Echevarria}\label{theorem: presymplectic reduction}
    Let $\Phi: G \times M \to M$ be a presymplectic Hamiltonian action of a Lie group $G$ on the presymplectic manifold $(M,\omega)$ with $Ad^*$-equivariant momentum map $J:M \to \g^*$. Moreover, suppose that $\Phi$ is infinitesimally free and proper. Then, if $\nu \in \g^*$, $\nu$ is a regular value of the momentum map $J$, $J^{-1}(\nu)$ is a submanifold of $M$ on which acts the isotropy subgroup $G_\nu$ (with respect to the coadjoint action of $G$ on $\g^*$) and the orbit space $J^{-1}(\nu)/G_\nu$ is a smooth manifold. In addition, if $\pi_\nu: J^{-1}(\nu) \to J^{-1}(\nu)/G_\nu$ denotes the canonical projection, there is a presymplectic closed $2$-form $\omega_\nu$ on $J^{-1}(\nu)/G_\nu$ such that 
    $$\pi_\nu^*(\omega_\nu)=\iota_\nu^*\omega$$
    where $\iota_\nu : J^{-1}(\nu) \hookrightarrow M$ is the inclusion map.
\end{theorem}

\begin{remark}\label{remark: infinitesimally free}
    If the action $\phi$ is free then it is infinitesimally free and the previous theorem works. On the other hand, instead of requiring the action $T^*\phi$ to be proper, it is sufficient to assume that the quotient space $J^{-1}(\nu)/G_\nu$ is a smooth manifold and that the canonical projection $\pi_\nu: J^{-1}(\nu) \to J^{-1}(\nu)/G_\nu$ is a submersion.
\end{remark}

Now, we can reduce the presymplectic manifold $(T^*Q, \omega_h)$. The necessary ingredients for the reduction are:
\begin{itemize}
    \item The cotangent bundle $T^*Q$ with the presymplectic structure $\omega_h$, given by (see \eqref{eq:def omega_h})
    $$\omega_h = \omega_Q + dF_h \wedge \eta_0,$$
    with $$\ker\omega_h = \langle R, X_{F_h}^{\omega_Q}\rangle,$$
    where $X_{F_h}^{\omega_Q}\in \mathfrak{X}(T^*Q)$ is the Hamiltonian vector field of $F_h: T^*Q \to \R$ with respect to $\omega_Q$ and $R$ denotes the infinitesimal generator of the action $\psi$ associated to the principal bundle $\mu: T^*Q \to V^*\Pi$ (see Proposition \ref{prop:ker omega_h}). Note that $h: V^*\Pi \to T^*Q$ is a Hamiltonian section satisfying Assumption \ref{assum: h equivariant} and that $F_h$ is the corresponding $G$-invariant Hamiltonian function.
    \item The cotangent lift action $T^*\phi: G \times T^*Q \to T^*Q$ given by (\ref{eq: def T*phi}), where $\phi : G \times Q \to Q$ is a free action.
    \item The $Ad^*$-equivariant momentum map $J:T^*Q \to \g^*$, given by (\ref{eq:J momentum map})  
    $$J(\alpha_q)(\xi) = \alpha_q(\xi_Q(q)) - F_h(\alpha_q)c_\phi(\xi), \quad \text{for all} \quad \alpha_q\in T_q^*Q\quad \text{and} \quad \xi\in \g.$$
\end{itemize}

From Remark \ref{remark: T*phi free and proper}, we have that $T^*\phi$ is free and then it is infinitesimally free (see Remark \ref{remark: infinitesimally free}). As a consequence of Theorem \ref{theorem: presymplectic reduction}, for a fixed $\nu \in \g^*$ (which is a regular value of $J$ by Proposition \ref{prop: regular value}), we obtain a presymplectic structure on the quotient space $J^{-1}(\nu)/ G_\nu$, as stated in the following result.

\begin{corollary}\label{corollary: reduction corank 2}
    Let $(T^*Q, \omega_h, T^*\phi, J)$ be the symmetric presymplectic Hamiltonian system given in Proposition \ref{prop:J momentum map}. Moreover, let $\nu\in \g^*$ such that $J^{-1}(\nu) \neq \emptyset$ and suppose that $\phi$ is free and the space of orbits $(T^*Q)_\nu := J^{-1}(\nu)/G_\nu$ is a manifold such that the canonical projection $\pi_\nu: J^{-1}(\nu) \to J^{-1}(\nu)/G_\nu$ is a submersion. Then, there is a presymplectic closed $2$-form $\omega_h^\nu$ on $(T^*Q)_\nu$ such that
    \begin{equation}\label{eq:(pi_nu)*omega_h^nu = (i_nu)*omega_h}
        \pi_\nu^*(\omega_h^\nu)=\iota_\nu^*(\omega_h),
    \end{equation}
    where $\iota_\nu: J^{-1}(\nu) \hookrightarrow T^*Q$ is the canonical inclusion.
\end{corollary}

In what follows, we will study the dynamics on the reduced presymplectic system and the corank of the reduced presymplectic structures.

\begin{proposition}\label{prop: reduced X_F_h^omega_Q}
    In the hypothesis of Corollary \ref{corollary: reduction corank 2}, the Hamiltonian vector field $X_{F_h}^{\omega_Q}$ restricts to $J^{-1}(\nu)$, and its restriction is $\pi_\nu$-projectable on a vector field $X_\nu$ on $(T^*Q)_\nu$, that is, 
    \begin{equation}\label{eq: X_F_h^omega_Q is pi_nu projectable over widetilde X_nu}
        T_{\alpha_q}\pi_\nu(X_{F_h}^{\omega_Q}(\alpha_q)) = X_\nu(\pi_\nu(\alpha_q)),\quad \text{for all} \quad \alpha_q\in J^{-1}(\nu).
    \end{equation}
\end{proposition}

\begin{proof}
    First, we prove that $X_{F_h}^{\omega_Q}$ restricts to $J^{-1}(\nu)$. Indeed, let $\xi\in\g$, $\alpha_q \in J^{-1}(\nu)$ and $\widehat{\xi}: \g^*\to \R$ be the linear function induced by $\xi$. Taking into account that $F_h$ is $G$-invariant, it follows that
    \begin{align*}
        T_{\alpha_q}J(X_{F_h}^{\omega_Q}(\alpha_q))(\widehat{\xi}) & 
        = dJ_\xi(\alpha_q)(X_{F_h}^{\omega_Q}(\alpha_q))
         = (i_{\xi_{T^*Q}}\omega_Q)(\alpha_q)(X_{F_h}^{\omega_Q}(\alpha_q)) 
        \\ & = -dF_h(\xi_{T^*Q})(\alpha_q) = -\xi_{T^*Q}(F_h)(\alpha_q) = 0.
    \end{align*}
    Therefore, we conclude that $X_{F_h}^{\omega_Q}$ is tangent to $J^{-1}(\nu)$.

    Now, we prove that ${X_{F_h}^{\omega_Q}}_{|J^{-1}(\nu)}$ is $\pi_\nu$-projectable. In fact, using that $\omega_Q$ and $F_h$ are $G$-invariant, it is easy to prove that $X_{F_h}^{\omega_Q}$ is also $G$-invariant. This implies that ${X_{F_h}^{\omega_Q}}_{\vert J^{-1}(\nu)}$ is $G_\nu$-invariant and, therefore, it is $\pi_\nu$-projectable. Thus, there exists a vector field $X_\nu$ on $(T^*Q)_\nu$ such that 
    $$T_{\alpha_q}\pi_\nu(X_{F_h}^{\omega_Q}(\alpha_q)) = X_\nu(\pi_\nu(\alpha_q)),\quad \text{for all} \quad \alpha_q\in J^{-1}(\nu).$$
\end{proof}

Analogously, the infinitesimal generator $R$ of $\psi$ is $\pi_\nu$-projectable on a vector field on $(T^*Q)_\nu$, as will be established in the following result.

\begin{proposition}\label{prop: reduced psi}
    In the hypothesis of Corollary \ref{corollary: reduction corank 2}, the principal action $\psi: \R \times T^*Q \to T^*Q$ restricts to an action of $\R$ on $J^{-1}(\nu)$ and it induces an action $\psi_\nu: \R \times (T^*Q)_\nu \to (T^*Q)_\nu$ of $\R$ on $(T^*Q)_\nu$ characterized by
    \begin{equation}\label{eq:(psi_nu)_s o pi_nu = pi_nu o psi_s}
        (\psi_\nu)_s \circ \pi_\nu = \pi_\nu \circ {\psi_s}_{\vert J^{-1}(\nu)}, \quad \text{for all} \quad s\in \R.
    \end{equation}
    Moreover, the infinitesimal generator $R$ of $\psi$ restricts to $J^{-1}(\nu)$, and its restriction is $\pi_\nu$-projectable on the infinitesimal generator $R_\nu$ of $\psi_\nu$, that is,
    \begin{equation}\label{eq: R is pi_nu projectable over R_nu}
        T_{\alpha_q}\pi_\nu(R(\alpha_q)) = R_\nu(\pi_\nu(\alpha_q)),\quad \text{for all} \quad \alpha_q\in J^{-1}(\nu).
    \end{equation}
\end{proposition}

\begin{proof}
    From (\ref{eq:J o psi = J}), we see that the principal action $\psi: \R \times T^*Q \to T^*Q$ restricts to an action of $\R$ on $J^{-1}(\nu)$. Consequently, from Proposition \ref{prop: T*phi and psi commute}, $\psi$ induces an action $\psi_\nu : \R \times (T^*Q)_\nu \to (T^*Q)_\nu$ of $\R$ on $(T^*Q)_\nu$ characterized by 
    $$(\psi_\nu)_s \circ \pi_\nu = \pi_\nu \circ {\psi_s}_{\vert J^{-1}(\nu)}, \quad \text{for all} \quad s\in \R.$$
        
    On the other hand, from (\ref{eq:R(J_xi)=0}), it follows that the restriction to $J^{-1}(\nu)$ of $R$ is tangent to $J^{-1}(\nu)$. Moreover, $R_{|J^{-1}(\nu)}$ is the infinitesimal generator of the action $\psi:\R \times J^{-1}(\nu) \to J^{-1}(\nu)$, and, due to (\ref{eq:(psi_nu)_s o pi_nu = pi_nu o psi_s}), it is $\pi_\nu$-projectable over the infinitesimal generator $R_\nu$ of $\psi_\nu$.
\end{proof}

The reduced vector fields $X_\nu$ and $R_\nu$ of Propositions \ref{prop: reduced X_F_h^omega_Q} and \ref{prop: reduced psi}, respectively, generate the kernel of the reduced presymplectic structure $\omega_h^\nu$ on $(T^*Q)_\nu$. To this end, we present the following lemma, whose proof follows from similar arguments to those used in Lemma 4.3.2 of \cite{AM} (see also \cite{MARSDEN}).

\begin{lemma}\label{lemma MARSDEN}
    Let $\phi: G \times M \to M$ be a presymplectic Hamiltonian action of a Lie group $G$ on the presymplectic manifold $(M,\omega)$ with $Ad^*$-equivariant momentum map $J:M \to \g^*$. Let $\nu \in \g^*$ a regular value of $J$ and $x\in J^{-1}(\nu)$, then
    \begin{itemize}
        \item[(i)] $T_x(G_\nu \cdot x) = T_x(G\cdot x) \cap T_x(J^{-1}(\nu))$.
        \item[(ii)] $T_x(J^{-1}(\nu))=(T_x(G\cdot x))^{\perp,\omega}$, where $(T_x(G\cdot x))^{\perp,\omega}$ denotes the $\omega$-orthogonal complement of $T_x(G\cdot x)$.
    \end{itemize}
\end{lemma}

\begin{remark}
    The definition of the $\omega$-orthogonal complement in Lemma \ref{lemma MARSDEN} is a particular case of the general notion of the $\Phi$-orthogonal complement for a $2$-form $\Phi: V\times V \to \R$ over a vector space $V$ of finite dimension. In fact, if $W$ is a subspace of $V$ then 
    $$W^{\perp, \Phi}= \{v \in V \,\mid\, \Phi(v,w) = 0, \, \forall w\in W \}.$$
    Note that one may prove that (see, for instance, \cite[page 231]{MR1021489}),
    \begin{equation}\label{eq: ortho-com-W}
        \dim W^{\perp, \Phi} = \dim V - \dim W + \dim (W\cap \ker \Phi).
    \end{equation}
    So, we deduce that
    \begin{equation}\label{eq: ortho-com-ortho-com}
        (W^{\perp, \Phi})^{\perp, \Phi} = W + \ker \Phi.
    \end{equation}
    In fact, it is easy to see that 
    $$W + \ker \Phi \subseteq (W^{\perp, \Phi})^{\perp, \Phi}$$
    and, using (\ref{eq: ortho-com-W}) and the fact that $\ker\Phi \subseteq W^{\perp, \Phi}$, we obtain that
    $$\dim (W^{\perp, \Phi})^{\perp, \Phi} = \dim (W + \ker \Phi).$$
\end{remark}

Now, we prove that the $2$-form $\omega_h^\nu$ on $(T^*Q)_\nu$ is a presymplectic structure of corank $2$.

\begin{theorem}\label{theorem: reduced presymplectic structure corank 2}
    Let $\omega_h^\nu$ be the reduced presymplectic structure on the reduced presymplectic manifold $(T^*Q)_\nu = J^{-1}(\nu)/G_\nu$ deduced in Corollary \ref{corollary: reduction corank 2}. Then, 
    \begin{equation}\label{eq: ker omega_h^nu =  R_nu, X_nu}
        \ker \omega_h^\nu = \langle R_\nu, X_\nu\rangle.
    \end{equation}
    Thus, $((T^*Q)_\nu,\omega_h^\nu)$ is a presymplectic manifold of corank $2$.
\end{theorem}
    
\begin{proof}
    First, we will prove that $R_\nu$ and $X_\nu$ are linearly independent at every point of $(T^*Q)_\nu$. Let $\alpha_q \in J^{-1}(\nu)$ and $\lambda_1, \lambda_2 \in \R$ such that 
    $$\lambda_1 R_\nu(\pi_\nu(\alpha_q)) + \lambda_2 X_\nu(\pi_\nu(\alpha_q)) = 0.$$

    Then, from (\ref{eq: X_F_h^omega_Q is pi_nu projectable over widetilde X_nu}) and (\ref{eq: R is pi_nu projectable over R_nu}), it follows that
    $$\lambda_1 R(\alpha_q) + \lambda_2 X_{F_h}^{\omega_Q}(\alpha_q) \in T_{\alpha_q} (G_\nu \cdot \alpha_q) \subseteq T_{\alpha_q} (G \cdot \alpha_q).$$

    So, using Proposition \ref{prop:ker omega_h}, we have that 
    $$\lambda_1 R(\alpha_q) + \lambda_2 X_{F_h}^{\omega_Q}(\alpha_q) \in \ker \omega_h(\alpha_q) \cap T_{\alpha_q} (G \cdot \alpha_q),$$
    and, since $\ker \omega_h(\alpha_q) \cap T_{\alpha_q} (G \cdot \alpha_q) = \{0\}$ and the vector fields $R(\alpha_q)$ and $X_{F_h}^{\omega_Q}(\alpha_q)$ are linearly independent, we deduce that $\lambda_1 = \lambda_2 =0$. This proves that $R_\nu(\pi_\nu(\alpha_q))$ and $X_\nu(\pi_\nu(\alpha_q))$ are linearly independent. 
    
    Now, we will see that $R_\nu(\pi_\nu(\alpha_q)), X_\nu(\pi_\nu(\alpha_q)) \in \ker \omega_h^\nu(\pi_\nu(\alpha_q))$. In fact, if $\mathcal{Z}_{\alpha_q} \in T_{\alpha_q}(J^{-1}(\nu))$ then, using (\ref{eq:(pi_nu)*omega_h^nu = (i_nu)*omega_h}), (\ref{eq: X_F_h^omega_Q is pi_nu projectable over widetilde X_nu}) and Proposition \ref{prop:ker omega_h}, it follows that
    \begin{align*}
        (i_{X_\nu(\pi_\nu(\alpha_q))}\omega_h^\nu (\pi_\nu(\alpha_q)))(T_{\alpha_q}\pi_\nu(\mathcal{Z}_{\alpha_q})) & =  ((\pi_\nu)^*\omega_h^\nu)(\alpha_q)(X_{F_h}^{\omega_q}(\alpha_q),\mathcal{Z}_{\alpha_q}) 
        \\ & = (\iota_\nu^*\omega_h)(\alpha_q)(X_{F_h}^{\omega_q}(\alpha_q),\mathcal{Z}_{\alpha_q}) =0.
    \end{align*}
    Thus, $X_\nu(\pi_\nu(\alpha_q)) \in \ker\, \omega_h^\nu(\pi_\nu(\alpha_q))$. Analogously, from (\ref{eq:(pi_nu)*omega_h^nu = (i_nu)*omega_h}), (\ref{eq: R is pi_nu projectable over R_nu}) and Proposition \ref{prop:ker omega_h}, it is straightforward to see that $R_\nu(\pi_\nu(\alpha_q)) \in \ker\, \omega_h^\nu(\pi_\nu(\alpha_q))$.

    Finally, we will prove that $\ker \omega_h^\nu = \langle R_\nu, X_\nu\rangle$. Let $\alpha_q \in J^{-1}(\nu)$ and $\mathcal{X}_{\pi_\nu(\alpha_q)} \in T_{\pi_\nu(\alpha_q)}(J^{-1}(\nu)/G_\nu)$ such that $\mathcal{X}_{\pi_\nu(\alpha_q)} \in \ker \omega_h^\nu(\pi_\nu(\alpha_q))$. Then, there exists $\mathcal{Y}_{\alpha_q} \in T_{\alpha_q}(J^{-1}(\nu))$ such that $T_{\alpha_q}\pi_\nu(\mathcal{Y}_{\alpha_q}) = \mathcal{X}_{\pi_\nu(\alpha_q)}$. Thus, to prove (\ref{eq: ker omega_h^nu =  R_nu, X_nu}) it is enough to see that 
    \begin{equation}\label{eq: mathcal Y_alpha_q}
        \mathcal{Y}_{\alpha_q} \in \langle R(\alpha_q), {X_{F_h}^{\omega_Q}}(\alpha_q)\rangle + T_{\alpha_q}(G_\nu \cdot \alpha_q).
    \end{equation}
    
    Note that 
    \begin{equation}\label{eq: mathcal{Y}_{alpha_q} in ker (i_nu^*omega_h(alpha_q))}
        \mathcal{Y}_{\alpha_q} \in \ker (\iota_\nu^*\omega_h(\alpha_q)).
    \end{equation}
    In fact, if $\mathcal{Z}_{\alpha_q} \in T_{\alpha_q}(J^{-1}(\nu))$, using the fact that $\mathcal{X}_{\pi_\nu(\alpha_q)} \in \ker \omega_h^\nu(\pi_\nu(\alpha_q))$ and (\ref{eq:(pi_nu)*omega_h^nu = (i_nu)*omega_h}), it follows that 
    \begin{align*}
        (\iota_\nu^*\omega_h)(\alpha_q)(\mathcal{Y}_{\alpha_q},\mathcal{Z}_{\alpha_q}) & = (\pi_\nu^*(\omega_h^\nu))(\alpha_q)(\mathcal{Y}_{\alpha_q},\mathcal{Z}_{\alpha_q}) = \omega_h^\nu(\pi_\nu(\alpha_q))(T_{\alpha_q}\pi_\nu(\mathcal{Y}_{\alpha_q}),T_{\alpha_q}\pi_\nu(\mathcal{Z}_{\alpha_q}))
        \\ & = \omega_h^\nu(\pi_\nu(\alpha_q))(\mathcal{X}_{\pi_\nu(\alpha_q)},T_{\alpha_q}\pi_\nu(\mathcal{Z}_{\alpha_q})) = 0.
    \end{align*}

    Moreover, since 
    $$\ker(\iota_\nu^*\omega_h(\alpha_q)) = T_{\alpha_q} J^{-1}(\nu) \cap (T_{\alpha_q} J^{-1}(\nu))^{\perp, \omega_h},$$
    then, using (\ref{eq: ortho-com-ortho-com}), Proposition \ref{prop:ker omega_h} and Lemma \ref{lemma MARSDEN}, we have that
    \begin{equation}\label{eq: ker (iota_nu*(omega_h)(alpha_q))}
        \begin{split}
            \ker (\iota_\nu^*\omega_h(\alpha_q)) & = T_{\alpha_q}(J^{-1}(\nu)) \cap (((T_{\alpha_q}(G\cdot \alpha_q))^{\perp,\omega_h})^{\perp,\omega_h}) 
        \\ & = T_{\alpha_q}(J^{-1}(\nu)) \cap (\ker (\omega_h(\alpha_q)) + T_{\alpha_q}(G\cdot \alpha_q))  \\ & = T_{\alpha_q}(J^{-1}(\nu)) \cap (\langle R(\alpha_q), X_{F_h}^{\omega_Q}(\alpha_q)\rangle + T_{\alpha_q}(G\cdot \alpha_q)).
        \end{split}
    \end{equation}

    On the other hand, $R(\alpha_q),{X_{F_h}^{\omega_Q}}(\alpha_q) \in T_{\alpha_q}(J^{-1}(\nu))$ (see Propositions \ref{prop: reduced X_F_h^omega_Q} and \ref{prop: reduced psi}). Therefore, from item \textit{(i)} in Lemma \ref{lemma MARSDEN} and (\ref{eq: ker (iota_nu*(omega_h)(alpha_q))}), we deduce that
    $$\ker (\iota_\nu^*\omega_h(\alpha_q)) = \langle R(\alpha_q), {X_{F_h}^{\omega_Q}}(\alpha_q)\rangle + T_{\alpha_q}(G_\nu \cdot \alpha_q).$$

    This, using (\ref{eq: mathcal{Y}_{alpha_q} in ker (i_nu^*omega_h(alpha_q))}), proves that (\ref{eq: mathcal Y_alpha_q}) holds.
    
\end{proof}

The following diagram summarizes this reduction:
$$\xymatrix{
\left(T^*Q,\,\omega_h,\, X_{F_h}^{\omega_Q}\right)  \ar@(u,ur)[]^{(\psi, R)}  & (J^{-1}(\nu),\, \iota_\nu^*\omega_h,\, {X_{F_h}^{\omega_Q}}_{\vert J^{-1}(\nu)})  \ar[r]^{\hspace{20pt}\pi_\nu} \ar@{_{(}->}[l]_{\hspace{-20pt}\iota_\nu}  \ar@(u,ur)[]^{(\psi, R_{{\vert J^{-1}(\nu)}})} & \left((T^*Q)_\nu,\, \omega_h^\nu,\, X_\nu\right) \ar@(u,ur)[]^{(\psi_\nu, R_\nu)}
}$$

\begin{remark}\label{remark: reduction U}
    Note that Theorem \ref{theorem: reduced presymplectic structure corank 2} holds under the conditions of Remark \ref{remark: U = T*Q}. If this assumption is not satisfied, that is, if $\mathcal{U}\neq T^*Q$, then, the reduced presymplectic manifold $((T^*Q)_\nu, \omega_h^\nu)$ of corank $2$ is replaced by $(\mathcal{U}_\nu, \iota_{\mathcal{U}_\nu}^*({\omega_h^\nu}))$, where 
$$\mathcal{U}_\nu = (J \circ \iota_{\mathcal{U}})^{-1}(\nu)/ G_\nu$$
is a quotient manifold since $(J \circ \iota_{\mathcal{U}})^{-1}(\nu)$ is a $\R$-invariant and $G_\nu$-invariant open subset of $J^{-1}(\nu)$ and $\iota_{\mathcal{U}_\nu} : \mathcal{U}_\nu \hookrightarrow (T^*Q)_\nu$ is the canonical inclusion. In fact, $\mathcal{U}_\nu$ can be identified with an open subset of $(T^*Q)_\nu = J^{-1}(\nu)/G_\nu$ and we can consider the restrictions to $\mathcal{U}_\nu$ of the vector field $X_\nu$, the principal action $\psi_\nu: \R \times \mathcal{U}_\nu \to \mathcal{U}_\nu$ and the infinitesimal generator $R_\nu$ of $\psi_\nu$.
\end{remark}

Now, we continue with Example \ref{ex: momentum map}.

\begin{example}[$N$-dimensional harmonic oscillator seen by an observer that moves with a constant velocity]\label{ex: presympl reduc}
    Continuing again with Example \ref{ex:harmonic oscillator}, from the momentum map $J: T^*Q \to \R$, given in (\ref{ex:J}), we obtain that
    \begin{itemize}
        \item For any $\nu \geq \frac{mv^2}{2}$, $(J \circ \iota_{\mathcal{U}})^{-1}(\nu) = \emptyset$.
        \item For any $\nu < \frac{mv^2}{2}$, $(J \circ \iota_{\mathcal{U}})^{-1}(\nu)$ is the $2N+1$-dimensional submanifold given by
    $$(J \circ \iota_{\mathcal{U}})^{-1}(\nu)= \left\{ \begin{array}{ll} (q^i,t,p_i,p)\in \mathcal{U} \,\vert\, & \frac{1}{2m}\left((p_1+mv)^2 + p_\alpha^2 \right) + \frac{1}{2}m\Omega^2\left((q^1+vt)^2+(q^\alpha)^2\right) \\[6pt]  & = \frac{m v^2}{2}-\nu \end{array} \right\}$$
    \end{itemize}

    To describe the quotient $(J \circ \iota_{\mathcal{U}})^{-1}(\nu) / \R$ for any $\nu < \frac{mv^2}{2}$, we define $\mathcal{N}=(J \circ \iota_{\mathcal{U}})^{-1}(\nu) \cap \varphi^{-1}(0)$, where $\varphi$ is the first projection, i.e., $\varphi(q^i,t,p_i,p)= q^1$. In fact,
    $$\mathcal{N}= \left\{ \begin{array}{ll} & (q^\alpha,t,p_1,p_\alpha,p)\in \R^{2N+1}\setminus \{(0,0,-mv,0,p)\} \,\vert\, \\[6pt] & \frac{1}{2m}\left((p_1+mv)^2 + p_\alpha^2 \right) + \frac{1}{2}m\Omega^2\left((vt)^2+(q^\alpha)^2\right) = \frac{m v^2}{2}-\nu \end{array} \right\}$$
    
    Deriving the constraint that defines $\mathcal{N}$, we deduce that the points of this submanifold satisfy
    \begin{equation}\label{ex:iota*(vdp_1+ 1/m p_idp_i + mOmega^2(v^2tdt + q^alpha dq^alpha))=0}
        \iota^*(vdp_1+ \frac{1}{m}p_idp_i + m\Omega^2(v^2tdt + q^\alpha dq^\alpha))=0,
    \end{equation}
    where $\iota : \mathcal{N} \hookrightarrow (J \circ \iota_{\mathcal{U}})^{-1}(\nu)$ is the inclusion map.

    From the integral curves of $1_{T^*Q} = \frac{\partial}{\partial t} - v\frac{\partial}{\partial q^1}$, we define the map 
    \begin{equation}\label{ex:dif Phi}
        \begin{array}{rccl}
        \Phi \colon & \R \times \mathcal{N} & \longrightarrow & (J \circ \iota_{\mathcal{U}})^{-1}(\nu)\\
        &(\hat{t},(q^\alpha,t,p_i,p)) & \longmapsto & (-v\hat{t},q^\alpha, t+\hat{t},p_i,p)
    \end{array}
    \end{equation}
    which is a diffeomorphism, with inverse map given by
    $$
    \begin{array}{rccl}
        \Phi^{-1} \colon & (J \circ \iota_{\mathcal{U}})^{-1}(\nu) & \longrightarrow & \R \times \mathcal{N}\\
        &(q^1,q^\alpha,t,p_i,p) & \longmapsto & (-\frac{1}{v}q^1,q^\alpha, t+\frac{1}{v}q^1,p_i,p).
    \end{array}
    $$

    Now, consider the action $\vartheta: \R \times (\R \times \mathcal{N}) \to \R \times \mathcal{N}$ defined by $$\vartheta(s,(\hat{t},q^\alpha,t,p_i,p)) = (s+\hat{t},q^\alpha,t,p_i,p).$$ We have that the diffeomorphism $\Phi$ is equivariant with respect to the actions $\vartheta$ and $T^*\phi_{\vert (J \circ \iota_{\mathcal{U}})^{-1}(\nu)}$, i.e., the following diagram is commutative
    $$\xymatrix{ \R \times \mathcal{N} \ar[r]^{\Phi\hspace{15pt}} \ar[d]_{(\vartheta)_s} & (J \circ \iota_{\mathcal{U}})^{-1}(\nu) \ar[d]^{T_s^*\phi_{\vert (J \circ \iota_{\mathcal{U}})^{-1}(\nu)}} 
    \\ \R \times \mathcal{N} \ar[r]_{\hspace{-15pt}\Phi} & (J \circ \iota_{\mathcal{U}})^{-1}(\nu) }$$

    Thus, we obtain that 
    $$(J \circ \iota_{\mathcal{U}})^{-1}(\nu) / \R \cong (\R \times \mathcal{N}) / \R \cong \mathcal{N}.$$

    Taking into account the diffeomorphism $\Phi$ given in (\ref{ex:dif Phi}), and that $\iota_\nu^*(\omega_h)\in \Omega^2((J \circ \iota_{\mathcal{U}})^{-1}(\nu))$, where $\iota_\nu : (J \circ \iota_{\mathcal{U}})^{-1}(\nu) \hookrightarrow T^*Q$ is the canonical inclusion, we can consider the $2$-form $\Phi^*(\iota_\nu^*(\omega_h))\in \Omega^2(\R \times \mathcal{N})$, which is given by
    \begin{align*}
        \Phi^*(\iota_\nu^*(\omega_h)) =  & dq^\alpha\wedge dp_\alpha + \frac{1}{m}p_idp_i \wedge dt + m\Omega^2q^\alpha dq^\alpha\wedge dt \\ & + \left(vdp_1+ \frac{1}{m}p_idp_i + m\Omega^2(v^2tdt + q^\alpha dq^\alpha)\right) \wedge d\hat{t}.
    \end{align*}

    However, from (\ref{ex:iota*(vdp_1+ 1/m p_idp_i + mOmega^2(v^2tdt + q^alpha dq^alpha))=0}), the last term vanishes identically on $\mathcal{N}$. Consequently, the reduced $2$-form on $\mathcal{N}$ is 
    \begin{equation}\label{ex: omega_h^nu}
        \omega_h^\nu = dq^\alpha\wedge dp_\alpha + \frac{1}{m}p_idp_i \wedge dt + m\Omega^2q^\alpha dq^\alpha\wedge dt.
    \end{equation}

    Finally, the restriction of the Hamiltonian vector field $X_{F_h}^{\omega_Q}$ to $(J \circ \iota_{\mathcal{U}})^{-1}(\nu)$ can be seen on $\R \times \mathcal{N}$ as
    $$X_{F_h}^{\omega_Q} = \left(\frac{p_1}{mv}+1\right)\frac{\partial}{\partial t} + \frac{1}{m}\left(p_\alpha \frac{\partial}{\partial q^\alpha} - \frac{p_1}{v}\frac{\partial}{\partial \hat{t}}\right) - m \Omega^2\left( vt\left(v \frac{\partial}{\partial p} + \frac{\partial}{\partial p_1} \right) + q^\alpha \frac{\partial}{\partial p_\alpha}\right),$$
    which projects on the reduced space $\mathcal{N}$ in the reduced Hamiltonian vector field
    \begin{equation}\label{ex: widetilde X_nu}
        X_\nu = \left(\frac{p_1}{mv}+1\right)\frac{\partial}{\partial t} + \frac{1}{m}p_\alpha \frac{\partial}{\partial q^\alpha} - m \Omega^2\left( vt\left(v \frac{\partial}{\partial p} + \frac{\partial}{\partial p_1} \right) + q^\alpha \frac{\partial}{\partial p_\alpha}\right).
    \end{equation}

    In addition, we have
    \begin{equation}\label{ex: R_nu}
        R_\nu = \frac{\partial}{\partial p}.
    \end{equation}
    Thus, we have that $(\mathcal{N}, \omega_h^\nu)$ is a presymplectic manifold of corank $2$ with $\ker \omega_h^\nu = \langle R_\nu, X_\nu \rangle$.
\end{example}

\subsubsection{Presymplectic reduction of corank $1$}
On the other hand, analogously to the previous reduction, we can reduce the space $(V^*\Pi, \Omega_h)$. Taking into account the reduction of presymplectic manifolds of corank $1$ developed in \cite{GUTIERREZSAGREDO}, we can apply the following theorem. 
\begin{theorem}\label{theorem: presym reduc corank 1}\cite{GUTIERREZSAGREDO}
    Let $\phi: G \times M \to M$ be a Hamiltonian presymplectic action of a Lie group $G$ for a presymplectic structure $(\omega, R)$ of corank $1$ on $M$ with $G$-equivariant momentum map $J: M \to \g^*$. Moreover, let $\nu\in \g^*$ such that $J^{-1}(\nu)\neq \emptyset$ and suppose that $\phi$ is infinitesimally free and the space of orbits $M_\nu = J^{-1}(\nu)/G_\nu$ is a manifold such that the canonical projection $\pi_\nu : J^{-1}(\nu) \to M_\nu$ is a submersion. Then,
    \begin{itemize}
        \item[(i)] There exists a unique closed $2$-form $\omega_\nu$ on $M_\nu$ such that
        $$\pi_\nu^*(\omega_\nu) = \iota_\nu^*\omega,$$
        where $\iota_\nu : J^{-1}(\nu) \hookrightarrow M$ is the inclusion map.
        \item[(ii)] The vector field $R$ restricts to $J^{-1}(\nu)$ and its restriction is $\pi_\nu$-projectable. Its projection is a vector field $R^\nu$ such that 
        $$\ker \, \omega_\nu = \langle R^\nu\rangle.$$
        Thus, $(M_\nu,\omega_\nu)$ is a presymplectic manifold of corank $1$. 
    \end{itemize}
\end{theorem}

Now, we will apply the previous result to the presymplectic structure $(\Omega_h, R_h)$ of corank $1$ on $V^*\Pi$. The necessary ingredients for this reduction are: 
\begin{itemize}
    \item The vertical bundle $V^*\Pi$ with the presymplectic structure $\Omega_h$, given in (\ref{eq:def Omega_h}) by
    $$\Omega_h = h^*(\omega_Q) = h^*(\omega_h),$$
    with $\ker\Omega_h = \langle R_h \rangle$.
    \item The action $V^*\phi: G \times V^*\Pi \to V^*\Pi$ given by (\ref{eq: def V*phi}) deduced from a free action $\phi: G \times Q \to Q$.
    \item The $Ad^*$-equivariant momentum map $J^V:V^*\Pi \to \g^*$, given in Proposition \ref{prop:J^V momentum map} and characterized by 
    $$J = J^V \circ \mu.$$
\end{itemize}

Moreover, using Remark \ref{remark: T*phi free and proper}, we have that $V^*\phi$ is free and then it is infinitesimally free (see Remark \ref{remark: infinitesimally free}). Thus, if $\nu \in \g^*$ is such that $(J^V)^{-1}(\nu)\neq \emptyset$, then, from Proposition \ref{prop: regular value}, we have that $(J^V)^{-1}(\nu)$ is a submanifold of $V^*\Pi$. Assume that the quotient space $(J^V)^{-1}(\nu)/G_\nu$ is a smooth manifold and the canonical projection $\pi_\nu^V: (J^V)^{-1}(\nu) \to (J^V)^{-1}(\nu)/G_\nu$ is a submersion. Then, using Theorem \ref{theorem: presym reduc corank 1}, we deduce the following result. 

\begin{corollary}\label{corollary: reduction corank 1}
    Let $(V^*\Pi, \Omega_h, V^*\phi, J^V)$ be the symmetric presymplectic Hamiltonian system given in Proposition \ref{prop:J^V momentum map}. Moreover, let $\nu\in \g^*$ such that $J^{-1}(\nu) \neq \emptyset$ and suppose that $\phi$ is free and the space of orbits $(V^*\Pi)_\nu := (J^V)^{-1}(\nu)/G_\nu$ is a manifold such that the canonical projection $\pi_\nu^V: (J^V)^{-1}(\nu) \to (J^V)^{-1}(\nu)/G_\nu$ is a submersion. Then, 
    \begin{itemize}
        \item[(i)] There exists a unique presymplectic closed $2$-form $\Omega_h^\nu$ on $(V^*\Pi)_\nu$ such that
        \begin{equation}\label{eq:(pi_nu^V)*Omega_h^nu = (i_nu^V)*Omega_h}
        (\pi_\nu^V)^*\Omega_h^\nu = (\iota_\nu^V)^*\Omega_h,
        \end{equation}
        where $\iota_\nu^V: (J^V)^{-1}(\nu) \hookrightarrow V^*\Pi$ is the canonical inclusion.

        \item[(ii)] The vector field $R_h$, characterized by $\ker\,\Omega_h = \langle R_h\rangle$, restricts to $(J^V)^{-1}(\nu)$ and its restriction is $\pi_\nu^V$-projectable. Its projection is a vector field $R_h^\nu$ such that
        \begin{equation}\label{eq: ker Omega_h^nu = R_h^nu}
            \ker\, \Omega_h^\nu = \langle R_h^\nu \rangle.
        \end{equation}
    \end{itemize}
    Thus, $((V^*\Pi)_\nu,\Omega_h^\nu)$ is a presymplectic manifold of corank $1$.
\end{corollary}

The following diagram shows this reduction:
$$\xymatrix{
\left(V^*\Pi,\,\Omega_h, \, R_h\right)   & ((J^V)^{-1}(\nu),\, (\iota_\nu^V)^*\Omega_h,\,{R_h}_{\vert (J^V)^{-1}(\nu)})  \ar@{_{(}->}[l]_{\hspace{-40pt}\iota_\nu^V} \ar[r]^{\hspace{40pt}\pi_\nu^V}   & \left((V^*\Pi)_\nu,\, \Omega_h^\nu,\, R_h^\nu\right) 
}$$

\begin{remark}\label{remark: reduction V}
    Note that Cororally \ref{corollary: reduction corank 1} holds under the conditions of Remark \ref{remark: U = T*Q}. If this assumption is not satisfied, that is, if $\mathcal{U}\neq T^*Q$, then, we consider the open subset $\mathcal{V} = \mu(\mathcal{U})$ of $V^*\Pi$ and the principal $\R$-bundle $\mu: \mathcal{U} \to \mathcal{V}$. Thus, the reduced presymplectic manifold $((V^*\Pi)_\nu, \Omega_h^\nu)$ of corank $1$ is replaced by $(\mathcal{V}_\nu, \iota_{\mathcal{V}_\nu}^*(\Omega_h^\nu))$, where 
$$\mathcal{V}_\nu = (J^V \circ \iota_{\mathcal{V}})^{-1}(\nu)/ G_\nu$$
is a quotient manifold, since $(J^V \circ \iota_{\mathcal{V}})^{-1}(\nu)$ is a $G_\nu$-invariant open subset of $(J^V)^{-1}(\nu)$ and $\iota_{\mathcal{V}_\nu}: \mathcal{V}_\nu \hookrightarrow (V^*\Pi)_\nu$ is the canonical inclusion. In fact, $\mathcal{V}_\nu$ can be identified with an open subset of $(V^*\Pi)_\nu = (J^V)^{-1}(\nu)/G_\nu$ and we can consider the restriction to $\mathcal{V}_\nu$ of the vector field $R_h^\nu$.
\end{remark}

Now, we continue with Example \ref{ex: momentum map}.

\begin{example}[$N$-dimensional harmonic oscillator seen by an observer that moves with a constant velocity]\label{ex: mechan presympl reduc}
 Another time, continuing with Example \ref{ex:harmonic oscillator}, from the momentum map $J^V: V^*\Pi \to \R$, given in (\ref{ex: J^V}), we deduce that
    \begin{itemize}
        \item For any $\nu \geq \frac{mv^2}{2}$, $(J^V \circ \iota_{\mathcal{V}})^{-1}(\nu)= \emptyset$.
        \item For any $\nu < \frac{mv^2}{2}$, $(J^V \circ \iota_{\mathcal{V}})^{-1}(\nu)$ is the $2N$-dimensional submanifold given by
    $$(J^V \circ \iota_{\mathcal{V}})^{-1}(\nu)= \left\{ \begin{array}{ll} (q^i,t,p_i)\in \mathcal{V} \,\vert\, & \frac{1}{2m}\left((p_1+mv)^2 + p_\alpha^2 \right) + \frac{1}{2}m\Omega^2\left((q^1+vt)^2+(q^\alpha)^2\right) \\[6pt]  & = \frac{m v^2}{2}-\nu \end{array} \right\}$$
    \end{itemize}
  
    To describe the quotient $(J^V \circ \iota_{\mathcal{V}})^{-1}(\nu) / \R$ for any $\mu < \frac{mv^2}{2}$, we define $\mathcal{N}^V=(J^V \circ \iota_{\mathcal{V}})^{-1}(\nu)\cap (\varphi^V)^{-1}(0)$, where $\varphi^V$ is the first projection, i.e., $\varphi^V(q^i,t,p_i)= q^1$. In fact,
    $$\mathcal{N}^V= \left\{ \begin{array}{ll} & (q^\alpha,t,p_1,p_\alpha)\in \R^{2N}\setminus \{(0,0,-mv,0,p)\} \,\vert\, \\[6pt] & \frac{1}{2m}\left((p_1+mv)^2 + p_\alpha^2 \right) + \frac{1}{2}m\Omega^2\left((vt)^2+(q^\alpha)^2\right) = \frac{m v^2}{2}-\nu \end{array} \right\}$$
    
    Using the constraint that defines $\mathcal{N}^V$, we deduce that 
    \begin{equation}\label{ex:iota^V*(vdp_1+ 1/m p_idp_i + mOmega^2(v^2tdt + q^alpha dq^alpha))=0}
        (\iota^V)^*(vdp_1+ \frac{1}{m}p_idp_i + m\Omega^2(v^2tdt + q^\alpha dq^\alpha))=0,
    \end{equation}
    where $\iota^V : \mathcal{N}^V \hookrightarrow (J^V \circ \iota_{\mathcal{V}})^{-1}(\nu)$ is the inclusion map.

    From the flow of the infinitesimal generator $1_{V^*\Pi}$ of $V^*\phi$, we consider the map 
    \begin{equation}\label{ex:dif Phi^V}
        \begin{array}{rccl}
        \Phi^V \colon & \R \times \mathcal{N}^V & \longrightarrow & (J^V \circ \iota_{\mathcal{V}})^{-1}(\nu)\\
        &(\hat{t},(q^\alpha,t,p_i)) & \longmapsto & (-v\hat{t},q^\alpha, t+\hat{t},p_i)
    \end{array}
    \end{equation}
    which is a diffeomorphism, with inverse map given by
    $$
    \begin{array}{rccl}
        (\Phi^V)^{-1} \colon & (J^V \circ \iota_{\mathcal{V}})^{-1}(\nu) & \longrightarrow & \R \times \mathcal{N}^V\\
        &(q^1,q^\alpha,t,p_i) & \longmapsto & (-\frac{1}{v}q^1,q^\alpha, t+\frac{1}{v}q^1,p_i).
    \end{array}
    $$

    Now, consider the action $\vartheta^V: \R \times (\R \times \mathcal{N}^V) \to \R \times \mathcal{N}^V$ defined by $$\vartheta^V(s,(\hat{t},q^\alpha,t,p_i)) = (s+\hat{t},q^\alpha,t,p_i).$$ We have that the diffeomorphism $\Phi^V$ is equivariant with respect to the actions $\vartheta^V$ and $V^*\phi_{\vert (J^V \circ \iota_{\mathcal{V}})^{-1}(\nu)}$, i.e., the following diagram is commutative
    $$\xymatrix{ \R \times \mathcal{N}^V \ar[r]^{\hspace{-15pt}\Phi^V} \ar[d]_{(\vartheta^V)_s} & (J^V \circ \iota_{\mathcal{V}})^{-1}(\nu) \ar[d]^{V_s^*\phi_{\vert (J^V \circ \iota_{\mathcal{V}})^{-1}(\nu)}} 
    \\ \R \times \mathcal{N}^V \ar[r]_{\hspace{-15pt}\Phi^V} & (J^V \circ \iota_{\mathcal{V}})^{-1}(\nu) }$$

    As a consequence, we obtain that 
    $$(J^V \circ \iota_{\mathcal{V}})^{-1}(\nu) / \R \cong (\R \times \mathcal{N}^V) / \R \cong \mathcal{N}^V.$$
    Taking into account the diffeomorphism $\Phi^V$ given in (\ref{ex:dif Phi^V}), and that $(\iota_\nu^V)^*(\Omega_h)\in \Omega^2((J^V \circ \iota_{\mathcal{V}})^{-1}(\nu))$, where $\iota_\nu^V: (J^V \circ \iota_{\mathcal{V}})^{-1}(\nu) \hookrightarrow V^*\Pi$ is the canonical inclusion, we can consider the $2$-form $(\Phi^V)^*((\iota_\nu^V)^*(\Omega_h))\in \Omega^2(\R \times \mathcal{N}^V)$, which is given by
    \begin{align*}
        (\Phi^V)^*((\iota_\nu^V)^*(\Omega_h)) =  & dq^\alpha\wedge dp_\alpha + \frac{1}{m}p_idp_i \wedge dt + m\Omega^2q^\alpha dq^\alpha\wedge dt \\ & + \left(vdp_1+ \frac{1}{m}p_idp_i + m\Omega^2(v^2tdt + q^\alpha dq^\alpha)\right) \wedge d\hat{t}.
    \end{align*}

    However, from (\ref{ex:iota^V*(vdp_1+ 1/m p_idp_i + mOmega^2(v^2tdt + q^alpha dq^alpha))=0}), the last term vanishes identically on $\mathcal{N}^V$. Consequently, the reduced $2$-form on $\mathcal{N}^V$ is 
    \begin{equation}\label{ex: Omega_h^nu}
        \Omega_h^\nu = dq^\alpha\wedge dp_\alpha + \frac{1}{m}p_idp_i \wedge dt + m\Omega^2q^\alpha dq^\alpha\wedge dt.
    \end{equation}

    Finally, the restriction of the vector field $R_h$ to $(J^V \circ \iota_{\mathcal{V}})^{-1}(\nu)$ can be seen on $\R \times \mathcal{N}^V$ as
    $$R_h = \left(\frac{p_1}{mv}+1\right)\frac{\partial}{\partial t} + \frac{1}{m}\left(p_\alpha \frac{\partial}{\partial q^\alpha} - \frac{p_1}{v}\frac{\partial}{\partial \hat{t}}\right) - m \Omega^2\left( vt\frac{\partial}{\partial p_1} + q^\alpha \frac{\partial}{\partial p_\alpha}\right),$$
    which projects on the reduced space $\mathcal{N}^V$ in the reduced vector field
    \begin{equation}\label{ex: R_h^nu}
        R_h^\nu = \left(\frac{p_1}{mv}+1\right)\frac{\partial}{\partial t} + \frac{1}{m}p_\alpha \frac{\partial}{\partial q^\alpha} - m \Omega^2\left( vt\frac{\partial}{\partial p_1} + q^\alpha \frac{\partial}{\partial p_\alpha}\right).
    \end{equation}

    Thus, $(\mathcal{N}^V,\Omega_h^\nu)$ is a presymplectic manifold of corank $1$.

\end{example}

\subsection{Reduction of the principal $\R$-bundle $\mu$}
From the previous subsections, we have reduced the presymplectic manifolds $(T^*Q,\omega_h)$ and $(V^*\Pi, \Omega_h)$ to the presymplectic manifolds $((T^*Q)_\nu, \omega_h^\nu)$ and $((V^*\Pi)_\nu, \Omega_h^\nu)$, respectively. 
Since
\[\mu:T^*Q \longrightarrow V^*\Pi
\] is a $G$-equivariant map and a principal $\R$-bundle with principal action $\psi$ given by (\ref{eq:action psi}),
our next goal, following ideas of \cite{IGNACIO}, is to prove that $(T^*Q)_\nu$ and $(V^*\Pi)_\nu$ are the total space and the base manifold, respectively, of a reduced principal $\R$-bundle $\mu_\nu: (T^*Q)_\nu \to (V^*\Pi)_\nu$ which is defined as follows.

We know, from Proposition \ref{prop:J^V momentum map}, that $J= J^V\circ \mu$, which means that the restriction $\mu_{\vert J^{-1}(\nu)}: J^{-1}(\nu) \to (J^V)^{-1}(\nu)$ of $\mu$ to the closed submanifold $J^{-1}(\nu)$ is a surjective submersion. Thus,
\begin{equation}\label{eq: i_nu^V o mu = mu o i_nu}
    \iota_\nu^V \circ \mu_{\vert J^{-1}(\nu)} = \mu \circ \iota_\nu,
\end{equation}
where $\iota_\nu: J^{-1}(\nu) \hookrightarrow T^*Q$ and $\iota_\nu^V: (J^V)^{-1}(\nu) \hookrightarrow V^*\Pi$ are the canonical inclusions. Moreover, we have the actions $T^*\phi: G_\nu \times J^{-1}(\nu) \to J^{-1}(\nu)$ and $V^*\phi: G_\nu \times (J^V)^{-1}(\nu) \to (J^V)^{-1}(\nu)$ of the isotropy group $G_\nu$ on $J^{-1}(\nu)$ and $(J^V)^{-1}(\nu)$, respectively, and from Proposition \ref{prop: mu is G-equivariant}, it follows that $\mu_{\vert J^{-1}(\nu)}: J^{-1}(\nu) \to (J^V)^{-1}(\nu) $ is equivariant with respect to the previous actions. Denote by 
\begin{equation}\label{eq:mu_nu principal R-bundle}
    \mu_\nu: (T^*Q)_\nu= J^{-1}(\nu)/G_\nu  \longrightarrow (V^*\Pi)_\nu = (J^V)^{-1}(\nu)/G_\nu
\end{equation}
the induced map on the quotient spaces, which is characterized by
\begin{equation}\label{eq:mu_nu o pi_nu = pi_nu^V o mu}
    \mu_\nu \circ \pi_\nu = \pi_\nu^V \circ \mu_{\vert J^{-1}(\nu)}.
\end{equation}

Furthermore, from Proposition \ref{prop: reduced psi}, we know that the principal action $\psi: \R \times T^*Q \to T^*Q$ induces an action $\psi_\nu : \R \times (T^*Q)_\nu \to (T^*Q)_\nu$ of $\R$ on $(T^*Q)_\nu$. In fact, using (\ref{eq:action psi}) and (\ref{eq:(psi_nu)_s o pi_nu = pi_nu o psi_s}), we know that
\begin{equation}\label{eq: action psi_nu}
    \psi_\nu(s,(\pi_\nu(\alpha_q))) = \pi_\nu(\alpha_q) + s \pi_\nu(\eta(q))
\end{equation}
for all $s\in \R$ and $\alpha_q \in J^{-1}(\nu)$. In particular, the following theorem tells us that the map $\mu_\nu$ is a principal $\R$-bundle with $\psi_\nu$ as the associated principal action.

\begin{theorem}\label{theorem: reduction mu}
    Let $\mu: T^*Q \to V^*\Pi$ be the principal $\R$-bundle with $\psi: \R \times T^*Q \to T^*Q$ the associated principal action. Then, the reduced map $\mu_\nu: (T^*Q)_\nu \to (V^*\Pi)_\nu$ of $\mu$ is a principal $\R$-bundle with principal action given by the reduced action $\psi_\nu : \R \times (T^*Q)_\nu \to (T^*Q)_\nu$ of $\psi$.
\end{theorem}

\begin{proof}
    First, we will prove that $\psi_\nu$ is a free action. Indeed, suppose that
    $$(\psi_\nu)_s(\pi_\nu(\alpha_q))= \pi_\nu(\alpha_q),$$
    for some $\alpha_q\in J^{-1}(\nu)$. Then, using (\ref{eq:(psi_nu)_s o pi_nu = pi_nu o psi_s}), we have 
    $$\pi_\nu(\psi_s(\alpha_q))= (\psi_\nu)_s(\pi_\nu(\alpha_q)) = \pi_\nu(\alpha_q).$$
    Thus, from Proposition \ref{prop: T*phi and psi commute}, there exists $g\in G_\nu$ such that
    $$\psi_s(T_g^*\phi(\alpha_q))= T_g^*\phi(\psi_s(\alpha_q)) = \alpha_q.$$
    Now, taking into account Proposition \ref{prop: mu is G-equivariant} and the fact that $\mu \circ \psi_s = \mu$, it follows that
    $$\mu(\alpha_q)= \mu(\psi_s(T_g^*\phi(\alpha_q))) = \mu(T_g^*\phi(\alpha_q)) = V_g^*\phi(\mu(\alpha_q)).$$
    Finally, since $V^*\phi$ is free, we obtain that $g=e$. Thus, since $\psi_s(T_g^*\phi(\alpha_q))= \alpha_q$ and $\psi_s$ is free, we conclude that $s=0$.

    Next, we will see that the fibers of $\mu_\nu$ are precisely the orbits of the action $\psi_\nu$, that is, 
    \begin{equation}\label{eq:mu_nu o psi_nu = mu_nu}
        (\psi_\nu)_{\pi_\nu(\alpha_q)}(\R) = (\mu_\nu)^{-1}(\mu_\nu(\pi_\nu(\alpha_q))), \quad \text{for all} \quad \alpha_q\in J^{-1}(\nu).
    \end{equation}
    Indeed, let $\alpha_q \in J^{-1}(\nu)$ and $s\in \R$, by equations (\ref{eq:mu_nu o pi_nu = pi_nu^V o mu}) and (\ref{eq:(psi_nu)_s o pi_nu = pi_nu o psi_s}), and the fact that $\mu \circ \psi_s = \mu$, we have 
    \begin{align*}
        \mu_\nu((\psi_\nu)_s(\pi_\nu(\alpha_q))) & = \mu_\nu(\pi_\nu(\psi_s(\alpha_q))) = \pi_\nu^V(\mu(\psi_s(\alpha_q))) = \pi_\nu^V(\mu(\alpha_q)) = \mu_\nu(\pi_\nu(\alpha_q)).
    \end{align*} 

    This proves that 
    $$(\psi_\nu)_{\pi_\nu(\alpha_q)}(\R) \subseteq (\mu_\nu)^{-1}(\mu_\nu(\pi_\nu(\alpha_q))).$$

    Conversely, let $\gamma_q \in J^{-1}(\nu)$ such that
    $$\mu_\nu(\pi_\nu(\gamma_q)) = \mu_\nu(\pi_\nu(\alpha_q)).$$
    Then, from (\ref{eq:mu_nu o pi_nu = pi_nu^V o mu}), we deduce that 
    $$\pi_\nu^V(\mu(\gamma_q)) = \pi_\nu^V(\mu(\alpha_q)),$$
    which implies that there exists $g\in G_\nu$ satisfying 
    $$\mu(\gamma_q) = (V_g^*\phi)(\mu(\alpha_q)) = \mu((T_g^*\phi)(\alpha_q)),$$
    where the last equality follows using Proposition \ref{prop: mu is G-equivariant}. So, there exists $s\in \R$ such that 
    $$\gamma_q = \psi_s((T_g^*\phi)(\alpha_q)).$$
    Thus, from (\ref{eq:(psi_nu)_s o pi_nu = pi_nu o psi_s}), we conclude that
    $$\pi_\nu(\gamma_q) = \pi_\nu(\psi_s((T_g^*\phi)(\alpha_q))) = (\psi_\nu)_s(\pi_\nu((T_g^*\phi)(\alpha_q))) = (\psi_\nu)_s(\pi_\nu(\alpha_q)) =(\psi_\nu)_{\pi_\nu(\alpha_q)}(s).$$
    This ends the proof of (\ref{eq:mu_nu o psi_nu = mu_nu}). Consequently, $\mu_\nu$ is a principal $\R$-bundle with principal action $\psi_\nu$.
\end{proof}

Now, we have the principal $\R$-bundle $\mu_\nu :(T^*Q)_\nu \to (V^*\Pi)_\nu$. The following result describes how $\mu_\nu$ connects the presymplectic structures $\omega_h^\nu$ and $\Omega_h^\nu$, as well as the vector fields $X_\nu$ and $R_h^\nu$, on the presymplectic manifolds $(T^*Q)_\nu$ and $(V^*\Pi)_\nu$, respectively.

\begin{proposition}\label{prop: (mu_nu)*Omega_h^nu = omega_h^nu and X_nu is mu_nu projectable}
    Let $\mu_\nu: (T^*Q)_\nu \to (V^*\Pi)_\nu$ be the reduced principal $\R$-bundle with principal action $\psi_\nu$. If $\omega_h^\nu$ and $\Omega_h^\nu$ are the reduced presymplectic structures on $(T^*Q)_\nu$ and $(V^*\Pi)_\nu$, respectively, then, 
    \begin{itemize}
        \item[(i)] $\mu_\nu^*\Omega_h^\nu = \omega_h^\nu$.
        \item[(ii)] The vector field $X_\nu$ on $(T^*Q)_\nu$, given in (\ref{eq: X_F_h^omega_Q is pi_nu projectable over widetilde X_nu}), is $\mu_\nu$-projectable over the vector field $R_h^\nu$ on $(V^*\Pi)_\nu$, characterized by (\ref{eq: ker Omega_h^nu = R_h^nu}), i.e.,
        $$T_{\pi_\nu(\alpha_q)}\mu_\nu(X_\nu(\pi_\nu(\alpha_q))) = R_h^\nu(\mu_\nu(\pi_\nu(\alpha_q))), \quad \alpha_q \in J^{-1}(\nu).$$
    \end{itemize}
\end{proposition}

\begin{proof}
    \begin{itemize}
        \item[\textit{(i)}] We know that the $2$-form $\omega_h^\nu$ is completely determined by the condition (\ref{eq:(pi_nu)*omega_h^nu = (i_nu)*omega_h}), i.e., 
        $$\pi_\nu^*(\omega_h^\nu) = \iota_\nu^*(\omega_h).$$
        Therefore, to prove that $\mu_\nu^*\Omega_h^\nu = \omega_h^\nu$, it is enough to show that
        $$\pi_\nu^*(\mu_\nu^*\Omega_h^\nu)= \iota_\nu^*(\omega_h).$$
        Now, using (\ref{eq:(pi_nu^V)*Omega_h^nu = (i_nu^V)*Omega_h}), (\ref{eq:mu_nu o pi_nu = pi_nu^V o mu}), (\ref{eq: i_nu^V o mu = mu o i_nu}) and Proposition \ref{prop:mu*(Omega_h)=omega_h Omega_h = h*(omega_h)}, it follows that 
        \begin{align*}
            \pi_\nu^*(\mu_\nu^*\Omega_h^\nu) & = (\mu_\nu \circ \pi_\nu)^*\Omega_h^\nu= (\pi_\nu^V \circ \mu)^*\Omega_h^\nu = \mu^*((\pi_\nu^V)^*\Omega_h^\nu) = \mu^*((\iota_\nu^V)^*\Omega_h) \\ & = (\iota_\nu^V \circ \mu)^*\Omega_h= (\mu \circ \iota_\nu)^*\Omega_h= \iota_\nu^*(\mu^*\Omega_h) = \iota_\nu^*(\omega_h).
        \end{align*}

        \item[\textit{(ii)}] Let $\alpha_q \in J^{-1}(\nu)$. Taking into account (\ref{eq:mu_nu o pi_nu = pi_nu^V o mu}), item \textit{(ii)} in Corollary \ref{corollary: reduction corank 1} and Propositions \ref{prop:X_F_h^omega_Q is mu projectable} and \ref{prop: reduced X_F_h^omega_Q}, we have 
        \begin{align*}
            R_h^\nu(\mu_\nu(\pi_\nu(\alpha_q))) & = R_h^\nu(\pi_\nu^V(\mu(\alpha_q))) = T_{\mu(\alpha_q)}\pi_\nu^V(R_h(\mu(\alpha_q)))
            \\ & = T_{\mu(\alpha_q)}\pi_\nu^V(T_{\alpha_q}\mu(X_{F_h}^{\omega_Q}(\alpha_q)))= T_{\alpha_q}(\pi_\nu^V \circ \mu)(X_{F_h}^{\omega_Q}(\alpha_q)) 
            \\ & = T_{\alpha_q}(\mu_\nu \circ \pi_\nu)(X_{F_h}^{\omega_Q}(\alpha_q))= T_{\pi_\nu(\alpha_q)}\mu_\nu(T_{\alpha_q}\pi_\nu(X_{F_h}^{\omega_Q}(\alpha_q)))
            \\ & = T_{\pi_\nu(\alpha_q)}\mu_\nu(X_\nu(\pi_\nu(\alpha_q))).
        \end{align*}
    \end{itemize}
\end{proof}

Furthermore, the following result shows that the section $h: V^*\Pi \to T^*Q$ of $\mu$ induces a section $h_\nu: (V^*\Pi)_\nu \to (T^*Q)_\nu$ of $\mu_\nu$.

\begin{proposition}\label{prop: section h_nu}
    The section $h: V^*\Pi \to T^*Q$ restricts to a $G_\nu$-equivariant map $h_{\vert (J^V)^{-1}(\nu)}:(J^V)^{-1}(\nu) \to J^{-1}(\nu)$. Then, 
    \begin{itemize}
        \item[(i)] The restriction $h_{|(J^V)^{-1}(\nu)}$ induces a section 
    $$h_\nu: (V^*\Pi)_\nu = (J^V)^{-1}(\nu)/G_\nu \longrightarrow (T^*Q)_\nu = J^{-1}(\nu)/G_\nu$$ characterized by 
    \begin{equation}\label{eq:h_nu o pi_nu^V = pi_nu o h}
        h_\nu \circ \pi_\nu^V = \pi_\nu \circ h.
    \end{equation}

        \item[(ii)] $h_\nu^*\omega_h^\nu = \Omega_h^\nu$.
        \item[(iii)] The vector field $R_h^\nu$ on $(V^*\Pi)_\nu$, given in (\ref{eq: ker Omega_h^nu = R_h^nu}), is $h_\nu$-projectable over the vector field $X_\nu$ on $(T^*Q)_\nu$, characterized by (\ref{eq: X_F_h^omega_Q is pi_nu projectable over widetilde X_nu}), i.e.,
        $$T_{\pi_\nu^V(\beta_q)}h_\nu(R_h^\nu(\pi_\nu^V(\beta_q))) = X_\nu(h_\nu(\pi_\nu^V(\beta_q))), \quad \text{for all} \quad \beta_q \in (J^V)^{-1}(\nu).$$
    \end{itemize}

\end{proposition}

\begin{proof}
    
First, we prove that the section $h: V^*\Pi \to T^*Q$ restricts to a map $h:(J^V)^{-1}(\nu) \to J^{-1}(\nu)$, in such a way that
\begin{equation}\label{eq: i_nu o h = h o i_nu^V}
    \iota_\nu \circ h = h \circ \iota_\nu^V.
\end{equation}
In fact, if $\beta \in (J^V)^{-1}(\nu)$ and $h(\beta_q) = \alpha_q$ then, as $h:V^*\Pi \to T^*Q$ is a section of $\mu:T^*Q\to V^*\Pi$, we have that $\mu(\alpha_q) = \beta_q$. Thus, using (\ref{eq: J=J^V o mu}), we deduce that 
$$\nu = J^V(\beta_q) = J^V(\mu(\alpha_q)) = J(\alpha_q),$$
and $\alpha_q \in J^{-1}(\nu)$.

In addition, the $G_\nu$-equivariance of the restriction $h_{\vert (J^V)^{-1}(\nu)}:(J^V)^{-1}(\nu) \to J^{-1}(\nu)$ follows from the equivariance of $h$. Then, the map $h_{|(J^V)^{-1}(\nu)}$ induces a map on the quotient spaces
$$h_\nu: (V^*\Pi)_\nu = (J^V)^{-1}(\nu)/G_\nu \longrightarrow (T^*Q)_\nu = J^{-1}(\nu)/G_\nu$$ characterized by 
$$h_\nu \circ \pi_\nu^V = \pi_\nu \circ h.$$

Now, we will prove \textit{(ii)}. We know that the presymplectic structure $\Omega_h^\nu$ is completely characterized by condition (\ref{eq:(pi_nu^V)*Omega_h^nu = (i_nu^V)*Omega_h}), i.e.,
$$(\pi_\nu^V)^*\Omega_h^\nu = (\iota_\nu^V)^*\Omega_h.$$ Therefore, to prove \textit{(ii)}, it suffices to show that $$(\pi_\nu^V)^*(h_\nu^*\omega_h^\nu) = (\iota_\nu^V)^*\Omega_h.$$ But, using (\ref{eq:(pi_nu)*omega_h^nu = (i_nu)*omega_h}), (\ref{eq:h_nu o pi_nu^V = pi_nu o h}), (\ref{eq: i_nu o h = h o i_nu^V}) and the first item in Proposition \ref{prop:mu*(Omega_h)=omega_h Omega_h = h*(omega_h)}, we have
\begin{align*}
    (\pi_\nu^V)^*(h_\nu^*\omega_h^\nu) & = (h_\nu \circ \pi_\nu^V)^*(\omega_h^\nu)  = (\pi_\nu \circ h)^*(\omega_h^\nu) = h^*(\pi_\nu^*(\omega_h^\nu)) = h^*(\iota_\nu^*(\omega_h))
    \\ & = (\iota_\nu \circ h)^*(\omega_h) = (h \circ \iota_\nu^V)^*(\omega_h)= (\iota_\nu^V)^*(h^*\omega_h) = (\iota_\nu^V)^*\Omega_h.
\end{align*}

Finally, we will prove \textit{(iii)}. If $\beta_q \in (J^V)^{-1}(\nu)$ then, using (\ref{eq:h_nu o pi_nu^V = pi_nu o h}), item \textit{(ii)} in Corollary \ref{corollary: reduction corank 1} and Propositions \ref{prop: R_h h-projectable X_F_h^omega_Q} and \ref{prop: reduced X_F_h^omega_Q}, we conclude
\begin{align*}
    T_{\pi_\nu^V(\beta_q)}h_\nu(R_h^\nu(\pi_\nu^V(\beta_q))) & = T_{\pi_\nu^V(\beta_q)}h_\nu(T_{\beta_q}\pi_\nu^V(R_h(\beta_q))) = T_{\beta_q}(h_\nu \circ \pi_\nu^V)(R_h(\beta_q))
    \\ & = T_{\beta_q}(\pi_\nu \circ h)(R_h(\beta_q)) = T_{h(\beta_q)}\pi_\nu(T_{\beta_q}h(R_h(\beta_q))) 
    \\ & = T_{h(\beta_q)}\pi_\nu(X_{F_h}^{\omega_Q}(h(\beta_q))) = X_\nu(\pi_\nu(h(\beta_q))) = X_\nu(h_\nu(\pi_\nu^V(\beta_q))).
\end{align*}

\end{proof}

Analogously to Proposition \ref{prop:function F_h}, we obtain the following result. 

\begin{proposition}
    Let $h_\nu: (V^*\Pi)_\nu \to (T^*Q)_\nu$ the section of $\mu_\nu$ induced by the Hamiltonian section $h$ of $\mu$. Then, there exists a unique function $F_h^\nu:(T^*Q)_\nu\to\R$ characterized by
    \begin{equation}\label{eq: F_h^nu o pi_nu = F_h o i_nu}
        F_h^\nu \circ \pi_\nu = F_h \circ \iota_\nu.  
    \end{equation}
    Moreover, $F_h^\nu$ is homogeneous of degree $1$ with respect to $\psi_\nu$, i.e., $R_\nu(F_h^\nu)=1$, where $R_\nu$ denotes the infinitesimal generator of the action $\psi_\nu$.
\end{proposition}

\begin{proof}
    From Proposition \ref{prop:F_h G-invariant}, the $G$-invariant function $F_h: T^*Q \to \R$ induces a reduced function $F_h^\nu: (T^*Q)_\nu \to \R$ characterized by $F_h^\nu \circ \pi_\nu = F_h \circ \iota_\nu$.

    Finally, let $\pi_\nu(\alpha_q)\in (T^*Q)_\nu$ with $\alpha_q\in J^{-1}(\nu)$, taking into account (\ref{eq: R is pi_nu projectable over R_nu}), (\ref{eq: F_h^nu o pi_nu = F_h o i_nu}) and Proposition \ref{prop:function F_h}, we conclude that 
    $$R_\nu(F_h^\nu)(\pi_\nu(\alpha_q)) = T_{\alpha_q}\pi_\nu(R(\alpha_q))(F_h^\nu) = R(F_h^\nu \circ \pi_\nu)(\alpha_q) = R(F_h)(\alpha_q) = 1.$$
\end{proof}

The following diagram summarizes all the reduction processes. 
$$\xymatrix{
\left(T^*Q,\,\omega_h,\, X_{F_h}^{\omega_Q}\right) \ar[dd]^{\mu} \ar@(u,ur)[]^{(\psi, R)} & (J^{-1}(\nu),\, \iota_\nu^*\omega_h,\, {X_{F_h}^{\omega_Q}}_{{\vert J^{-1}(\nu)}})  \ar[r]^{\hspace{20pt}\pi_\nu} \ar@{_{(}->}[l]_{\hspace{-20pt}\iota_\nu} \ar[dd]^{\mu_{\vert J^{-1}(\nu)}} \ar@(u,ur)[]^{(\psi, R_{{\vert J^{-1}(\nu)}})} & \left((T^*Q)_\nu,\, \omega_h^\nu,\, X_\nu\right)  \ar[dd]^{\mu_\nu} \ar@(u,ur)[]^{(\psi_\nu, R_\nu)}
\\ \\ \left(V^*\Pi,\,\Omega_h, \, R_h\right) \ar@/^2pc/[uu]^{h} & ((J^V)^{-1}(\nu),\, (\iota_\nu^V)^*\Omega_h, \, {R_h}_{\vert (J^V)^{-1}(\nu)}) \ar@/^2pc/[uu]^{h_{\vert(J^V)^{-1}(\nu)}} \ar@{_{(}->}[l]_{\hspace{-40pt}\iota_\nu^V} \ar[r]^{\hspace{40pt}\pi_\nu^V} & \left((V^*\Pi)_\nu,\, \Omega_h^\nu,\, R_h^\nu\right) \ar@/^2pc/[uu]^{h_\nu}
}$$

\begin{remark}
    Let $\mathcal{U}$ be the open $\mathbb{R}$-invariant (and $G$-invariant) subset of $T^*Q$ satisfying (\ref{eq: transversal condition T*Q}) and $\mathcal{V} = \mu(\mathcal{U})$ the open ($G$-invariant) subset of $V^*\Pi$ such that $\mu_{\vert \mathcal{U}}: \mathcal{U} \to \mathcal{V}$ is a principal $\R$-bundle with the associated principal action given by the restriction $\psi_{\vert \R \times \mathcal{U}}: \R \times \mathcal{U} \to \mathcal{U}$ of $\psi$ to $\R \times \mathcal{U}$. 
    
    If the assumption of Remark \ref{remark: U = T*Q} is not satisfied, that is, if $\mathcal{U}\neq T^*Q$, then the reduced principal $\R$-bundle $\mu_\nu$ given in Theorem \ref{theorem: reduction mu} and its section $h_\nu$ given in Proposition \ref{prop: section h_nu} are replaced by $$\mu_\nu: \mathcal{U}_\nu \to \mathcal{V}_\nu \qquad \text{and} \qquad  h_\nu : \mathcal{V}_\nu \to \mathcal{U}_\nu,$$ where $\mathcal{U}_\nu \subset (T^*Q)_\nu$ and $\mathcal{V}_\nu \subset (V^*\Pi)_\nu$ are the open subsets given in Remarks \ref{remark: reduction U} and \ref{remark: reduction V}, respectively. In addition, the reduced principal action associated with $\mu_\nu$ is the restriction ${\psi_\nu}_{\vert \R \times \mathcal{U}_\nu}: \R \times \mathcal{U}_\nu \to \mathcal{U}_\nu$ to $\R \times \mathcal{U}_\nu$ of the principal action $\psi_\nu$ given by (\ref{eq: action psi_nu}).
    
\end{remark}

Finally, we continue the discussion of Example \ref{ex: presympl reduc} (see also Example \ref{ex: mechan presympl reduc}). 
 
\begin{example}[$N$-dimensional harmonic oscillator seen by an observer that moves with a constant velocity]
    Continuing with Example \ref{ex: presympl reduc} (see also Example \ref{ex: mechan presympl reduc}), we have the reduced spaces $(\mathcal{N}, \omega_h^\nu)$ and $(\mathcal{N}^V, \Omega_h^\nu)$ of the presymplectic manifolds $(\R^{2N+2}, \omega_h)$ and $(\R^{2N+1},\Omega_h)$, respectively. Thus, from the expressions of $\mu$ and $h$, given in (\ref{ex: mu}) and (\ref{ex: h}), respectively, the reduced principal $\R$-bundle $\mu_\nu : \mathcal{N} \to \mathcal{N}^V$ and its reduced section $h_\nu:\mathcal{N}^V \to \mathcal{N}$ are given by
    $$\mu_\nu(q^\alpha,t,p_i,p)=(q^\alpha,t,p_i),$$
    and 
    $$h_\nu(q^\alpha,t,p_i) = \left(q^\alpha,t,p_i, -\frac{1}{2m}p_i^2 - \frac{1}{2}m\Omega^2\left( (vt)^2 + (q^\alpha)^2\right)\right).$$

\end{example}

\section{Example: Time-dependent Elroy's Beanie}\label{sec: example elroy}

Consider a time-dependent version of Elroy’s Beanie (for the classical Elroy’s Beanie, see, for instance, \cite{MR2977681}), consisting of two rigid bodies moving in the plane,  connected by an ideal frictionless hinge (with no torsional spring acting at the joint) and subject to an external interaction whose preferred direction rotates uniformly in the plane with constant angular velocity $\w$. This rotating interaction exerts a torque on one of the bodies, tending to align it with the instantaneous direction of the external field. As a consequence, although the kinetic energy of the system remains time-independent, the potential energy depends explicitly on time through the relative angle between the body and the rotating external direction.

Therefore, the configuration space for the system is $Q = SE(2) \times S^1 \times \R$, where $SE(2)$ is the special Euclidean Lie group, $S^1$ denotes the circle and $\R$ determines the time. Let $(\theta,x,y)$ be local coordinates on $SE(2)$ such that $(x,y)$ is the position of the common hinge point and $\theta$ is the rotation angle of the first rigid with respect to the fixed axis of the inertial reference frame. The angle $\varphi$ is  the local coordinate on $S^1$, which determines the rotation angle of the second rigid with respect to the fixed axis of the inertial reference frame, and $t$ the coordinate on $\R$ (see Figure \ref{fig:elroy}).

\begin{figure}[ht]
    \centering
    \includegraphics[scale=0.17]{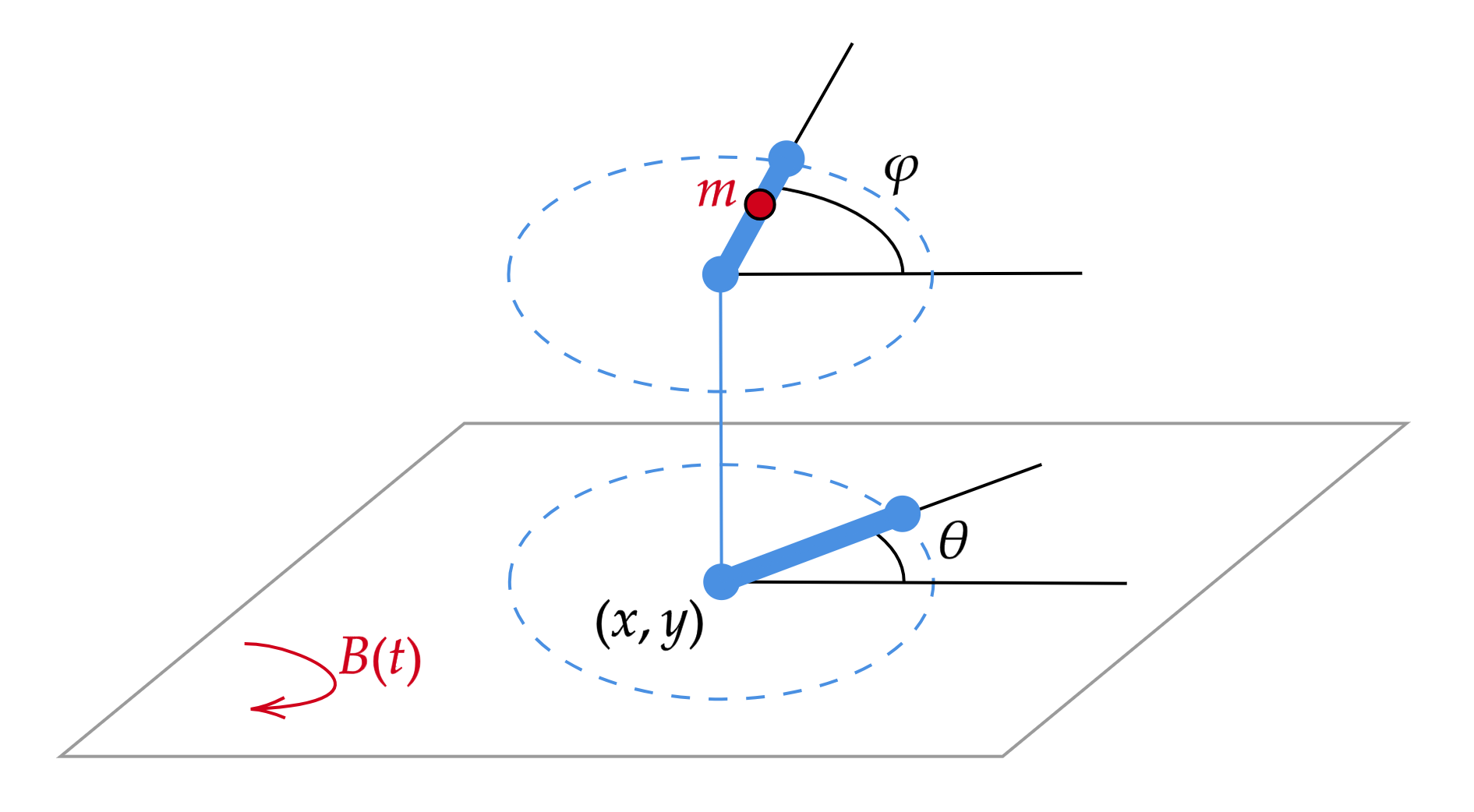}
    \caption{Time-dependent Elroy's Beanie}\label{fig:elroy}
\end{figure}

\begin{remark}
For the purposes of this paper, we adopt a different interpretation for the coordinate $\varphi$ from that used in \cite{MR2977681}. More precisely, in \cite{MR2977681}, $\varphi$ denotes the relative rotation of the second body with respect to the first.
\end{remark}

\begin{remark}
As a particular example of this kind of systems, suppose that a magnetic dipole ${\bf m}$ is rigidly attached to one of the bodies and that the system is subjected to a uniform external magnetic field ${\bf B}$ rotating in the plane with constant angular velocity $\w$. In this case, 
$${\bf B}(t)=B_0(\cos (t\w),-\sin(t\w))$$
and $${\bf m}(\varphi)=m_d(\cos\varphi,\sin\varphi)$$
where $B_0, m_d>0$ are the constant magnitudes of the magnetic and dipolar fields, respectively (see Figure \ref{fig:elroy}).

The rotational effect is described by
$$
{\boldsymbol{\tau}}=\mathbf{m}\times\mathbf{B},
$$
where $\boldsymbol{\tau}$ is the torque (or moment of force) exerted by the magnetic field on the dipole.

Since the dipole is rigidly attached to the body, this torque also acts on the body. Consequently, the interaction energy depends on the relative angle between the dipole and the rotating magnetic field, giving rise to an explicitly time-dependent potential
$$V(\varphi +t\w)={\bf m}(\varphi)\cdot {\bf B}(t)=m_dB_0 \cos(\varphi +t\w).$$
\end{remark}

Now, we return to the general case. For the system, the fibration $\Pi: SE(2) \times S^1 \times \R \to \R$ is the projection on $\R$. 

If we consider the local coordinates $(\theta,x,y,\varphi)$ on $SE(2)\times S^1$ and the corresponding coordinates $(\theta, x,y, \varphi, p_{\theta}, p_x,p_y, p_{\varphi})$  of $T^*(SE(2) \times S^1)$, then we have that the time-dependent Hamiltonian function $H: T^*(SE(2) \times S^1) \times \R \to \R$ is given by
$$H(\theta,x,y,\varphi,t, p_\theta, p_x,p_y, p_\varphi) = \frac{1}{2m}(p_x^2+ p_y^2) + \frac{1}{2I_1}p_\theta^2 + \frac{1}{2I_2}p_\varphi^2 + V(\varphi + t\w),$$
where $I_1$ and $I_2$ denote the moments of inertia and $V: S^1 \to \R$ is a potential function.

Moreover, using (\ref{eq: local mu and psi}) and (\ref{eq: local h}), we know that the principal $\R$-bundle $\mu: T^*(SE(2) \times S^1 \times \R) \to T^*(SE(2) \times S^1) \times \R$ and its section $h: T^*(SE(2) \times S^1) \times \R \to T^*(SE(2) \times S^1 \times \R)$ are given by
\begin{equation}\label{ex2: mu}
    \mu(\theta,x,y,\varphi,t, p_\theta, p_x,p_y, p_\varphi, p_t) = (\theta,x,y,\varphi,t, p_\theta, p_x,p_y, p_\varphi),
\end{equation}
and 
\begin{equation}\label{ex2: h}
    \begin{split}
        h(\theta,x,y,\varphi,t, p_\theta, p_x,p_y, p_\varphi) = & ( \theta,x,y,\varphi,t, p_\theta, p_x,p_y, p_\varphi, \\ & -\tfrac{1}{2m}(p_x^2+ p_y^2) - \tfrac{1}{2I_1}p_\theta^2 - \tfrac{1}{2I_2}p_\varphi^2 - V(\varphi + t\w)).
    \end{split}
\end{equation}
Here $(t,p_t)$ are the corresponding coordinates on $T^*\R$. Then, from (\ref{eq: local F_h}), we have that the extended Hamiltonian function $F_h: T^*(SE(2) \times S^1 \times \R) \to \R$ is 
$$F_h(\theta,x,y,\varphi,t, p_\theta, p_x,p_y, p_\varphi, p_t) = p_t + \frac{1}{2m}(p_x^2+ p_y^2) + \frac{1}{2I_1}p_\theta^2 + \frac{1}{2I_2}p_\varphi^2 + V(\varphi + t\w).$$

\subsection{Presymplectic Hamiltonian Systems}\label{subsection: ex Presymplectic Hamiltonian Systems}

In what follows, we will describe the structures of the Hamiltonian presymplectic system on the extended phase space of momenta $T^*(SE(2) \times S^1 \times \R)$ and on the phase space of momenta $T^*(SE(2) \times S^1) \times \R$.

\begin{itemize}
    \item \textbf{Presymplectic structures on $T^*(SE(2) \times S^1 \times \R)$ and $T^*(SE(2) \times S^1) \times \R$.}

    According to (\ref{eq: coordinates omega_h and coordinates Omega_h}), the presymplectic structure $\Omega_h$ on $T^*(SE(2) \times S^1) \times \R$ is given by
\begin{equation}\label{ex2: Omega_h}
    \begin{split}
        \Omega_h = & d\theta\wedge dp_\theta + dx\wedge dp_x +dy\wedge dp_y + d\varphi\wedge dp_\varphi + \frac{dV}{d\varphi}(\varphi+t\w)d\varphi \wedge dt \\ & + \frac{1}{I_1}p_\theta dp_\theta \wedge dt + \frac{1}{m}p_x dp_x \wedge dt + \frac{1}{m}p_y dp_y \wedge dt + \frac{1}{I_2}p_\varphi dp_\varphi \wedge dt
    \end{split}
\end{equation}
and 
\begin{equation}\label{ex2: omega_h}
    \omega_h = \mu^*\Omega_h,
\end{equation}
has the same expression on $T^*(SE(2) \times S^1 \times \R)$.

\item \textbf{Dynamics on $T^*(SE(2) \times S^1 \times \R)$ and $T^*(SE(2) \times S^1) \times \R$.}

Following (\ref{eq: ker Omega_h = R_h}) and Proposition \ref{prop:ker omega_h}, we have $\ker \omega_h = \langle R, X_{F_h}^{\omega_Q}\rangle$ and $\ker \Omega_h = \langle R_h \rangle$. In this example, from (\ref{eq:coordinates X_F_h^omega_Q and R}) and (\ref{eq: R_h coordinates}), we obtain
\begin{equation}\label{ex2: X_F_H omegaQ y R}
    \begin{split}
        X_{F_h}^{\omega_Q} = & \frac{\partial}{\partial t} -\w \frac{dV}{d\varphi}(\varphi + t\w) \frac{\partial}{\partial p_t} + \frac{1}{I_1}p_\theta \frac{\partial}{\partial \theta} + \frac{1}{m}p_x \frac{\partial}{\partial x} +\frac{1}{m}p_y \frac{\partial}{\partial y} \\ & + \frac{1}{I_2}p_\varphi \frac{\partial}{\partial \varphi} - \frac{dV}{d\varphi}(\varphi + t\w) \frac{\partial}{\partial p_\varphi},
    \\ R =& \frac{\partial}{\partial p_t}
    \end{split}
\end{equation}
and 
\begin{equation}\label{ex2: R_h}
    \begin{split}
        R_h = & \frac{\partial}{\partial t} + \frac{1}{I_1}p_\theta \frac{\partial}{\partial \theta} + \frac{1}{m}p_x \frac{\partial}{\partial x} +\frac{1}{m}p_y \frac{\partial}{\partial y} + \frac{1}{I_2}p_\varphi \frac{\partial}{\partial \varphi} - \frac{dV}{d\varphi}(\varphi + t\w) \frac{\partial}{\partial p_\varphi}.
    \end{split}
\end{equation}

\item \textbf{Lie group actions on $T^*(SE(2) \times S^1 \times \R)$ and $T^*(SE(2) \times S^1) \times \R$.}

In what follows, we consider the Lie group $SE(2) \times \R$ and the free action $\phi: (SE(2) \times \R)\times (SE(2) \times S^1 \times \R) \to SE(2) \times S^1 \times \R$ of this Lie group on $SE(2) \times S^1 \times \R$ given by
\begin{equation}\label{ex2: phi}
    \begin{split}
        \phi(&(\theta_0, x_0, y_0, t_0),(\theta,x,y,\varphi,t)) = \\ = &( \theta_0 + \theta,\, x_0 + x \cos \theta_0 - y \sin \theta_0,\, y_0 + x \sin \theta_0 + y \cos \theta_0,\, \varphi - t_0\w,\, t+t_0).
    \end{split}
\end{equation}

It is clear that the $1$-form $\eta = \Pi^*(dt)$ is $(SE(2) \times \R)$-invariant and the cotangent lift $T^*\phi: (SE(2) \times \R)\times T^*(SE(2) \times S^1 \times \R) \to T^*(SE(2) \times S^1 \times \R)$ is 
\begin{equation}\label{ex2: T*phi}
    \begin{split}
        T^*\phi& ((\theta_0, x_0, y_0, t_0),(\theta,x,y,\varphi,t, p_\theta, p_x, p_y, p_\varphi,p_t)) = \\ = ( & \theta_0 + \theta,\, x_0 + x \cos \theta_0 - y \sin \theta_0,\, y_0 + x \sin \theta_0 + y \cos \theta_0,\, \varphi - t_0\w,\, t+t_0, \\ & p_\theta,\, p_x \cos\theta_0 - p_y\sin\theta_0,\, p_x\sin\theta_0 + p_y \cos\theta_0,\, p_\varphi,\, p_t).
    \end{split}
\end{equation}

The induced action $V^*\phi: (SE(2) \times \R)\times T^*(SE(2) \times S^1) \times \R \to T^*(SE(2) \times S^1) \times \R$ is described by 
\begin{equation}\label{ex2: V*phi}
    \begin{split}
        V^*\phi& ((\theta_0, x_0, y_0, t_0),(\theta,x,y,\varphi,t, p_\theta, p_x, p_y, p_\varphi)) = \\ = ( & \theta_0 + \theta,\, x_0 + x \cos \theta_0 - y \sin \theta_0,\, y_0 + x \sin \theta_0 + y \cos \theta_0,\, \varphi - t_0\w,\, t+t_0, \\ & p_\theta,\, p_x \cos\theta_0 - p_y\sin\theta_0,\, p_x\sin\theta_0 + p_y \cos\theta_0,\, p_\varphi).
    \end{split}
\end{equation}

From Remark \ref{remark: T*phi free and proper}, we have that the actions $T^*\phi$ and $V^*\phi$ are free too. 

\item \textbf{Momentum maps.}

Let $\mathfrak{se}(2) \times \R$ be the Lie algebra of $SE(2)\times \R$. The fundamental vector fields of $\xi = (\xi_1, \xi_2, \xi_3, \xi_4)\in \mathfrak{se}(2) \times \R$ associated with the actions $T^*\phi$ and $V^*\phi$, respectively, are
\begin{equation}\label{ex2: xi T*phi}
    \begin{split}
        \xi_{T^*(SE(2)\times S^1\times \R)} = & \xi_1\frac{\partial}{\partial \theta} + (\xi_2-\xi_1y)\frac{\partial}{\partial x} + (\xi_3+\xi_1x)\frac{\partial}{\partial y} \\ &-\xi_1 p_y \frac{\partial}{\partial p_x} +\xi_1 p_x \frac{\partial}{\partial p_y} - \xi_4 \w \frac{\partial}{\partial \varphi} + \xi_4 \frac{\partial}{\partial t}
    \end{split}
\end{equation}
and 
\begin{equation}\label{ex2: xi V*phi}
    \begin{split}
        \xi_{T^*(SE(2)\times S^1)\times \R} = & \xi_1\frac{\partial}{\partial \theta} + (\xi_2-\xi_1y)\frac{\partial}{\partial x} + (\xi_3+\xi_1x)\frac{\partial}{\partial y} \\ &-\xi_1 p_y \frac{\partial}{\partial p_x} +\xi_1 p_x \frac{\partial}{\partial p_y} - \xi_4 \w \frac{\partial}{\partial \varphi} + \xi_4 \frac{\partial}{\partial t}.
    \end{split}
\end{equation}

As a consequence, according to (\ref{ex2: T*phi}), (\ref{ex2: V*phi}), (\ref{ex2: xi T*phi}) and (\ref{ex2: xi V*phi}), straightforward computations show that the $Ad^*$-equivariant momentum maps on the presymplectic manifolds $(T^*(SE(2)\times S^1\times \R),\omega_h)$ and $(T^*(SE(2)\times S^1)\times \R, \Omega_h)$, respectively, are
\begin{equation}\label{ex2: J}
    \begin{split}
        J(\theta,x,y,&\varphi,t, p_\theta, p_x, p_y, p_\varphi,p_t) = (p_\theta + xp_y -yp_x,\, p_x,\, p_y, \\ & -\tfrac{1}{2m}(p_x^2+ p_y^2) - \tfrac{1}{2I_1}p_\theta^2 - \tfrac{1}{2I_2}p_\varphi^2 - \w p_\varphi - V(\varphi + t\w))
    \end{split}
\end{equation}
and 
\begin{equation}\label{ex2: J^V}
    \begin{split}
        J^V(\theta,x,y,&\varphi,t, p_\theta, p_x, p_y, p_\varphi) = (p_\theta + xp_y -yp_x,\, p_x,\, p_y, \\ & -\tfrac{1}{2m}(p_x^2+ p_y^2) - \tfrac{1}{2I_1}p_\theta^2 - \tfrac{1}{2I_2}p_\varphi^2 - \w p_\varphi - V(\varphi + t\w)).
    \end{split}
\end{equation}

\end{itemize}

In order to guarantee the condition 
\begin{equation}\label{ex2: condition}
    \ker \omega_h(\alpha_q) \cap T_{\alpha_q}(G \cdot \alpha_q) = \{ 0 \},
\end{equation}
we consider the open subset $U$ of $S^1\times \R$ given by
$$U = \left\{(\varphi,t)\in S^1\times \R \,\left|\, \frac{dV}{dt}(\varphi + t\w) \neq 0\right. \right\}.$$
Note that $SE(2) \times U$ is $\phi$-invariant (see (\ref{ex2: phi})). Therefore, $\mathcal{U} = T^*(SE(2) \times U)$ and $\mathcal{V} = T^*(SE(2)) \times U\times \R$ are $T^*\phi$-invariant and $V^*\phi$-invariant, respectively. 

Moreover, if $\alpha_q\in T^*(SE(2) \times U) \subseteq T^*(SE(2)\times S^1 \times \R)$, then (\ref{ex2: condition}) holds. Indeed, let $\xi \in \mathfrak{se}(2) \times \R$ and $\lambda_1,\lambda_2 \in \R$ such that
$$\xi_{T^*(SE(2) \times S^1 \times \R)}(\alpha_q) = \lambda_1 R(\alpha_q) + \lambda_2 X_{F_h}^{\omega_Q}(\alpha_q).$$
Comparing the terms of $\displaystyle \frac{\partial}{\partial p_t}$ and $\displaystyle\frac{\partial}{\partial p_\varphi}$ of (\ref{ex2: X_F_H omegaQ y R}) and (\ref{ex2: xi T*phi}), we conclude that $\lambda_1=\lambda_2=0$. 

Thus, $(\mathcal{U},\, \iota_{\mathcal{U}}^*(\omega_h),\, {T^*\phi}_{\vert G \times \mathcal{U}},\, J \circ \iota_{\mathcal{U}})$ and $(\mathcal{V},\, \iota_{\mathcal{V}}^*(\Omega_h),\, {V^*\phi}_{\vert G \times \mathcal{V}},\, J^V \circ \iota_{\mathcal{V}})$ are the symmetric presymplectic Hamiltonian systems of corank $2$ and $1$, respectively, where $\iota_{\mathcal{U}}: \mathcal{U} \hookrightarrow T^*Q$ and $\iota_{\mathcal{V}}: \mathcal{V} \hookrightarrow V^*\Pi$ denote the canonical inclusions.

\begin{remark}
    We remark that $\mathcal{U}$ is an open subset of the complement of the set of relative equilibria of the vector field $X_{F_h}^{\omega_Q}$. But, for simplicity in the computations, we consider $\mathcal{U}$ instead of this last set.
\end{remark}

Finally, we can see that the four components of the momentum map $J^V$ (see (\ref{ex2: J^V})) are first integrals of the time-dependent Hamiltonian dynamics $R_h$ given by (\ref{ex2: R_h}). Moreover, the last component is a first integral which depends explicitly on time.

\subsection{Presymplectic reduction}

Let $\nu= (\nu_1, \nu_2, \nu_3, \nu_4) \in \mathfrak{se}(2)^* \times \R$, which is a regular value of $J$ and $J^V$ due to the freeness of the actions $T^*\phi$ and $V^*\phi$. Then, $(J \circ \iota_{\mathcal{U}})^{-1}(\nu)$ is the $6$-dimensional submanifold of $\mathcal{U}$ that can be identified with the regular submanifold of $SE(2) \times U\times \R^2$ defined by
$$(J \circ \iota_{\mathcal{U}})^{-1}(\nu)= \left\{  (\theta,x,y,\varphi,t, p_\varphi, p_t)\in SE(2)\times U\times \R^2 \,\,\vert\,\,  H_\nu(\theta, x,y,\varphi,t, p_\varphi)= -\nu_4\right\}$$
where $H_\nu : SE(2) \times U \times \R \to \R$ is the function given by
\begin{equation}\label{ex: H_nu}
    \begin{split}
        H_\nu(\theta, x,y,\varphi,t, p_\varphi) = & \frac{1}{2m}(\nu_2^2+ \nu_3^2) +\frac{1}{2I_1}(\nu_1+y\nu_2-x\nu_3)^2 \\ & + \frac{1}{2I_2}p_\varphi^2 + \w p_\varphi + V(\varphi + t\w).
    \end{split}
\end{equation}

Moreover, the level set $(J^V \circ \iota_{\mathcal{V}})^{-1}(\nu)$ is the $5$-dimensional submanifold given by
$$(J^V \circ \iota_{\mathcal{V}})^{-1}(\nu)= \left\{ (\theta,x,y,\varphi,t, p_\varphi)\in SE(2)\times U\times \R \,\,\vert\,\, H_\nu(\theta, x,y,\varphi,t, p_\varphi)= -\nu_4 \right\}.$$

Note that $(J \circ \iota_{\mathcal{U}})^{-1}(\nu) = H_\nu^{-1}(-\nu_4) \times \R$ and $(J^V \circ \iota_{\mathcal{V}})^{-1}(\nu) = H_\nu^{-1}(-\nu_4)$. Using the constraint that defines $(J \circ \iota_{\mathcal{U}})^{-1}(\nu)$ and $(J^V \circ \iota_{\mathcal{V}})^{-1}(\nu)$, we deduce that
\begin{equation}\label{ex2: constraint J}
    \begin{split}
        \iota_\nu^*&\left(\frac{1}{I_1}(\nu_1+y\nu_2-x\nu_3)(\nu_2dy -\nu_3dx) + \frac{dV}{d\varphi}(\varphi + t\w)d\varphi \right.\\ &\left.+\w \frac{dV}{d\varphi}(\varphi + t\w)dt + \frac{1}{I_2}p_\varphi dp_\varphi +\w dp_\varphi\right)=0, 
    \end{split}
\end{equation}
for all $(\theta,x,y,\varphi,t, p_\varphi, p_t) \in (J \circ \iota_{\mathcal{U}})^{-1}(\nu)$, and 
\begin{equation}\label{ex2: constraint J^V}
    \begin{split}
        (\iota_\nu^V)^*&\left(\frac{1}{I_1}(\nu_1+y\nu_2-x\nu_3)(\nu_2dy -\nu_3dx) + \frac{dV}{d\varphi}(\varphi + t\w)d\varphi \right.\\ &\left.+\w \frac{dV}{d\varphi}(\varphi + t\w)dt + \frac{1}{I_2}p_\varphi dp_\varphi +\w dp_\varphi\right)=0,  
    \end{split}
\end{equation}
for all $(\theta,x,y,\varphi,t, p_\varphi) \in (J^V \circ \iota_{\mathcal{V}})^{-1}(\nu)$. Here, $\iota_\nu : (J \circ \iota_{\mathcal{U}})^{-1}(\nu) \hookrightarrow T^*(SE(2) \times U)$ and $\iota_\nu^V : (J^V \circ \iota_{\mathcal{V}})^{-1}(\nu) \hookrightarrow T^*(SE(2)) \times U\times \R$ denote the canonical inclusions.

In what follows, we will describe the quotient spaces $(J \circ \iota_{\mathcal{U}})^{-1}(\nu)/(SE(2) \times \R)_\nu$ and $(J^V \circ \iota_{\mathcal{V}})^{-1}(\nu)/(SE(2) \times \R)_\nu$, where $(SE(2) \times \R)_\nu$ denotes the isotropy subgroup (with respect to the coadjoint action of $SE(2) \times \R$ on $(\mathfrak{se}(2) \times \R)^*$) of $\nu$.

\begin{remark}\label{ex2: SE(2)}
    A direct computation proves that 
    \begin{itemize}
    \item[(i)] If $\nu= (\nu_1, \nu_2, \nu_3, \nu_4) \in \mathfrak{se}(2)^* \times \R$ is such that $(\nu_2, \nu_3)\neq (0,0)$, then
        $$(SE(2)\times \R)_\nu = SE(2)_\nu\times \R = \{(0,\lambda \nu_2, \lambda \nu_3) \in SE(2) \,\mid\, \lambda\in\R \}\times \R \cong \R^2$$
    \item[(ii)] If $\nu = (\nu_1, 0,0,\nu_4) \in \mathfrak{se}(2)^* \times \R$, then $(SE(2) \times \R)_\nu = SE(2) \times \R$.
    \end{itemize}
\end{remark}

We will begin with case \textit{(i)}. The treatment of case \textit{(ii)} is analogous; therefore, at the end of this section we simply describe the corresponding reduced objects.

\subsubsection{Presymplectic reduction: case $(\nu_2, \nu_3)\neq (0,0)$}

In this case, we have:

\begin{itemize}
    \item \textbf{Reduced spaces $(J \circ \iota_{\mathcal{U}})^{-1}(\nu) /(SE(2) \times \R)_\nu$ and $(J^V \circ \iota_{\mathcal{V}})^{-1}(\nu) /(SE(2) \times \R)_\nu$.}
    
According to (\ref{ex2: T*phi}) and (\ref{ex2: V*phi}), the isotropy subgroup $(SE(2)\times \R)_\nu \cong \R^2$ acts over $SE(2)\times U\times \R^2$ and $SE(2)\times U\times \R$, respectively, as follows
$$(\lambda, t_0) \cdot (\theta,x,y,\varphi,t, p_\varphi, p_t) = ( \theta, \lambda \nu_2 + x, \lambda \nu_3 +y, \varphi - t_0\w, t+t_0, p_\varphi, p_t),$$
and
$$(\lambda, t_0) \cdot (\theta,x,y,\varphi,t, p_\varphi) = ( \theta, \lambda \nu_2 + x, \lambda \nu_3 +y, \varphi - t_0\w, t+t_0, p_\varphi).$$
 
Thus, we have the diffeomorphisms
$$(SE(2)\times U\times \R^2)/(SE(2) \times \R)_\nu \cong SE(2)/SE(2)_\nu \times U/\R \times \R^2$$
and
$$(SE(2)\times U\times \R)/(SE(2) \times \R)_\nu \cong SE(2)/SE(2)_\nu \times U/\R \times \R$$
where the action of $SE(2)_\nu$ over the group $SE(2)$ is the left action induced by the group structure and the action of $\R$ over $U$ is the translation by $(-t_0\w, t_0)$.

Moreover, the following map 
$$\begin{array}{rcl}
        \Phi\colon SE(2)/SE(2)_\nu \times U/\R \times \R^2& \longrightarrow & T^*S^1 \times \widetilde{U} \times \R^2 \\
        ([(\theta,x,y)],[(\varphi,t)], p_\varphi, p_t) & \longmapsto & \left(\theta, \displaystyle \frac{x\nu_3-y\nu_2}{\nu_2^2+\nu_3^2}, \varphi + t\w, p_\varphi,p_t\right)
    \end{array}$$
is a diffeomorphism, where $\widetilde{U}$ is the open subset of $S^1$ given by $$\widetilde{U} = \left\{\widetilde{\varphi}\in S^1 \,\left|\, dV(\widetilde{\varphi}) \neq 0\right. \right\}.$$ 
Its inverse is given by
$$\begin{array}{rcl}
        \Phi^{-1} \colon T^*S^1 \times \widetilde{U} \times \R^2 & \longrightarrow &  SE(2)/SE(2)_\nu \times U/\R \times \R^2\\
        (\widetilde{\theta}, \widetilde{p}_\theta, \widetilde{\varphi}, \widetilde{p}_\varphi, \widetilde{p}_t) & \longmapsto & ([(\widetilde{\theta}, \nu_3 \widetilde{p}_\theta, -\nu_2 \widetilde{p}_\theta)], [(\widetilde{\varphi},0)],\widetilde{p}_\varphi,\widetilde{p}_t)
\end{array}$$
So,
$$(SE(2)\times U\times \R^2)/(SE(2) \times \R)_\nu \cong T^*S^1 \times \widetilde{U} \times \R^2\subseteq T^*\mathbb{T}^2 \times \R$$

In the same way, one proves
$$(SE(2)\times U\times \R)/(SE(2) \times \R)_\nu \cong T^*S^1 \times \widetilde{U} \times \R \subseteq T^*\mathbb{T}^2$$

On the other hand, the function $H_\nu$, given in (\ref{ex: H_nu}), is $(SE(2)\times \R)_\nu$-invariant. Then, it induces a function $\widetilde{H}_\nu : (SE(2) \times U \times \R)/(SE(2)\times \R)_\nu\to \R$ which is seen under the identification $\Phi$ by the function $\widetilde{H}_\nu: T^*S^1 \times \widetilde{U} \times \R \to \R$ given by
$$\widetilde{H}_\nu(\widetilde{\theta}, \widetilde{p}_\theta, \widetilde{\varphi}, \widetilde{p}_\varphi) = \frac{1}{2m}(\nu_2^2+ \nu_3^2) +\frac{1}{2I_1}(\nu_1-(\nu_2^2+ \nu_3^2)\widetilde{p}_\theta)^2 + \frac{1}{2I_2}\widetilde{p}_\varphi^2 + \w \widetilde{p}_\varphi + V(\widetilde{\varphi}),$$
for all $(\widetilde{\theta}, \widetilde{p}_\theta, \widetilde{\varphi}, \widetilde{p}_\varphi) \in T^*S^1 \times \widetilde{U} \times \R$. Thus, the reduced spaces are
$$(J \circ \iota_{\mathcal{U}})^{-1}(\nu) /(SE(2)\times \R)_\nu \cong \widetilde{H}_\nu^{-1}(-\nu_4) \times \R \subset T^*S^1 \times \widetilde{U} \times \R^2,$$
and
$$(J^V \circ \iota_{\mathcal{V}})^{-1}(\nu) /(SE(2)\times \R)_\nu \cong \widetilde{H}_\nu^{-1}(-\nu_4) \subset T^*S^1 \times \widetilde{U} \times \R.$$

\item \textbf{Reduced presymplectic structures.}

Let $\iota_\nu : (J \circ \iota_{\mathcal{U}})^{-1}(\nu) \hookrightarrow T^*(SE(2) \times U)$ and $\iota_\nu^V : (J^V \circ \iota_{\mathcal{V}})^{-1}(\nu) \hookrightarrow T^*(SE(2)) \times U \times \R$ be the canonical inclusions. Then, from (\ref{ex2: omega_h}) and (\ref{ex2: constraint J}), it follows that
\begin{equation}
        \iota_\nu^*\omega_h = d\theta \wedge (\nu_2dy - \nu_3dx) + (d\varphi +\w dt) \wedge dp_\varphi.
\end{equation}
Analogously, from (\ref{ex2: Omega_h}) and (\ref{ex2: constraint J^V}), we obtain that
\begin{equation}
    (\iota_\nu^V)^*\Omega_h = d\theta \wedge (\nu_2dy - \nu_3dx) + (d\varphi +\w dt) \wedge dp_\varphi.
\end{equation}

On the other hand, let $\pi_\nu : (J \circ \iota_{\mathcal{U}})^{-1}(\nu) \to \widetilde{H}_\nu^{-1}(-\nu_4) \times \R$ and $\pi_\nu^V : (J^V \circ \iota_{\mathcal{V}})^{-1}(\nu) \to \widetilde{H}_\nu^{-1}(-\nu_4)$ be the quotient projections given by
$$\pi_\nu(\theta,x,y,\varphi,t, p_\varphi, p_t) = \left(\theta, \frac{x\nu_3-y\nu_2}{\nu_2^2+\nu_3^2}, \varphi + t\w, p_\varphi, p_t\right),$$
for all $(\theta,x,y,\varphi,t, p_\varphi, p_t) \in (J \circ \iota_{\mathcal{U}})^{-1}(\nu)$ and
$$\pi_\nu^V(\theta,x,y,\varphi,t, p_\varphi) = \left(\theta, \frac{x\nu_3-y\nu_2}{\nu_2^2+\nu_3^2}, \varphi + t\w, p_\varphi\right),$$
for all $(\theta,x,y,\varphi,t, p_\varphi) \in (J^V \circ \iota_{\mathcal{V}})^{-1}(\nu)$. Then, if $j_\nu: \widetilde{H}_\nu^{-1}(-\nu_4)\times \R \hookrightarrow T^*S^1 \times \widetilde{U} \times \R^2$ and $j_\nu^V: \widetilde{H}_\nu^{-1}(-\nu_4) \hookrightarrow T^*S^1 \times \widetilde{U} \times \R$ are the canonical inclusions, the reduced presymplectic $2$-forms $\omega_h^\nu \in \Omega^2(\widetilde{H}_\nu^{-1}(-\nu_4) \times \R)$ and $\Omega_h^\nu \in \Omega^2(\widetilde{H}_\nu^{-1}(-\nu_4))$ are
\begin{equation}\label{ex2: omega h nu}
    \omega_h^\nu = j_\nu^*\left(-(\nu_2^2+ \nu_3^2)d\widetilde{\theta}\wedge d\widetilde{p}_\theta + d\widetilde{\varphi} \wedge d\widetilde{p}_\varphi\right)
\end{equation}
and
\begin{equation}\label{ex2: Omega h nu}
    \Omega_h^\nu = (j_\nu^V)^*\left(-(\nu_2^2+ \nu_3^2)d\widetilde{\theta}\wedge d\widetilde{p}_\theta + d\widetilde{\varphi} \wedge d\widetilde{p}_\varphi\right).
\end{equation}

\item \textbf{Reduced dynamics.}

The restrictions to $(J \circ \iota_{\mathcal{U}})^{-1}(\nu)$ of vector fields $X_{F_h}^{\omega_Q}$ and $R$, given by (\ref{ex2: X_F_H omegaQ y R}) are
\begin{align*}
    {X_{F_h}^{\omega_Q}}_{|(J \circ \iota_{\mathcal{U}})^{-1}(\nu)}  = & \frac{\partial}{\partial t} -\w \frac{dV}{d\varphi}(\varphi + t\w) \frac{\partial}{\partial p_t} + \frac{1}{I_1}(\nu_1+y\nu_2 -x\nu_3) \frac{\partial}{\partial \theta} + \frac{1}{m}\nu_2 \frac{\partial}{\partial x} \\ &+\frac{1}{m}\nu_3 \frac{\partial}{\partial y}  + \frac{1}{I_2}p_\varphi \frac{\partial}{\partial \varphi} - \frac{dV}{d\varphi}(\varphi + t\w) \frac{\partial}{\partial p_\varphi},
    \\ R_{|(J \circ \iota_{\mathcal{U}})^{-1}(\nu)} = & \frac{\partial}{\partial p_t}.
\end{align*}
In addition, these vector fields are $\pi_\nu$-projectable, respectively, over the vector fields
\begin{align*}
    X_\nu = &\w \frac{\partial}{\partial\widetilde{\varphi}} - \w \frac{dV}{d\varphi}(\widetilde{\varphi}) \frac{\partial}{\partial\widetilde{p}_t} + \frac{1}{I_1}(\nu_1-(\nu_2^2 + \nu_3^2)\widetilde{p}_\theta) \frac{\partial}{\partial \widetilde{\theta}} + \frac{1}{I_2}\widetilde{p}_\varphi \frac{\partial}{\partial \widetilde{\varphi}} - \frac{dV}{d\varphi}(\widetilde{\varphi}) \frac{\partial}{\partial \widetilde{p}_\varphi},
    \\ R_\nu =& \frac{\partial}{\partial \widetilde{p}_t}
\end{align*}

On the other hand, the restriction to $(J^V \circ \iota_{\mathcal{V}})^{-1}(\nu)$ of the vector field $R_h$, given by (\ref{ex2: R_h}) is
\begin{align*}
    {R_h}_{|(J^V \circ \iota_{\mathcal{V}})^{-1}(\nu)} = & \frac{\partial}{\partial t} + \frac{1}{I_1}(\nu_1+y\nu_2 -x\nu_3) \frac{\partial}{\partial \theta} + \frac{1}{m}\nu_2 \frac{\partial}{\partial x} +\frac{1}{m}\nu_3 \frac{\partial}{\partial y} \\ & + \frac{1}{I_2}p_\varphi\frac{\partial}{\partial \varphi} - \frac{dV}{d\varphi}(\varphi + t\w) \frac{\partial}{\partial p_\varphi},
\end{align*}
and its $\pi_\nu^V$-projection is 
\begin{align*}
    R_h^\nu = \w \frac{\partial}{\partial\widetilde{\varphi}}  + \frac{1}{I_1}(\nu_1-(\nu_2^2 + \nu_3^2)\widetilde{p}_\theta) \frac{\partial}{\partial \widetilde{\theta}} + \frac{1}{I_2}\widetilde{p}_\varphi\frac{\partial}{\partial \widetilde{\varphi}} - \frac{dV}{d\varphi}(\widetilde{\varphi}) \frac{\partial}{\partial \widetilde{p}_\varphi}.
\end{align*}

\end{itemize}

Thus, we have obtained the reduced presymplectic manifold of corank $2$
$$\left(\widetilde{H}_\nu^{-1}(-\nu_4) \times \R, \,\, j_\nu^*\left(-(\nu_2^2+ \nu_3^2)d\widetilde{\theta}\wedge d\widetilde{p}_\theta + d\widetilde{\varphi} \wedge d\widetilde{p}_\varphi\right), \,\, X_\nu, \,\,R_\nu\right),$$
and the reduced presymplectic manifold of corank $1$
$$\left(\widetilde{H}_\nu^{-1}(-\nu_4), \,\, (j_\nu^V)^*\left(-(\nu_2^2+ \nu_3^2)d\widetilde{\theta}\wedge d\widetilde{p}_\theta + d\widetilde{\varphi} \wedge d\widetilde{p}_\varphi\right),\,\, R_h^\nu \right).$$

\subsubsection{Presymplectic reduction: case $\nu=(\nu_1,0,0,\nu_4)$}
    Following the same ideas, let us consider the case in which item \textit{(ii)} of Remark \ref{ex2: SE(2)} holds, that is, $\nu=(\nu_1,0,0,\nu_4)$. Then, the reduced presymplectic manifolds $(J \circ \iota_{\mathcal{U}})^{-1}(\nu)/(SE(2) \times \R)$ and $(J^V \circ \iota_{\mathcal{V}})^{-1}(\nu)/(SE(2) \times \R)$ of corank $2$ and $1$, respectively, are given by
    $$\left(\widehat{H}_\nu^{-1}(-\nu_4) \times \R, \,\, \omega_h^\nu =0, \,\, X_\nu, \,\,R_\nu\right) \quad \text{and} \quad \left(\widehat{H}_\nu^{-1}(-\nu_4), \,\, \Omega_h^\nu=0,\,\, R_h^\nu \right).$$
    Here, if $\widehat{U}$ is the open subset of $S^1$ defined by
    $$\widehat{U} = \left\{\widehat{\varphi}\in S^1 \,\left|\, dV(\widehat{\varphi}) \neq 0\right. \right\},$$
    then $\widehat{H}_\nu: \widehat{U} \times \R \to \R$ is given by 
    $$\widehat{H}_\nu(\widehat{\varphi}, \widehat{p}_\varphi)= \tfrac{1}{2I_1}\nu_1^2 +\tfrac{1}{2I_2}\widehat{p}_\varphi^2+ V(\widehat{\varphi}) + \w \widehat{p}_\varphi,$$
    for all $(\widehat{\varphi}, \widehat{p}_\varphi) \in \widehat{U} \times \R$. Moreover, the reduced dynamics is given by 
    \begin{align*}
        X_\nu & = \left(\w + \frac{1}{I_2}\widehat{p}_\varphi \right) \frac{\partial}{\partial \widehat{\varphi}}- \frac{dV}{d\varphi}(\widehat{\varphi})\frac{\partial}{\partial \widehat{p}_\varphi} -\w \frac{dV}{d\varphi}(\widehat{\varphi})\frac{\partial}{\partial \widehat{p}_t},
        \\ R_\nu & = \frac{\partial}{\partial \widehat{p}_t}
    \end{align*}
    and 
    $$R_h^\nu = \left(\w + \frac{1}{I_2}\widehat{p}_\varphi \right) \frac{\partial}{\partial \widehat{\varphi}}- \frac{dV}{d\varphi}(\widehat{\varphi})\frac{\partial}{\partial \widehat{p}_\varphi}.$$

\section{Conclusions and future work}\label{sec: future}
In this paper, we have developed a reduction procedure of time-dependent Hamiltonian systems, where the configuration space is the total space of a fibration over the real line $\Pi: Q \to \R$. Next, we point out our results:
\begin{itemize}
    \item The extended phase space of momenta $T^*Q$ is a principal $\mathbb{R}$-bundle over $V^*\Pi$,
    the phase space of momenta, which is the dual bundle of the vertical bundle $V\Pi$ of $\Pi$. The dual map $\mu:T^*Q \to V^*\Pi$ of the canonical inclusion $i: V\Pi \hookrightarrow TQ$ is the principal bundle projection. Given a Hamiltonian section $h: V^*\Pi\to T^*Q$ of $\mu: T^*Q\to V^*\Pi$, \textit{we have endowed $T^*Q$ and $V^*\Pi$ with presymplectic structures $\omega_h$ and $\Omega_h$ of corank 2 and 1, respectively}.
    \item Our main contribution is \textit{a reduction process for a symmetry Lie group $G$ which acts on $Q$ and does not necessarily preserve the fibration $\Pi$}, extending previous results in \cite{IGNACIO}. More precisely, we impose the weaker condition that the action preserves the $1$-form $\Pi^*(dt)$. This condition is equivalent to the existence of a multiplicative function $m:G \to \R$ on $G$ and, infinitesimally, to the existence of a Lie algebra $1$-cocycle $c_\phi: \g \to \R$, such that the failure of the action to preserve the fibration depends on the function $m$. In addition, \textit{the cocycle plays a fundamental role in the construction of a presymplectic momentum map $J$ on $T^*Q$, which induces another presymplectic momentum map $J^V$ on $V^*\Pi$}. As in the time-independent case, the momentum map $J^V$ is a first integral of the time-dependent Hamiltonian dynamics.

\item With all the previous ingredients, for a regular value $\nu \in \g^*$ of the momentum maps 
\textit{we have shown that the quotient spaces $J^{-1}(\nu)/G_\nu$ and $(J^V)^{-1}(\nu)/G_\nu$ inherit presymplectic structures $\omega_h^\nu$ and $\Omega_h^\nu$ of corank 2 and 1}, respectively. Moreover, these quotients remain related by a reduced principal $\R$-bundle. Thus, the principal $\mathbb{R}$-bundle structure relating the extended and restricted phase spaces is preserved by the reduction process.
\item Finally, the theory has been illustrated through two examples: the time-dependent $N$-dimensional harmonic oscillator and the time-dependent Elroy's Beanie.
\end{itemize}
In order to ensure that the reduced presymplectic structures have the expected coranks, we have to exclude from the discussion the relative equilibria of the Hamiltonian vector field $X_{F_h}^{\omega_Q}$ of the homogeneous function $F_h:T^*Q \to \R$ (with respect to the canonical symplectic structure $\omega_Q$ on $T^*Q$) induced by the Hamiltonian section $h$. The study of these relative equilibria and their stability, following some results in \cite{Marsden_1992}, is left for future work. Note that the image by $\mu$ of a relative equilibrium for $X_{F_h}^{\omega_Q}$ is a relative equilibrium for the time-dependent Hamiltonian dynamics. 

As another future direction, we plan to discuss a more explicit description of the reduced presymplectic spaces and the reduced dynamics on them for the particular case when the isotropy subgroup $G_\nu$, with $\nu \in \g^*$, coincides with the full symmetry Lie group $G$. This theory has been intensively discussed in the time-independent case (see, for instance, \cite{MR2337886}) and we will try to extend these results for the time-dependent case.

On the other hand, in \cite{IGNACIO}, the authors developed a reduction procedure in the framework of symplectic principal $\R$-bundles and applied it to the standard symplectic principal $\R$-bundle associated with a fibration. Therefore, it would be interesting to formulate the theory developed in the present paper replacing the presymplectic principal $\R$-bundles by symplectic principal $\R$-bundles, thereby extending the results of \cite{IGNACIO} to the case where the fibration $\Pi$ is not $G$-invariant. This is the aim of a work in progress (see \cite{BENITEZ}).

\bibliographystyle{plain}
\bibliography{sample}

\end{document}